\documentclass[preprint,authoryear, 11pt]{elsarticle}
\usepackage[ngerman,english]{babel}
\usepackage[ansinew]{inputenc}
\usepackage{a4wide}
\usepackage{calc}
\usepackage{times}
\usepackage{amssymb,amsmath,amstext,amsthm,amsfonts, ams fonts}
\usepackage{colortbl,color,graphics,graphicx,epsfig}
\definecolor{gray}{rgb}{0.8,0.8,0.8}
\graphicspath{{pics/}}
\usepackage{dsfont}
\usepackage{bbm}
\usepackage{nicefrac}
\usepackage{lineno}
\usepackage{comment}
\usepackage[pdflinkmargin=5pt,pdfstartview={FitBH -32768},plainpages=false]{hyperref}
\makeatletter
\def\Links{\tagsleft@true}\def\Rechts{\tagsleft@false}
\makeatother
\renewcommand{\P}{\mathbb{P}}
\newcommand{\E}{\mathbb{E}}
\newcommand{\N}{\mathds{N}}
\newcommand{\R}{\mathds{R}}

\renewcommand{\d}{\Delta}
\newcommand{\KLEINO}{{\scriptstyle{\mathcal{O}}}}

\DeclareMathAccent{\verywidehat}{\mathord}{largesymbols}{'144}
\newcommand{\var}{\mathbb{V}\hspace*{-0.05cm}\textnormal{a\hspace*{0.02cm}r}}
\newcommand{\Cov}{\mathbb{C}\textnormal{O\hspace*{0.02cm}V}}

\DeclareMathOperator{\IV}{\mathbf{\widehat{IV}}}

\DeclareMathOperator{\ICV}{\mathbf{\widehat{ICV}}}

\DeclareMathOperator{\LMM}{\mathbf{{LMM}}}

\renewcommand{\:}{\mathrel{\mathop{:}}}
\newdefinition{remark}{Remark}
\newdefinition{defi}{Definition}
\newtheorem{theo}{Theorem}

\newtheorem{assump}{Assumption}

\newtheorem{assum}{Assumption}

\newtheorem{assumpn}{Assumption}

\newtheorem{assumn}{Assumption}

\newtheorem{prop}{Proposition}[section]
\newtheorem{lem}[prop]{Lemma}

\allowdisplaybreaks[3]
\begin{document}
\global\long\def\I{\mathbf{1}}
\global\long\def\f#1{\lfloor#1\rfloor}
\global\long\def\c#1{\mathcal{#1}}
\global\long\def\l#1{\left.#1\right|}
\allowdisplaybreaks

\begin{frontmatter}

\title{Functional stable limit theorems for quasi-efficient spectral covolatility estimators}
\author{Randolf Altmeyer\footnote{Financial support from the Deutsche Forschungsgemeinschaft via SFB 649 `\"Okonomisches Risiko', Humboldt-Universit\"at zu Berlin, is gratefully acknowledged. R.A. gratefully acknowledges the financial support from the DFG Research Training Group 1845.\\
The authors are grateful to two referees whose constructive valuable remarks helped improving the first version of the paper considerably.}\& Markus Bibinger\footnotemark[1]}
\address{Institut f\"ur Mathematik, Humboldt-Universit\"at zu Berlin, Unter den Linden 6, 10099 Berlin, Germany}
\date{\today}
\begin{abstract}
We consider noisy non-synchronous discrete observations of a continuous semimartingale with random volatility. Functional stable central limit theorems are established under high-frequency asymptotics in three setups: one-dimensional for the spectral estimator of integrated volatility, from two-dimensional asynchronous observations for a bivariate spectral covolatility estimator and multivariate for a local method of moments. The results demonstrate that local adaptivity and smoothing noise dilution in the Fourier domain facilitate substantial efficiency gains compared to previous approaches. In particular, the derived asymptotic variances coincide with the benchmarks of semiparametric Cram\'{e}r-Rao lower bounds and the considered estimators are thus asymptotically efficient in idealized sub-experiments. Feasible central limit theorems allowing for confidence bounds are provided.
\end{abstract}

\begin{keyword}
adaptive estimation\sep asymptotic efficiency \sep local parametric estimation\sep  microstructure noise \sep integrated volatility \sep non-synchronous observations  
\\[.2cm]
\end{keyword}
\end{frontmatter}
\thispagestyle{plain}
\section{Introduction\label{sec:1}}
The estimation of integrated volatility and integrated covolatility (sometimes referred to as integrated variance and covariance) from high-frequency data is a vibrant current research topic in statistics for stochastic processes. Semimartingale log-price models are widely used in financial econometrics. Its pivotal role for portfolio optimization and risk management makes the integrated covolatility matrix a key quantity of interest in econometrics and finance. During the last decade the increasing availability of data from high-frequency trades or orders has provided the statistician with rich data sets. Yet, the consequence of this richness of data is double-edged. On the one hand, high-frequency observations are almost close to continuous-time observations which should allow for a very precise and highly efficient semiparametric estimation of integrated covolatilities. On the other hand, it has turned out that a traditional pure semimartingale model has several limitations in describing stylized facts of high-frequency data and therefore does not afford a suitable risk estimation. It is nowadays well-known that effects ascribed to market microstructure frictions interfere at high frequencies with a latent price evolution that can be appropriately described by a semimartingale. Also, non-synchronous observations in a multi-dimensional setup require a thorough handling or can cause unwanted effects.\\[.2cm]
A prominent model that accounts for market microstructure in high-frequency observations is an additive noise model, in which a continuous semimartingale is observed with i.i.d.\,observation errors. The problem of estimating the integrated volatility in the one-dimensional model with noise, as well as multi-dimensional covolatility matrix estimation from \textit{noisy and non-synchronous} observations, have attracted a lot of attention in recent years and stimulated numerous research contributions from different areas. Apart from the importance for applications, the model allures researchers from stochastic calculus and mathematical statistics foremost with its intriguing intrinsic properties and surprising new effects. 

In this work, we establish central limit theorems for estimators in a general setup, first in the one-dimensional case with noisy discrete observations of a continuous semimartingale, and then in the multi-dimensional case where observations are also non-synchronous. The volatility is allowed to be random. The vital novelty is that the obtained asymptotic variances of the spectral estimators are smaller than for previous estimation approaches and coincide with the lower Cram\'{e}r-Rao-type bounds in simplified sub-models with deterministic or independent volatility. The estimators are thus \textit{asymptotically efficient} in these sub-models in which the notion of efficiency is meaningful and explored in the literature. As Cram\'{e}r-Rao lower bounds apply only to the idealized setup, in which the estimated volatility can be treated as non-random and under Gaussian noise, we call the estimators which attain the associated variance in a more general setup \textit{quasi-efficient}. \textit{Stability} of weak convergence and \textit{feasible limit theorems} allow for confidence intervals in the model with stochastic volatility.\\[.2cm]
There exist two major objectives in the strand of literature on volatility estimation: 
\begin{enumerate}
\item
Providing methods for feasible inference in general and realistic models.
\item 
Attaining the lowest possible asymptotic variance.
\end{enumerate}
The one-dimensional parametric experiment in which the volatility $\sigma$ is a constant parameter without drift and with Gaussian i.i.d.\,noise has been well understood by a LAN (local asymptotic normality) result by \cite{gloter}. While it is commonly known that for $n$ regularly spaced discrete observations $n^{1/2}$ is the optimal convergence rate in the absence of noise and $2\sigma^4$ is the variance lower bound, \cite{gloter} showed that with noise the optimal rate declines to $n^{1/4}$ and the lower variance bound is $8\eta\sigma^3$, when $\eta^2$ is the variance of the noise. Recent years have witnessed the development and suggestion of various estimation methods in a nonparametric framework that can provide consistent rate-optimal estimators. \textit{Stable central limit theorems} have been proved. Let us mention the prominent approaches by \cite{zhang}, \cite{bn2}, \cite{JLMPV} and \cite{xiu} for the one-dimensional case and \cite{sahalia}, \cite{bn1}, \cite{bibinger} and \cite{PHY} for the general multi-dimensional setup.  A major focus has been to attain a minimum (asymptotic) variance among all proposed estimators which at the slow optimal convergence rate could result in substantial finite-sample precision gains. For instance \cite{bn2} have put emphasis on the construction of a version of their kernel estimator which asymptotically attains the bound $8\eta\sigma^3$ in the parametric sub-experiment. Nonparametric efficiency is considered by \cite{reiss} in the one-dimensional setup and recently by \cite{BHMR} in a multi-dimensional non-synchronous framework. Also for the nonparametric experiment without noise efficiency is subject of current research, see \cite{renault} and \cite{gloter2} for recent advances. \cite{jacodmykland} have recently proposed an adaptive version of their pre-average estimator which achieves an asymptotic variance of ca.\,$1.07\cdot 8\eta\int_0^t\sigma_s^3\,ds$ where $8\eta\int_0^t\sigma_s^3\,ds$ is the nonparametric lower bound. In case of random endogenous observation times or noise, the statistical properties could potentially change, see e.g.\,\cite{lizhang}.\\[.2cm]
\cite{reiss} introduced a spectral approach motivated by an equivalence of continuous-time observation nonparametric and locally parametric experiments using a local method of moments. His estimator has been extended to discrete observations and the multi-dimensional setting in \cite{bibingerreiss} and \cite{BHMR}. In contrast to all previous estimation techniques, the spectral estimator attains the Cram\'{e}r-Rao efficiency lower bound for the asymptotic variance. However, the notion of nonparametric efficiency and the construction of the estimators have been restricted to the simplified statistical experiment where a continuous martingale without drift and with time-varying but deterministic volatility is observed with additive Gaussian noise. The main contribution of this work is to investigate the spectral approach under model misspecification in the general nonparametric standard setup, i.e.\,with drift, a general random volatility process and a more general error distribution. We show that the estimators significantly improve upon existing estimation methods also in more complex models which are of central interest in finance and econometrics. 
We pursue a high-frequency asymptotic distribution theory. The main results are functional stable limit theorems with optimal convergence rates and with asymptotic variances that coincide with the lower bounds in the sub-experiments. The asymptotic analysis combines the theory of \cite{jacodkey}, applied in similar context also in \cite{fukasawa} and \cite{hy3}, with Fourier analysis and matrix algebra. This is due to the fact that the efficient spectral estimation employs smoothing in the Fourier domain and because the method of moments is based on multi-dimensional Fisher information matrix calculus. In principle, the spectral estimators are based on weighted averages in the spectral domain of the observations with optimal weights depending on the local covolatility matrix. We therefore provide also adaptive versions of the spectral estimators where in a first step the local covolatility matrices are pre-estimated from the same data. This two-stage method readily provides feasible limit theorems.\\
This article is structured into six sections. Following this introduction, Section 2 introduces the statistical model and outlines all main results in a concise overview. Section 3 revisits the elements of the spectral estimation approach and the multivariate local method of moments. In Section 4 we explain the main steps for proving the functional central limit theorems. Mathematical details are given in Section 6. In Section 5 we present a Monte Carlo study.
\section{Statistical model \& Main results\label{sec:2}}
Let us first introduce the statistical model, fix the notation and gather all assumptions for the one- and the multi-dimensional setup. 
\subsection{Theoretical setup and assumptions\label{sec:2.1}}
First, consider a one-dimensional continuous It\^{o} semimartingale
\begin{align}\label{sm}X_t=X_0+\int_0^tb_s\,ds+\int_0^t\sigma_s\,dW_s\,,\end{align}
on a filtered probability space $\big(\Omega^0,\mathcal{F},(\mathcal{F}_t)_{0\le t\leq1},\P^0\big)$ with $(\mathcal{F}_t)$ being a right-continuous and complete filtration and $W$ a one-dimensional standard Brownian motion. 
\begin{assump}\label{H1}
We pursue the asymptotic analysis under two structural hypotheses for the volatility process of which one must be satisfied:
\begin{itemize}
\item[($\sigma-1$)]
There exists a locally bounded process $L_t, 0\leq t\leq1$,
such that $t\mapsto\sigma_{t}$ is almost surely $\alpha$-H\"older
continuous on $\left[0,t\right]$ for $\alpha>1/2$ and H\"older
constant $L_t$, i.e.\,$\left|\sigma_{\tau}-\sigma_{s}\right|\leq L_t\left|\tau-s\right|^{\alpha},\,0\leq \tau,s\leq t$,
almost surely.
\item[($\sigma-2$)]
The process $\sigma_t$ is an It\^{o} semimartingale with locally bounded characteristics.
\end{itemize}
Furthermore, suppose $\sigma_{t}$ never vanishes. For the drift process, assume there exists a locally bounded process $L_t', 0\leq t\leq1$,
such that $b_s=g(b_s^A,b_s^B)$ with a function $g$, continuously differentiable in both
coordinates, and $t\mapsto b_{t}^B$ is almost surely $\nu$-H\"older continuous
on $\left[0,t\right]$ for $\nu>0$ and H\"older constant $L_t'$, i.e.\,$\left|b_{\tau}^B-b_{s}^B\right|\leq L_t'\left|\tau-s\right|^{\nu}$, $0\leq \tau,s\leq t$,
almost surely. The process $b_s^A,0\leq s\leq 1$, is an It\^{o} semimartingale with locally bounded characteristics.
\end{assump}
The assumption grants for $0\leq t+s\leq1$, $t\ge 0$, some constants $C_n,K_n>0$, some $\alpha>1/2$ and a sequence of stopping times $T_n$ increasing to $\infty$ that 
\begin{align}
\label{JM2}\Big|\E\big[\sigma_{(t+s)\wedge T_n}-\sigma_{t\wedge T_n}|\mathcal{F}_t\big]\Big| & \le C_n\,s^{\alpha}\,,\\
\label{JM}\E\Big[\sup_{\tau\in[0,s]}(\sigma_{(\tau+t)\wedge T_n}-\sigma_{t \wedge T_n})^2\Big] & \le K_n\,s\,.
\end{align}
The estimate \eqref{JM} corresponds to Assumption (L) and Equation (4.1) of \cite{jacodmykland}, and provides the bound which is required for the proofs. It does not exclude jumps, but fixed times of discontinuity. In particular $\sigma_t$ may be correlated with $X_t$, thus allowing for leverage. The mild smoothness assumption on $b$ turns out be essential, see the proof of Proposition \ref{propremainder}. \\ 
We work within the model where we have indirect observations of $X$ diluted by noise. 
\begin{assumpn}
\label{noise1}Let $\left(\epsilon_{t}\right)_{0\leq t\leq 1}$
be an i.i.d. white noise process with $\E[\epsilon_{t}^{8}]<\infty$
and variances $\E[\epsilon_{t}^{2}]=\eta^{2}>0$. We assume the noise
process is independent of $\c F$. Set $\mathcal{G}_{t}=\mathcal{F}_{t}\otimes\sigma(\epsilon_{s}:s\leq t)$ for $0\leq t\leq 1$ and let $(\Omega,\mathcal{G},(\mathcal{G}_{t})_{0\leq t\leq 1},\P)$
be a filtered probability space which accommodates the signal and
the noise processes and extends the space $(\Omega^{0},\c F,(\c F_{t})_{0\leq t\leq 1},\P^{0})$.
On this extension the semimartingale $X$ is observed at regular times
$i/n,\,i=0,\dots,n$, with additive noise: 
\begin{align}
Y_{i}=X_{i/n}+\epsilon_{i/n}\,,i=0,\ldots,n~.\label{observ}
\end{align}
\end{assumpn}
In fact, we consider a sequence of observed processes $(Y_i)_{i=0,\ldots,n}$ depending on $n$. For notational brevity, we shortly write $Y$ for the observation process. For more details on the extension of the probability space in a similar setup and a discussion of different assumptions for the noise process we refer to Equation (2.2) and Example 2.1 of \cite{jacodmykland}. Let us mention that presumably our method is not restricted to noise which is independent of the signal. This is merely a technical assumption simplifying the proofs.
For the multi-dimensional case, we focus on a $d$-dimensional continuous It\^{o} semimartingale
\begin{align}\label{smd}X_t&=X_0+\int_0^t b_s\,ds+\int_0^t\sigma_s\,dW_s\end{align}
on a filtered probability space $(\Omega^0,\mathcal{F},(\mathcal{F}_t)_{0\leq t\leq 1},\P^0)$ with $(\mathcal{F}_t)$ being a right-continuous and complete filtration and $W$ being here a $d$-dimensional $(\mathcal{F}_t)$-adapted standard Brownian motion. The {\textit{integrated covolatility matrix}} is denoted $\int_0^t \Sigma_s\,ds$, $\Sigma_s=\sigma_s\sigma_s^{\top}$. It coincides with the quadratic covariation matrix $[X,X]_t$ of the continuous semimartingale $X$. We denote the spectral norm by $\|\cdot\|$ and define $\|f\|_{\infty}\:=\sup_{t\in[0,1]}{\|f(t)\|}$ for functions $f:[0,1]\rightarrow\mathbb{R}^{d\times d'}$. Consider H\"older balls of order $\alpha\in(0,1]$ and with radius $R>0$:
\begin{align*}&C^{\alpha}(R)=\{f\in C^{\alpha}([0,1],\mathds{R}^{d\times d^{\prime}})|\|f\|_{C^{\alpha}}\le R\}\,,\,\|f\|_{C^{\alpha}}\:=\|f\|_{\infty}+\sup_{x\ne y}{\frac{\|f(x)-f(y)\|}{|x-y|^{\alpha}}}\,.\end{align*}
We assume the following regularity conditions. 
\begin{assum}\label{Hd}
The stochastic instantaneous volatility process $\sigma$ is a $(d\times d^{\prime})$-dimensional $(\mathcal{F}_t)$-adapted process satisfying one of the following regularity conditions:
\begin{itemize}
\item[($\Sigma-1$)]$\sigma\in C^{\alpha}(R)$ for some $R>0$ and with H\"older exponent $\alpha>1/2$.
\item[($\Sigma-2$)]
$\sigma$ is an It\^{o} semimartingale whose characteristics are assumed to be locally bounded.
\end{itemize}
Furthermore, we assume that the positive definite matrix $\Sigma_s$ satisfies $\Sigma_s\ge \underline\Sigma E_d$, where $E_d$ is the $d$-dimensional identity matrix, in the sense of L\"owner ordering of positive definite matrices. This is the analogue of ``$\sigma_t$ never vanishes'' for $d=1$. The drift $b$ is a $d$-dimensional $(\mathcal{F}_t)$-adapted process given by a function $g(b_s^A, b_s^B)$ which is continuously differentiable in all coordinates with $b^B\in C^{\nu}(R)$ for some $R>0,\nu>0$, and $b^A$ a $d$-dimensional It\^{o} semimartingale with locally bounded characteristics.
\end{assum}
We consider a very general framework with noise and in which observations come at non-syn\-chronous sampling times.
\begin{assumn}\label{noise2}
Let $\left(\epsilon_{t}\right)_{0\leq t\leq 1}$
be a $d$-dimensional i.i.d. white noise process with independent
components such that $\E[(\epsilon_{t}^{(l)})^{8}]<\infty$ and $\var(\epsilon_{t}^{(l)})=\eta_l^{2}>0$
for all $0\leq t\leq 1$ and $l=1,\dots,d$. We assume the noise process
is independent of $\c F$. Set $\mathcal{G}_{t}=\mathcal{F}_{t}\otimes\sigma(\epsilon_{s}:s\leq t)$
for $t\geq0$ and let $(\Omega,\mathcal{G},(\mathcal{G}_{t})_{0\leq t\leq1},\P)$
be a filtered probability space which accommodates the signal and
the noise processes and extends the space $(\Omega^{0},\c F,(\c F_{t})_{0\leq t\leq1},\P^{0})$.
The signal process $X$ of the form \eqref{smd} is discretely and
non-synchronously observed and subject to additive noise. Observation
times $t_{i}^{(l)},i=0,\dots,n_{l},l=1,\ldots,d$, are described by
quantile transformations $t_{i}^{(l)}=F_{l}^{-1}(i/n_{l})$, with differentiable possibly random distribution functions
$F_{l}, F_{l}(0)=0$, $F_{l}(1)=1$ and $F_{l}'\in C^{\alpha}([0,1],[0,1])$
with $\|F'\|_{C^{\alpha}}$ bounded for some $\alpha>1/2$ and $F_{l}'$
strictly positive. The observation times are independent of $X$.
We assume that $n/n_{l}\rightarrow\nu_{l}$ as $n\rightarrow\infty$
with constants $0<\nu_{l}<\infty$. Observations
are then given by 
\[
Y_{i}^{(l)}=X_{t_{i}^{(l)}}^{(l)}+\epsilon_{t_{i}^{(l)}}^{(l)}\,,i=0,\ldots,n_l,l=1,\ldots,d~.
\]
\end{assumn} 
Our analysis includes deterministic and random observation times which are independent of $Y$. Though Assumption \ref{noise2} displays to some extent still an idealization of realistic market microstructure dynamics, our observation model constitutes the established setup in related literature and captures the main ingredients of realistic log-price models.
\subsection{Mathematical concepts and notation\label{sec:2.2}}
Denote $\Delta_i^n Y^{(l)}=Y^{(l)}_{i}-Y^{(l)}_{{i-1}},i=1,\dots,{n_l}, l=1,\ldots,d,$ the increments of $Y^{(l)}$ and analogously for $X$ and other processes. In the one-dimensional case, we write $\Delta^n Y=(\Delta_i^n Y)_{i=1,\dots,n}\in\mathds{R}^n$ for the vector of increments. We shall write $Z_n=\mathcal{O}_{\P}(\delta_n)$ ($Z_n=\KLEINO_{\P}(\delta_n)$) for real random variables, to express that the sequence $\delta_n^{-1}Z_n$ is bounded (tends to zero) in probability under $\P$. Analogously $\mathcal{O}$ and $\KLEINO$ are used for deterministic sequences. To express that terms are of the same asymptotic order we write $Z_n{\asymp}_p Y_n$ if $Z_n=\mathcal{O}_{\P}(Y_n)$ and $Y_n=\mathcal{O}_{\P}(Z_n)$ and likewise $\asymp$ for deterministic terms. Also, we use the short notation $A_n\lesssim B_n$ for $A_n=\mathcal{O}(B_n)$. Convergence in probability and weak convergence are denoted by $\stackrel{\P}{\rightarrow}$ and $\stackrel{d}{\longrightarrow}$;$\stackrel{st}{\longrightarrow}$ refers to stable weak convergence with respect to $\mathcal{G}$ - if not further specified. We write $X^n\stackrel{ucp}{\longrightarrow}X$ for processes $X^n,X$ to express shortly that $\sup_{t\in[0,1]}|X_t^n-X_t|\stackrel{\P}{\rightarrow} 0$. $\delta_{lm}$ is Kronecker's delta, i.e.\,$\delta_{lm}=1$ for $l=m$, $\delta_{lm}=0$ else. For functional convergence of processes we focus on the space $\c D\left(\left[0,1\right]\right)$, the space of c\`{a}dl\`{a}g functions (right-continuous with left limits).\\[.2cm]
Recall the definition of stable weak convergence which is an essential concept in the asymptotic theory for volatility estimation. For a sub-$\sigma$-field $\mathcal{A}\subseteq\mathcal{F}$, a sequence of random variables $(X_n)$ taking values in a Polish space $(E,\mathcal{E})$ converges $\mathcal{A}$-stably, if
$$\lim_{n\rightarrow\infty}\E\left[Zf(X_n)\right]=\int_{\Omega\times E}\mu(d\omega,dx)Z(\omega)f(x)$$
with a probability measure $\mu$ on $(\Omega\times E,\mathcal{A}\otimes\mathcal{E})$ and for all continuous and bounded $f$ and $\mathcal{A}$-measurable bounded random variables $Z$. The definition is equivalent to joint weak convergence of $(Z,X_n)$ for every $\mathcal{A}$-measurable random variable $Z$. 
Thus $\mathcal{F}$-stable weak convergence means \(\lim_{n\rightarrow\infty}\E\left[f(X_n)Z\right]=\E^{\prime}\left[f(X)Z\right]\)
for all bounded continuous $f$ and bounded measurable $Z$, where the limit $X$ of $(X_n)$ is defined on an extended probability space $(\Omega^{\prime},{\mathcal{F}}^{\prime},{\P}^{\prime})$. In our setup, the extended space will be given by the orthogonal product of $(\Omega^0,\mathcal{F},\P^0)$ and an auxiliary space $(\tilde \Omega,\tilde {\mathcal{F}},\tilde{\P})$. We refer to \cite{podvet} for more information on stable convergence.\\[.2cm]
For the multi-dimensional setting the $\operatorname{vec}$-operator and Kronecker products of matrices will be important. For a matrix $A\in\R^{d\times d}$ we write the entries $A_{pq},p,q=1,\dots,d$, and the vector of its entries obtained by stacking its columns on top of one another
\begin{align*}
&\operatorname{vec}(A)=\left(A_{11},A_{21},\ldots,A_{d1},A_{12},A_{22},\ldots,A_{d2},
\ldots,A_{d(d-1)},A_{dd}\right)^{\top}\in{\mathds{R}}^{d^2}.
\end{align*}
The transpose of a matrix $A$ is denoted by $A^{\top}$. For matrix functions in time , for instance the covolatility matrix, we write the entries as $A^{(pq)}_t$.
The Kronecker product $A\otimes B\in\R^{d^2\times d^2}$ for $A,B\in\R^{d\times d}$ is defined as
\[ (A\otimes B)_{d(p-1)+q,d(p'-1)+q'}=A_{pp'}B_{qq'},\quad p,q,p',q'=1,\ldots,d.\]
In the multivariate limit theorems, we account for effects by non-commutativity of matrix multiplication. It will be useful to standardize limit theorems such that the matrix
\begin{align}\label{Z}
{\cal Z}=\Cov(\operatorname{vec}(ZZ^\top)),~\text{ for }Z\sim N(0,E_d)~~\mbox{standard Gaussian},
\end{align}
appears as variance-covariance matrix of the standardized form instead of the identity matrix. This matrix is the sum of the $d^2$-dimensional identity matrix $E_{d^2}$ and the so-called commutation matrix $C_{d,d}$ that maps a vectorized $(d\times d)$ matrix to the $\operatorname{vec}$ of its transpose, i.e.  $C_{d,d}\operatorname{vec}(A)=\operatorname{vec}(A^{\top})$. The matrix ${\cal{Z}}/2$ is idempotent and introduced in \cite{matrixalgebra}, Chapter 11, as the {\textit{symmetrizer matrix}}. Note that in the multi-dimensional experiment under equidistant synchronous non-noisy observations of $X$, the sample realized covolatility matrix \(\widehat{IC}=\sum_{i=1}^n(X_{i/n}-X_{(i-1)/n}) (X_{i/n}-X_{(i-1)/n})^{\top}\) obeys the central limit theorem:
\begin{align}\label{avarwithoutnoise}\hspace*{-1cm}n^{1/2}\,\operatorname{vec}\Big(\widehat{IC}-\int_0^1\Sigma_s\,ds\Big)\stackrel{st}{\longrightarrow}N\Big(0,\int_0^1\big(\Sigma_s\otimes \Sigma_s\big)\,\mathcal{Z}\,ds\Big)\,,\end{align}
where similarly as in our result below the matrix $\mathcal{Z}$ appears as one factor in the asymptotic variance and remains after standardization.
For background information on matrix algebra, especially tensor calculus using the Kronecker product and $\operatorname{vec}$-operator we refer interested readers to \cite{matrixalgebra}. Note the crucial relation between the Kronecker product and the $\operatorname{vec}$-operator $\operatorname{vec}(ABC)=(C^\top \otimes A)\operatorname{vec}(B)$.\\
We employ the notion of empirical scalar products in the fashion of \cite{bibingerreiss}, which is recalled in Definition \ref{scalar} in Section \ref{sec:6}, along with some useful properties.
\subsection{Outline of the main results\label{sec:2.3}}
In the sequel, we present the three major results of this work in Theorem \ref{thm:1}, Theorem \ref{thm:2} and Theorem \ref{thm:3} and concisely discuss the consequences. Theorems \ref{thm:1} and \ref{thm:2} establish functional stable central limit theorems in a general semimartingale setting for the spectral estimators of  \cite{reiss} and \cite{bibingerreiss}. Theorem \ref{thm:3} gives a multivariate limit theorem for the localized method of moment approach of \cite{BHMR}. These methods are briefly explained in a nutshell in Section \ref{sec:3.1}. Note that they attain asymptotic efficiency lower variance bounds in simplified models without drift, with independent volatility and covolatility processes and normally distributed noise.
\begin{theo}\label{thm:1}In the one-dimensional experiment, on Assumption \ref{H1} and Assumption \ref{noise1}, for the adaptive spectral estimator of integrated squared volatility $\IV_{n,t}$, stated in \eqref{spev} below, the functional stable weak convergence 
\begin{align}\label{clt1}n^{1/4}\left(\IV_{n,t}-\int_0^t\sigma_s^2\,ds\right)\stackrel{st}{\longrightarrow}\int_{0}^{t}\sqrt{8\eta\left|\sigma_{s}^{3}\right|}\, dB_{s}\end{align}
applies as $n\rightarrow\infty$ on $\c D\left[0,1\right]$, where $B$
is a Brownian motion defined on an extension of the original probability
space $(\Omega,\c G,(\c G_{t})_{0\leq t\leq1},\P)$, independent of
the original $\sigma$-algebra $\c G$. Moreover, the variance estimator $\widehat{\mathcal{V}}^{{\scriptscriptstyle{{\mathcal{IV}}}}}_{n,t}$ in \eqref{stspev} provides for fixed $0\le t\le 1$ the {\it{feasible}} central limit theorem:
\begin{align}\label{fclt1}\big(\widehat{\mathcal{V}}^{{\scriptscriptstyle{{\mathcal{IV}}}}}_{n,t}\big)^{-1/2}\left(\IV_{n,t}-\int_0^t\sigma_s^2\,ds\right)\stackrel{d}{\longrightarrow}N(0,1)\,.\end{align}
\end{theo}
\begin{remark}The convergence rate in \eqref{clt1} and \eqref{fclt1} is {\it{optimal}}, already in the parametric subexperiment, see \cite{gloter}. $\IV_{n,1}$ is asymptotically mixed normally distributed with random asymptotic variance $\int_0^18\eta\left|\sigma_{s}^{3}\right|\,ds$. This asymptotic variance coincides with the lower bound derived by \cite{reiss} in the subexperiment with time-varying but deterministic volatility, without drift and with Gaussian error distribution. The spectral estimator of squared integrated volatility is hence {\it{asymptotically efficient}} in this setting. For the general semimartingale experiment the concept of asymptotic efficiency is not developed yet, it is conjectured that the lower bound has analogous structure, see Remark 3.1 of \cite{jacodmykland}. Theorem \ref{thm:1} establishes that the asymptotic variance of the estimator has the same form in the very general framework, what we call {\it{quasi-efficient}}, and stable convergence holds true. The {\it{feasible limit theorem}} \eqref{fclt1} allows to provide confidence intervals and is of pivotal importance for practical capability.\end{remark}
\begin{theo}\label{thm:2}In the multi-dimensional experiment, on Assumption \ref{Hd} and Assumption \ref{noise2}, for the adaptive spectral estimator of integrated covolatility $\ICV_{n,t}^{(p,q)}$, stated in \eqref{specv} below, the functional stable weak convergence 
\begin{align}\label{clt2}n^{1/4}\left(\ICV^{(p,q)}_{n,t}-\int_0^t\Sigma_s^{(pq)}\,ds\right)\stackrel{st}{\longrightarrow}\int_{0}^{t} v_s^{(p,q)}\, dB_{s}\end{align}
applies for $n/n_p\rightarrow \nu_p$ and $n/n_q\rightarrow \nu_q$ with $0<\nu_p,\nu_q<\infty$, as $n\rightarrow\infty$ on $\c D\left[0,1\right]$, where $B$
is a Brownian motion defined on an extension of the original probability
space $(\Omega,\c G,(\c G_{t})_{0\leq t\leq1},\P)$, independent of
the original $\sigma$-algebra $\c G$. The asymptotic variance process is given by
\begin{align}\notag \big(v_s^{(p,q)}\big)^2&=2\left({(F_p^{-1})}^{\prime}(s){(F_q^{-1})}^{\prime}(s)\nu_p\nu_q(A_s^2-B_s)B_s\right)^{1/2}\\
&\quad\quad \times \label{avarspecv}\big(\sqrt{A_s+\sqrt{A_s^2-B_s}}-\operatorname{sgn}(A_s^2-B_s)\sqrt{A_s-\sqrt{A_s^2-B_s}}\big)\,,\end{align}
with the terms
$$A_s=\Sigma_s^{(pp)}\frac{{(F_q^{-1})}^{\prime}(s)\nu_q}{{(F_p^{-1})}^{\prime}(s)\nu_p}+\Sigma_s^{(qq)}\frac{{(F_p^{-1})}^{\prime}(s)\nu_p}{{(F_q^{-1})}^{\prime}(s)\nu_q}\,,$$
$$B_s=4\left(\Sigma_s^{(pp)}\Sigma_s^{(qq)}+\big(\Sigma_s^{(pq)}\big)^2\right)\,.$$ 
$\operatorname{sgn}$ denotes the sign taking values in $\{-1,+1\}$ and ensuring that the value of $\big(v_s^{(p,q)}\big)^2$ is always a positive real number.\\
Moreover, the variance estimator $\widehat{\mathcal{V}}^{{\scriptscriptstyle{{\mathcal{ICV}^{(p,q)}}}}}_{n,t}$ in \eqref{stspecv} provides for fixed $0\le t\le 1$ the {\it{feasible}} central limit theorem:
\begin{align}\label{fclt2}\big(\widehat{\mathcal{V}}^{{\scriptscriptstyle{{\mathcal{ICV}}^{(p,q)}}}}_{n,t}\big)^{-1/2}\left(\ICV^{(p,q)}_{n,t}-\int_0^t\Sigma_s^{(pq)}\,ds\right)\stackrel{d}{\longrightarrow}N(0,1)\,.\end{align}
\end{theo}
\begin{remark}The bivariate extension of the spectral method outperforms by its local adaptivity and Fourier domain smoothing previous approaches for integrated covolatility estimation in most cases, see \cite{bibingerreiss} for a detailed discussion and survey on the different methods. Yet, it attains the multi-dimensional variance lower bound in the submodel for estimating the integrated covolatility $\int_0^1\Sigma_s^{(pq)}\,ds$ only in case of zero correlation. On the other hand, the estimator already achieves a high efficiency and since it does not involve Fisher information weight matrices, it is less computationally costly than the efficient local method of moments approach. The general form of the asymptotic variance given in \eqref{avarspecv} looks a bit tedious. In case that $\Sigma_t^{(12)}= 0$ and for equal volatilities $\Sigma^{(11)}_t=\Sigma^{(22)}_t=\sigma_t$, it simplifies to $\int_0^t 4\eta|\sigma_s^3|\,ds$ which is efficient for this setup. By its rescaled version in \eqref{fclt2} allowing for confidence bounds, the estimator is of high practical value.\end{remark}
\begin{theo}\label{thm:3}In the multi-dimensional experiment, on Assumption \ref{Hd} and Assumption \ref{noise2}, for the local method of moments estimator of the vectorized integrated covolatility matrix $\LMM_{n,t}$, stated in \eqref{lmm} below, the functional stable weak convergence 
\begin{align}
\hspace*{-0.1cm}n^{1/4}\hspace*{-0.1cm}\left(\hspace*{-0.05cm}\LMM_{n,t}-\operatorname{vec}\hspace*{-0.05cm}\Big(\int_{0}^{t}\Sigma_{s}\,ds\Big)\hspace*{-0.05cm}\right)\hspace*{-0.05cm}\stackrel{st}{\longrightarrow}\hspace*{-0.1cm}\int_{0}^{t}\hspace*{-0.1cm}\big(\Sigma_{s}^{\frac{1}{2}}\otimes\big(\Sigma_{s}^{\mathcal{H}}\big)^{\frac{1}{4}}\big)\mathcal{Z}dB_{s}+\hspace*{-0.1cm}\int_{0}^{t}\hspace*{-0.1cm}\big(\big(\Sigma_{s}^{\mathcal{H}}\big)^{\frac14}\otimes\Sigma_{s}^{\frac12}\big)\mathcal{Z}dB_{s}^{\bot}\label{clt3}
\end{align}
applies, with ${\cal H}(t)=\operatorname{diag}(\eta_p\nu_p^{1/2}F_p'(t)^{-1/2})_{p}\in\R^{d\times d}$ and $\big(\Sigma^{\cal H}\big)^{1/4}$ the matrix square root of \(\big(\Sigma^{\cal H}\big)^{1/2}:={\cal H}({\cal H}^{-1}\Sigma {\cal H}^{-1})^{1/2}{\cal H}\), as $n\rightarrow\infty$ and $n/n_p\to \nu_p$ for $p=1,\ldots,d$, on $\c D\left[0,1\right]$, where $\mathcal{Z}$ is the matrix defined in \eqref{Z} and $B$ and $B^{\bot}$ are two independent $d^2$-dimensional Brownian motions, both defined on an extension of the original probability space $(\Omega,\c G,(\c G_{t})_{0\leq t\leq1},\P)$, independent of
the original $\sigma$-algebra $\c G$. For fixed time $t=1$ in \eqref{clt3}, the point-wise marginal central limit theorem reads
\begin{align}n^{1/4}\left(\LMM_{n,1}-\operatorname{vec}\Big(\int_0^1\Sigma_s\,ds\Big)\right)\stackrel{st}{\longrightarrow}
MN\left(0,\mathbf{I}^{-1}{\cal Z}\right)\,,\end{align}
where $MN$ means mixed normal distribution, with the asymptotic variance-covariance matrix 
\begin{align}\label{acov} \mathbf{I}^{-1}=2
\int_0^1(\Sigma_s\otimes \big(\Sigma_s^{\cal H}\big)^{1/2} + \big(\Sigma_s^{\cal H}\big)^{1/2} \otimes \Sigma_s)\,ds\,.
\end{align}
Moreover, the variance-covariance matrix estimator ${\bf{I}}_{n,t}^{-1}$ in \eqref{stlmm} provides for fixed $0\le t\le 1$ the {\it{feasible}} central limit theorem:
\begin{align}\label{fclt3}{\bf{I}}_{n,t}^{1/2}\left(\LMM_{n,t}-\operatorname{vec}\Big(\int_0^t\Sigma_s\,ds\Big)\right)\stackrel{d}{\longrightarrow}N\left(0,{\cal Z}\right)\,.\end{align}
\end{theo}
\begin{remark}The local method of moments attains the lower asymptotic variance bound derived in \cite{BHMR} for a nonparametric experiment with deterministic covolatility matrix, without drift and with Gaussian error distribution. Thus, the local method of moments is {\it{asymptotically efficient}} in this subexperiment. In the general setup it is {\it{quasi-efficient}}.\\
The asymptotic variance of estimating integrated squared volatility decreases as we can benefit from observing correlated components. In the multi-dimensional observation model the minimum asymptotic variance can become much smaller than the bound in \eqref{clt1} for $d=1$. In an idealized parametric model with $\sigma>0$, the variance can be reduced up to $(8/\sqrt{d})\eta\sigma^3$ in comparison to the one-dimensional lower bound $8\eta\sigma^3$, see \cite{BHMR} for a deeper discussion of the lower bound. In view of the complex geometry of the general multi-dimensional parameter space, expression \eqref{acov} provides a neat description of the asymptotic variance bound.\end{remark}

\section{Spectral estimation of integrated volatility and the integrated covolatility matrix\label{sec:3}}
\subsection{Elements of spectral estimation\label{sec:3.1}}
We shall concisely review the building blocks of spectral estimation.
For simplicity we start with the one-dimensional framework, $d=1$.
We partition the time interval $[0,1]$ into equidistant bins $[(k-1)h_{n},kh_{n}]$,
$k=1,\ldots,h_{n}^{-1}\in\N$, $h_{n}\rightarrow0$ as $n\rightarrow\infty$,
such that $nh_{n}$ is the average number of observations per bin. Consider a statistical experiment in which we approximate $\sigma_t$ by a locally constant function. Precisely, on each bin $[(k-1)h_n,kh_n]$ a piecewise constant squared volatility $\sigma^2_{(k-1)h_n}$ is considered. In this experiment the task is then
to estimate on every block $\int_{(k-1)h_n}^{kh_n}\sigma^2_s\,ds$ by $h_n\hat{\sigma}^2_{(k-1)h_n}$, solving locally parametric estimation problems.
For this purpose \cite{reiss} motivated to use \textit{{spectral
statistics}} 
\begin{align}
S_{jk}=\sum_{i=1}^{n}\Delta_{i}^{n}Y\Phi_{jk}\left(\frac{i}{n}\right)~,j=1,\ldots,\lfloor nh_{n}\rfloor-1,k=1,\ldots,h_{n}^{-1},\label{spec}
\end{align}
which are discrete analogues of expressions obtained from diagonalizing
the covariance operator of observations in an equivalent white noise
experiment. Here, 
\begin{align}
\Phi_{j}(t)=\sqrt{\frac{2}{h_{n}}}\sin\left(j\pi{h_{n}^{-1}}t\right)\I_{\left[0,h_{n}\right]}(t)\,,\,\Phi_{jk}\left(t\right)=\Phi_{j}(t-(k-1)h_{n})
,\quad j\geq1,\,0\leq t\leq1\,,\label{phi}
\end{align}
\begin{align}
\varphi_{j}\left(t\right)=2n\sqrt{\frac{2}{h_{n}}}\sin\left(\frac{j\pi}{2 nh_{n} }\right)\cos\left(j\pi{h_{n}^{-1}}t\right)\I_{\left[0,h_{n}\right]}\left(t\right)\,,\,\varphi_{jk}\left(t\right)=\varphi_{j}(t-(k-1)h_{n})
\,,\label{varphi}
\end{align}
are systems of trigonometric functions orthogonal with respect to empirical
scalar products $\left\langle \cdot,\cdot\right\rangle _{n}$ and
$\left[\cdot,\cdot\right]_{n}$, respectively, see Section \ref{sec:6.1}. Efficient estimators $\hat{\sigma}_{(k-1)h_{n}}^{2}$
for $\sigma_{(k-1)h_{n}}^{2}$ are then constructed by weighted linear
combinations of bias-corrected squared spectral statistics over different frequencies.
Consequently, in this experiment we estimate $\int_{(k-1)h_{n}}^{kh_{n}}\sigma_{s}^{2}\,ds$
by $h_{n}\hat{\sigma}_{(k-1)h_{n}}^{2}$.

The spectral statistics are the principal elements of the considered
estimation techniques. They are related to the pre-averages of \cite{JLMPV}
which have been designed for our one-dimensional estimation problem,
as well. A main difference is that we keep the bins fixed which makes
the construction of the spectral approach simple. Bin-wise the spectral
estimation profits from an advanced smoothing method in the frequency
domain, i.e.\,using the weight function of a discrete sine transformation.
The spectral statistics hence \textit{de-correlate} the observations
and form their bin-wise principal components. The methodology can
be viewed also as localizing on bins the estimator by \cite{corsi}.
\cite{reiss} showed that this leads to a semiparametrically efficient
estimation approach of squared integrated volatility in a nonparametric
setup with deterministic volatility, without drift and normally distributed
noise. The bin-width is chosen as $h_{n}\asymp n^{-1/2}\log{n}$ to
attain the optimal convergence rates and for the results in Section
\ref{sec:2.1}. This becomes clear in the proofs in Section \ref{sec:6}.
The log-factor plays a role in the convergence of the sum of variances
over different frequencies. The leading asymptotic order $n^{-1/2}$
for the bin-width is analogous to the pre-average and kernel bandwidths,
cf.\,\citet{JLMPV} and \citet{bn2}, and balances the discretization
error which increases with increasing $h_{n}$ and the error due to
noise which decreases as $h_{n}$ increases. Let us point out that
the basis functions \eqref{phi} and \eqref{varphi} are slightly
scaled versions of the respective basis functions in \cite{bibingerreiss}
and \cite{BHMR} for a more convenient exposition. 

\subsection{The spectral estimator of integrated volatility\label{sec:3.2}}
In the sequel, we use the empirical norm of the functions \eqref{varphi}, see Definition \ref{scalar} and \eqref{or2}.
Locally parametric estimates for the squared volatility $\hat{\sigma}_{(k-1)h_{n}}^{2}$
are obtained by weighted linear combinations with weights $w_{jk}$
of bias-corrected squared spectral statistics: 
\begin{align}
\hat{\sigma}_{(k-1)h_{n}}^{2}=\sum_{j=1}^{\lfloor nh_{n}\rfloor-1}w_{jk}\Big(S_{jk}^{2}-[\varphi_{jk},\varphi_{jk}]_{n}\,\frac{\eta^{2}}{n}\Big)\,.\label{locvolest}
\end{align}
Asymptotically, we consider infinitely many frequencies, but a natural
maximal frequency is $\lfloor nh_{n}\rfloor-1$\footnote{The proofs reveal that even a largest cut-off frequency $J_{n}\asymp\log n$
suffices and higher frequencies are negligible.}. The correction for the bias due to noise incorporates the noise
level $\eta$ which is in general unknown -- but can be consistently
estimated from the data with $n^{1/2}$ convergence rate, e.\,g.\,by\\
 $\hat{\eta}^{2}=(2n)^{-1}\sum_{i=1}^{n}(\Delta_{i}^{n}Y)^{2}$, see
\citet{zhangmykland} for an asymptotic analysis of this estimator.
The principle of bias-correcting the squared spectral statistics still
relates to the early estimator by \citet{zhou} for volatility estimation
under microstructure noise.\\
The estimator of the integrated squared volatility $\int_{0}^{t}\sigma_{s}^{2}\,ds$
is constructed as Riemann sum 
\begin{align}
\sum_{k=1}^{\lfloor th_{n}^{-1}\rfloor}h_{n}\hat{\sigma}_{(k-1)h_{n}}^{2}=\sum_{k=1}^{\lfloor th_{n}^{-1}\rfloor}h_{n}\sum_{j=1}^{\lfloor nh_{n}\rfloor-1}w_{jk}\Big(S_{jk}^{2}-[\varphi_{jk},\varphi_{jk}]_{n}\,\frac{\eta^{2}}{n}\Big)\,,\label{eq:Riemann_sum}
\end{align}
such that the estimator at $t=1$ becomes simply the average of local
estimates in the case of equi\-spaced bins. Set $I_{jk}=\big(\var\big(S_{jk}^{2}\big)\big)^{-1}$
and $I_{k}=\sum_{j=1}^{\lfloor nh_{n}\rfloor-1}I_{jk}$. The variance
of the above estimator becomes minimal and equal to $\sum_{k=1}^{\lfloor th_{n}^{-1}\rfloor}h_n^2I_{k}^{-1}$
for the oracle weights 
\begin{align}
w_{jk}=I_{k}^{-1}I_{jk}=\frac{\left(\sigma_{\left(k-1\right)h_{n}}^{2}+\frac{\eta^{2}}{n}\left[\varphi_{jk},\varphi_{jk}\right]_{n}\right)^{-2}}{\sum_{m=1}^{\lfloor nh_{n}\rfloor-1}\left(\sigma_{\left(k-1\right)h_{n}}^{2}+\frac{\eta^{2}}{n}\left[\varphi_{mk},\varphi_{mk}\right]_{n}\right)^{-2}}\label{weights}
\end{align}
for $k=1,\dots,h_{n}^{-1}$ and $j=1,\dots,\lfloor nh_{n}\rfloor-1$,
when the noise is Gaussian and $\sigma_{(k-1)h_{n}}$ is bin-wise
constant. For general noise distribution it turns out that the first-order
variance is not affected. It is essential to develop an adaptive version
of the estimator, for which we replace the oracle optimal weights
by data-driven estimated optimal weights. Additionally to the estimated
noise variance, a bin-wise consistent estimator of the local volatilities
$\sigma_{(k-1)h_{n}}^{2}$ with some convergence rate suffices. Local
pre-estimates of the volatilities $\sigma_{(k-1)h_{n}}^{2}$ can be
constructed by using the same ansatz as in \eqref{locvolest}, but
involving only $J_{n}\ll\lfloor nh_{n}\rfloor-1$ frequencies and
constant weights $w_{jk}=J_{n}^{-1}$ and then averaging over $2K_{n}+1\asymp n^{1/4}$
bins in a neighborhood of $(k-1)h_{n}$: 
\begin{align}\label{pilot1}\hat\sigma_{(k-1)h_n}^{2,pilot}=(2K_{n}+1)^{-1}\sum_{m=(k-1-K_{n})\vee 1}^{(k-1+K_{n})\wedge h_n^{-1}}\hspace*{-.1cm}J_{n}^{-1}\,\sum_{j=1}^{J_{n}}\Big(S_{jm}^2-[\varphi_{jm},\varphi_{jm}]_n\,\frac{\hat\eta^2}{n}\Big)\,.\end{align}
This estimator attains $n^{1/8}$ as rate of convergence in case of
$\alpha\approx1/2$ under $(\sigma-1)$ or under $(\sigma-2)$ in
Assumption \ref{H1}. The estimated weights are then given by $\hat{w}_{jk}=\hat{I_{k}}^{-1}\hat{I}_{jk}$
where $\hat{I}_{k},\hat{I}_{jk}$ are obtained as above in \eqref{weights} but plugging
in the pre-estimates of local squared volatilities and of the noise
variance. We can define the final fully adaptive spectral estimator
of integrated volatility and the estimator for its variance based
on a two-stage approach: \begin{subequations} 
\begin{align}
\IV_{n,t} & =\sum_{k=1}^{\lfloor th_{n}^{-1}\rfloor}h_{n}\sum_{j=1}^{\lfloor nh_{n}\rfloor-1}\hat{w}_{jk}\Big(S_{jk}^{2}-\frac{\hat{\eta}^{2}}{n}\,[\varphi_{jk},\varphi_{jk}]_{n}\Big)\,,\label{spev}\\
\widehat{\mathcal{V}}_{n,t}^{{\scriptscriptstyle {\mathcal{IV}}}} & =\sum_{k=1}^{\lfloor th_{n}^{-1}\rfloor}h_{n}^{2}\,\hat{I}_{k}^{-1}\,.\label{stspev}
\end{align}
\end{subequations} 

\subsection{The spectral covolatility estimator\label{sec:3.3}}
The spectral covolatility estimator from \citet{bibingerreiss} is
the obvious extension of the one-dimensional estimator using cross-products
of spectral statistics: 
\begin{align}
S_{jk}^{(p)}=\sum_{i=1}^{n_{p}}\Delta_{i}^{n}Y^{(p)}\Phi_{jk}\left(\frac{t_{i}^{(p)}+t_{i-1}^{(p)}}{2}\right)~\,,\,j\ge1,p=1,\ldots,d,k=1,\ldots,h_{n}^{-1}.\label{specavg}
\end{align}
The basis functions $(\Phi_{jk})$ are defined as in \eqref{phi}.
Under non-synchronous observations, we can use instead of \eqref{varphi}
the simpler expression $\varphi_{jk}=\Phi_{jk}^{\prime}$ \footnote{This meets the original idea by \citet{reiss} for continuous-time
observations to use orthogonal systems of functions and their derivatives.
While in the case of regular observations on the grid $i/n,i=0,\ldots,n$,
we can slightly profit by the discrete Fourier analysis and the exact
form of \eqref{varphi}, for non-synchronous observations we rely
on continuous-time analogues as approximation which coincide with
the first-order discrete expressions.}. Instead of the empirical norm we now use $[\varphi_{jk},\varphi_{jk}]=\int_{0}^{1}\varphi_{jk}^{2}(t)\,dt=h_{n}^{-2}\pi^{2}j^{2}$. In the multi-dimensional setup we introduce a diagonal matrix function of noise levels $\mathcal{H}(t)=\operatorname{diag}\big(\eta_l\big(\nu_l(F_l^{-1})^{\prime}(t)\big)^{1/2}\big)_{l=1,\dots,d}$ incorporating constants $\nu_l$ when $n/n_l\rightarrow\nu_l$. By a locally constant approximation of the observation frequencies we get a bin-wise locally constant approximation of $\c {H}$:
\begin{align}\label{noiselevel}{\bf{H}}_k^n=\operatorname{diag}(n^{-1}\eta_l^2\nu_l(F_l^{-1})^{\prime}((k-1)h_n)\big)_{l=1,\dots,d}=\operatorname{diag}(H_l^{kh_n})_{l=1,\dots,d}\,.\end{align} 
The optimal weights are $\begin{aligned}w_{jk}^{p,q}=(I_{k}^{(p,q)})^{-1}I_{jk}^{(p,q)}\,\end{aligned}
$, where
\begin{align}\label{eq:-4}
I_{j(k+1)}^{(p,q)}\hspace*{-0.075cm}=\hspace*{-0.075cm}\left(\Sigma_{kh_{n}}^{(pp)}\Sigma_{kh_{n}}^{(qq)}\hspace*{-0.075cm}+\hspace*{-0.075cm}(\Sigma_{kh_{n}}^{(pq)})^{2}\hspace*{-0.075cm}+\hspace*{-0.075cm}H_{p}^{kh_{n}}H_{q}^{kh_{n}}\hspace*{-0.075cm}\left[\varphi_{jk},\varphi_{jk}\right]^{2}+\hspace*{-0.075cm}\big(\Sigma_{kh_{n}}^{(pp)}H_{q}^{kh_{n}}+\Sigma_{kh_{n}}^{(qq)}H_{p}^{kh_{n}}\big)\hspace*{-0.075cm}\left[\varphi_{jk},\varphi_{jk}\right]\right)\hspace*{-0.05cm}^{-1}.
\end{align}
They depend on the volatilities, covolatility and noise levels of the considered
components as defined in \eqref{noiselevel}. The local noise level
combines the global noise variance $\eta_{p}^{2}$ and local observation
densities. It can be estimated by 
\begin{align}
\hat{H}_{p}^{kh_{n}}=\frac{\sum_{i=1}^{n_{p}}\big(\Delta_{i}^nY^{(p)}\big)^{2}}{2n_ph_{n}}\sum_{kh_{n}\le t_{v}^{(p)}\le(k+1)h_{n}}\hspace*{-0.15cm}\big(t_{v}^{(p)}-t_{v-1}^{(p)}\big)^{2}\,,\label{noisehat}
\end{align}
see the asymptotic identity \eqref{qvt} below. Averaging empirical
covariances $S_{jk}S_{jk}^{\top}$ over different spectral frequencies
$j=1,\ldots,J_{n}$ and over a set of $2K_{n}+1$ adjacent bins
yields a consistent estimator of the instantaneous covolatility matrix:
\begin{align}
\hat{\Sigma}_{(k-1)h_n}^{pilot}=(2K_{n}+1)^{-1}\sum_{m=(k-1-K_{n})\vee 1}^{(k-1+K_{n})\wedge h_n^{-1}}\hspace*{-0.1cm}J_{n}^{-1}\,\sum_{j=1}^{J_{n}}\Big(S_{jk}S_{jk}^{\top}-\hat{{\bf {H}}}_{k}^{n}\Big)\,,\label{pilot}
\end{align}
with $\hat{{\bf H}}_{k}^{n}$ the estimated noise levels matrix \eqref{noiselevel}.\\
The two latter estimators provide adaptive pre-estimated optimal weights
${\hat{w}}_{jk}^{p,q}$, again by plug-in in \eqref{eq:-4}. The bivariate
spectral covolatility estimator with adaptive weights for $p\ne q,p,q=1,\ldots,d$,
is \begin{subequations} 
\begin{align}
\ICV_{n,t}^{(p,q)} & =\sum_{k=1}^{\lfloor th_{n}^{-1}\rfloor}h_{n}\sum_{j=1}^{\lfloor nh_{n}\rfloor-1}{\hat{w}}_{jk}^{p,q}\Big(S_{jk}^{(p)}S_{jk}^{(q)}\Big)\,.\label{specv}
\end{align}
Differently as in \eqref{spev} we do not have to correct the cross
product $S_{jk}^{(p)}S_{jk}^{(q)}$ for a bias, since the noise is
component-wise independent. The estimator of the variance is: 
\begin{align}
\widehat{\mathcal{V}}_{n,t}^{{\scriptscriptstyle {{\mathcal{ICV}}^{(p,q)}}}} & =\sum_{k=1}^{\lfloor th_{n}^{-1}\rfloor}h_{n}^{2}\,\big(\hat{I}_{k}^{(p,q)}\big)^{-1}\,.\label{stspecv}
\end{align}
A more general version of the spectral covolatility estimator for
a model including cross-correlation of the noise (in a synchronous
framework) can be found in \citet{bibingerreiss}. For a simpler exposition
and since this notion of cross-correlation is not adequate for the
more important non-synchronous case, we restrict ourselves here to
noise according to Assumption \ref{noise2}. \end{subequations} 

\subsection{Local method of moments\label{sec:3.4}}
Consider the \textit{{vectors}} of spectral statistics: 
\begin{align}
S_{jk}=\Big(\sum_{i=1}^{n_{p}}\Delta_{i}^{n}Y^{(p)}\Phi_{jk}\Big(\frac{t_{i}^{(p)}+t_{i-1}^{(p)}}{2}\Big)\Big)_{p=1,\dots,d}\,,\label{specmv}
\end{align}
for all $k=1,\ldots,h_{n}^{-1}$ and $j\ge1$.\\
The fundamental novelty of the local method of moments approach is
to involve multivariate Fisher informations as optimal weight matrices
which are $(d^{2}\times d^{2})$ matrices of the following form: 
\begin{align}
\hspace*{-0.1cm}W_{jk}=I_{k}^{-1}I_{jk}=\hspace*{-0.1cm}\Big(\sum_{u=1}^{\lfloor nh_{n}\rfloor-1}\hspace*{-0.175cm}\big(\Sigma_{(k-1)h_{n}}+[\varphi_{uk},\varphi_{uk}]{\bf {H}}_{k}^{n}\big)\hspace*{-0.05cm}^{-\otimes2}\Big)^{-1}\hspace*{-0.075cm}\big(\Sigma_{(k-1)h_{n}}+[\varphi_{jk},\varphi_{jk}]{\bf {H}}_{k}^{n}\big)\hspace*{-0.05cm}^{-\otimes2},\label{weightsmv}
\end{align}
with $I_{jk}^{-1}=\Cov(S_{jk}S_{jk}^{\top})$, where $A^{\otimes2}=A\otimes A$ denotes the Kronecker product of
a matrix with itself and $A^{-\otimes2}=(A^{\otimes2})^{-1}=(A^{-1})^{\otimes2}$.
The main difference to estimators \eqref{spev} and \eqref{specv}
is that for estimating one specific (co-) volatility of one (two)
components, estimator \eqref{weightsmv} does not only rely on observations
of the one (two) considered component(s) but profits from information
inherent in all other components with some correlation to the considered
ones. In general, this facilitates a much smaller variance in the
multivariate model.\\
With the pilot estimates \eqref{pilot} and estimators for the noise
level \eqref{noisehat} at hand, we derive estimated optimal weight
matrices for building a linear combination over spectral frequencies
$j=1,\ldots,\lfloor nh_{n}\rfloor-1$, similar as above. The final
estimator of the vectorization of the integrated covolatility matrix
$\operatorname{vec}(\int_{0}^{t}\Sigma_{s}\,ds)$, becomes \begin{subequations}
\begin{align}
\LMM_{n,t}=\sum_{k=1}^{\lfloor th_{n}^{-1}\rfloor}h_{n}\sum_{j=1}^{\lfloor nh_{n}\rfloor-1}\hat{W}_{jk}\operatorname{vec}\left(S_{jk}S_{jk}^{\top}-\hat{{\bf {H}}}_{k}^{n}\right)\,,\label{lmm}
\end{align}
and the estimator of its variance-covariance matrix: 
\begin{align}
\hat{\mathbf{I}}_{n,t}^{-1}=\sum_{k=0}^{\lfloor th_{n}^{-1}\rfloor}h_{n}^{2}\Big(\sum_{j=1}^{\lfloor nh_{n}\rfloor-1}\hat{I}_{jk}\Big)^{-1}\,.\label{stlmm}
\end{align}
\end{subequations} Compared to the approach by \citet{lixiu}, though
given similar names, our method is quite different. One common feature
is the two-stage adaptivity where pre-estimated spot volatilities
are plugged in for the final estimator.

\section{Asymptotic theory\label{sec:4}}
We start with the one-dimensional experiment. We decompose $X$ as
\begin{subequations}
\begin{align}X_t=X_0+\tilde X_t+(X_t-X_0-\tilde X_t)\,,\end{align} where $\tilde X$ is a simplified process without drift and with a piecewise constant approximation of the volatility:
\begin{align}\label{tildeX}\tilde X_t=\int_0^t \sigma_{\lfloor sh_n^{-1}\rfloor h_n}\,dW_s~.\end{align}
\end{subequations}
The asymptotic theory of the spectral estimators is conducted first with optimal oracle weights. Below, the effect of a pre-estimation of the weights for the fully adaptive estimator is shown to be asymptotically negligible at first order.
In the following, we distinguish between $\IV_{n,t}^{or}(Y)$, the oracle version of the spectral volatility estimator \eqref{spev}, and $\IV^{or}_{n,t}(\tilde X+\epsilon)$ for the oracle estimator in a simplified experiment in which $\tilde X$ instead of $X$ is observed with noise. It turns out that both have the same asymptotic limiting distribution, see Proposition \ref{propremainder}. In order to establish a functional limit theorem, we decompose the estimation error of the oracle version of \eqref{spev} in the following way:
\begin{subequations}
\begin{align}\label{dec1}\IV_{n,t}^{or}(Y)-\int_0^t\sigma_s^2\,ds&=\IV_{n,t}^{or}(\tilde X+\epsilon)-\int_0^t\sigma^2_{\lfloor sh_n^{-1}\rfloor h_n }\,ds\\
&\label{dec2}+\IV^{or}_{n,t}(Y)-\IV^{or}_{n,t}(\tilde X+\epsilon)-\int_0^t\big(\sigma^2_s-\sigma^2_{\lfloor sh_n^{-1}\rfloor h_n }\big)\,ds\,.\end{align}
\end{subequations}
The proof of the functional central limit theorem (CLT) falls into three major parts. First, we prove the result of Theorem \ref{thm:1} for the right-hand side of \eqref{dec1}. In the second step the approximation error in \eqref{dec2} is shown to be asymptotically negligible. Finally, we establish that the same functional stable CLT carries over to the adaptive estimators by proving that the error of the plug-in estimation of optimal weights is asymptotically negligible.
\begin{prop}\label{propmain}On the assumptions of Theorem \ref{thm:1}, it holds true that
\begin{align}n^{1/4}\Bigg(\IV_{n,t}^{or}(\tilde X+\epsilon)-h_n\sum_{k=1}^{\lfloor th_n^{-1}\rfloor}\sigma_{\left(k-1\right)h_n}^{2}\Bigg)\stackrel{st}{\longrightarrow}\int_{0}^{t}\sqrt{8\eta\left|\sigma_{s}^{3}\right|}\, dB_{s}\,,\end{align}
as $n\rightarrow\infty$ on $\mathcal{D}\left[0,1\right]$ where $B$
is a Brownian motion defined on an extension of the original probability
space $(\Omega,\mathcal{G},(\mathcal{G}_{t})_{0\leq t\leq1},\P)$, independent of
the original $\sigma$-algebra $\mathcal{G}$.
\end{prop}
\begin{prop}\label{propremainder}
On the assumptions of Theorem \ref{thm:1}, it holds true that: 
\begin{align}n^{1/4}\Big(\IV_{n,t}^{or}(Y)-\IV_{n,t}^{or}(\tilde X+\epsilon)-\int_0^t\big(\sigma^2_s-\sigma^2_{\lfloor sh_n^{-1}\rfloor h_n }\big)\,ds\Big)\stackrel{ucp}{\longrightarrow} 0,~\mbox{as}~n\rightarrow \infty\,.\end{align}
\end{prop}
Theorem \ref{thm:1} is then an immediate consequence of the following proposition:
\begin{prop}\label{tight1}On the assumptions of Theorem \ref{thm:1}:
\begin{align}n^{1/4}\left(\IV_{n,t}-\IV_{n,t}^{or}(Y)\right)\stackrel{ucp}{\longrightarrow} 0,~~~~~\mbox{as $n\rightarrow \infty$}\,.\end{align}
\end{prop}
Finally, by consistency of the variance estimators and stable convergence the feasible limit theorems for the adaptive estimators are valid. The asymptotic negligibility of the plug-in estimation in Proposition \ref{tight1} is proven in Section \ref{sec:6} exploiting a uniform bound on the derivative of the weights as function of $\sigma_t$. In fact, it turns out that the weights are robust enough in misspecification of the pre-estimated local volatility to render the difference between oracle and adaptive estimator asymptotically negligible. This carries over to the multivariate methods.

The proof of the functional stable CLT is based on the asymptotic theory developed by \cite{jacodkey}. In order to apply Theorem 3--1 of \cite{jacodkey} (or equivalently Theorem 2.6 of \cite{podvet}), we illustrate the rescaled estimation error as a sum of increments:
\begin{eqnarray}
n^{1/4}\Bigg(\IV_{n,t}^{or}(\tilde X+\epsilon)-h_n\sum_{k=1}^{\lfloor t h_n^{-1}\rfloor}\sigma_{\left(k-1\right)h_n}^{2}\Bigg)= \sum_{k=1}^{\lfloor{th_n^{-1}}\rfloor}\zeta_{k}^{n}\,,\\
\zeta_{k}^{n}=n^{1/4}h_n\sum_{j=1}^{\lfloor nh_n\rfloor-1}w_{jk}\left(\tilde S_{jk}^{2}-\E\left[\tilde S_{jk}^{2}|\mathcal{G}_{\left(k-1\right)h_n}\right]\right),\, k=1,\dots,h_n^{-1}\,,
\end{eqnarray} 
with $\tilde S_{jk}$ being spectral statistics build from observations of $\tilde X+\epsilon$.
For the proof of the functional stable CLT, we need to verify the following five conditions:
\Links
\begin{subequations}
\begin{align}\label{s1}\tag{J1}\sum_{k=1}^{\lfloor{th_n^{-1}}\rfloor}\E\left[\zeta_k^n\big|\mathcal{G}_{(k-1)h_n}\right]\stackrel{ucp}{\longrightarrow} 0\,.\end{align}
Convergence of the sum of conditional variances
\begin{align}\label{s2}\tag{J2}\sum_{k=1}^{\lfloor{th_n^{-1}}\rfloor}\E\left[\left(\zeta_k^n\right)^2\big|\mathcal{G}_{(k-1)h_n}\right]\stackrel{\P}{\rightarrow} \int_0^t v_s^2\,ds\,,\end{align}
with the predictable process $v_s=\sqrt{8\eta|\sigma_s|^3}$, and a Lyapunov-type condition
\begin{align}\label{s3}\tag{J3}\sum_{k=1}^{\lfloor{th_n^{-1}}\rfloor}\E\left[\left(\zeta_k^n\right)^4\big|\mathcal{G}_{(k-1)h_n}\right]\stackrel{\P}{\rightarrow} 0\,.\end{align}
 Finally, stability of weak convergence is ensured if
\begin{align}\label{s4}\tag{J4}\sum_{k=1}^{\lfloor{th_n^{-1}}\rfloor}\E\left[\zeta_k^n(W_{kh_n}-W_{(k-1)h_n})\big|\mathcal{G}_{(k-1)h_n}\right]\stackrel{\P}{\rightarrow} 0\,,\end{align}
where $W$ is the Brownian motion driving the signal process $X$, and if
\begin{align}\label{s5}\tag{J5}\sum_{k=1}^{\lfloor{th_n^{-1}}\rfloor}\E\left[\zeta_k^n(N_{kh_n}-N_{(k-1)h_n})\big|\mathcal{G}_{(k-1)h_n}\right]\stackrel{\P}{\rightarrow} 0\,,\end{align}
for all bounded $\mathcal{G}_t$-martingales $N$ which are orthogonal to $W$. 
\end{subequations}
Next, we strive for a stable functional CLT for the estimation errors of the covolatility estimator \eqref{specv} and the local method of moments approach \eqref{lmm}. A non-degenerate asymptotic variance is obtained when $n/n_l\rightarrow \nu_l$ with $0<\nu_l<\infty$ as $n\rightarrow\infty$ for all $l=1,\dots,d$. We transform the non-synchronous observation model from Assumption \ref{noise2} to a synchronous observation model and show that the first order asymptotics of the considered estimators remain invariant. Hence, the effect of non-synchronous sampling on the spectral estimators is shown to be asymptotically negligible. In the idealized martingale framework \cite{BHMR} have found that non-synchronicity effects are asymptotically immaterial in terms of the information content of underlying experiments by a (strong) asymptotic equivalence in the sense of Le Cam of the discrete non-synchronous and a continuous-time observation model. This constitutes a fundamental difference to the non-noisy case where the asymptotic variance of the prominent Hayashi-Yoshida estimator in the functional CLT hinges on interpolation effects, see \cite{hy3}. In the presence of the dominant noise part, however, at the slower optimal convergence rate, the influence of sampling schemes boils down to local observation densities. These time-varying local observation densities are shifted to locally time-varying noise levels (indeed locally increased noise is equivalent to locally less frequent observations). Here, we shall explicitly prove that if we pass from a non-synchronous to a synchronous reference scheme the transformation errors of the estimators are asymptotically negligible.
\begin{lem}\label{key}Denote $\bar t_i^{(l)}=\big(t_{i}^{(l)}+t_{i-1}^{(l)}\big)/2, l=1,\ldots,d$. On Assumptions \ref{Hd} and \ref{noise2}, we can work under synchronous sampling when considering the signal part $X$, i.e. we have for $l,m=1,\dots,d$ uniformly in $t$ for both, $w_{jk}^{l,m}$ as in Section \ref{sec:3.3} or defined as entries of \eqref{weightsmv}:
\begin{align*}&\sum_{k=1}^{\lfloor th_n^{-1}\rfloor}h_n\sum_{j\ge 1}w_{jk}^{l,m}\sum_{v=1}^{n_l}\Big(X_{t_v^{(l)}}^{(l)}-X_{t_{v-1}^{(l)}}^{(l)}\Big)\Phi_{jk}(\bar t_v^{(l)})\sum_{i=1}^{n_m}\Big(X_{t_i^{(m)}}^{(m)}-X_{t_{i-1}^{(m)}}^{(m)}\Big)\Phi_{jk}(\bar t_i^{(m)})+\KLEINO_{\P}(n^{-1/4})\\
 & =\sum_{k=1}^{\lfloor th_n^{-1}\rfloor}h_n\sum_{j\ge 1}w_{jk}^{l,m}\sum_{v=1}^{n_l}\Big(X_{t_v^{(l)}}^{(l)}-X_{t_{v-1}^{(l)}}^{(l)}\Big)\Phi_{jk}(\bar t_v^{(l)})\sum_{i=1}^{n_l}\Big(X_{t_i^{(l)}}^{(m)}-X_{t_{i-1}^{(l)}}^{(m)}\Big)\Phi_{jk}(\bar t_i^{(l)})\,.\end{align*}
\end{lem}
Note that $(F_l^{-1})^{\prime}, (F_m^{-1})^{\prime}$ affect the asymptotics of our estimators as can be seen in \eqref{avarspecv}, but are treated as part of the summands due to noise.\\ 
Under a synchronous reference observation scheme the strategy of the asymptotic analysis is similar to the one-dimensional setup. Analogous decompositions in leading terms from the simplified model without drift and with a locally constant covolatility matrix and remainders are considered for the multivariate method of moments estimator \eqref{lmm} and the spectral covolatility estimator \eqref{specv}. In order to prove Theorem \ref{thm:2} for instance, we apply Jacod's limit theorem to the sum of increments
 \Rechts\begin{align}\label{illustrbiv}\zeta_k^n=n^{1/4}h_n\sum_{j\ge 1}\Big( w_{jk}^{p,q}\tilde S_{jk}^{(p)}\tilde S_{jk}^{(q)}-\E\left[\tilde S_{jk}^{(p)}\tilde S_{jk}^{(q)}\Big|\mathcal{G}_{(k-1)h_n}\right]\Big),\end{align}
for $k=1,\ldots,h_n^{-1}$ with $\tilde S_{jk}^{(p)}$ as defined in \eqref{specavg}, but based on observations of $\tilde X+\epsilon$. By including the case $p=q$ with a bias correction the one-dimensional result is generalized to non-equidistant sampling.
\section{Simulations\label{sec:5}}
In the sequel, the one-dimensional spectral integrated volatility estimator's \eqref{spev} finite sample performance is investigated in a random volatility simulation scenario. We sample regular observations $Y_1,\ldots,Y_n$ as in \eqref{observ} with $\epsilon_i\stackrel{iid}{\sim}N(0,\eta^2)$ and the simulated diffusion
\begin{align*}X_t=\int_0^t b\,ds+\int_0^t\sigma_s\,dW_s\,.\end{align*}
In a first baseline scenario configuration we set $\sigma_s=1$ constant. In a second more realistic scenario we consider
\begin{align}\label{sigmasim}\sigma_t^2=\left(\int_0^t\tilde\sigma\cdot\lambda \,dW_s+\int_0^t \sqrt{1-\lambda^2}\cdot\tilde\sigma \,dW_s^{\bot}\right)\cdot f(t)\,,\end{align}
with $W^{\bot}$ a standard Brownian motion independent of $W$ and $f$ a deterministic seasonality function
\begin{align*}f(t)=0.1(1-t^{\frac13}+0.5\cdot t^2)\,,\end{align*}
such that $\sigma_0^2=0.1$. The drift is set $b=0.1$ and $\tilde\sigma=0.01$.
\setlength{\tabcolsep}{1em}
\renewcommand{\arraystretch}{1.5}
\begin{table}[t]
\centering
\begin{tabular}{|ccccc|cc|}
\hline
$n$&$\sigma$&$h_n^{-1}$&$\eta$&$\lambda$&RE($\IV_{n,1}^{or}$)&RE($\IV_{n,1}$)\\
\hline
$30000$&$1$&$25$&$0.01$&--&$1.01$ &$1.43$ \\
$5000$&$1$&$25$&$0.01$&--&$1.02$ &$1.47$ \\
$30000$&Eq.\,\eqref{sigmasim}&$25$&$0.01$&$0.5$& $1.09$& $1.75$\\
$30000$&Eq.\,\eqref{sigmasim}&$25$&$0.01$&$0.2$& $1.06$& $1.77$\\
$30000$&Eq.\,\eqref{sigmasim}&$25$&$0.01$&$0.8$&$1.09$ &$1.75$ \\
$30000$&Eq.\,\eqref{sigmasim}&$25$&$0.001$&$0.5$&$1.62$ & $1.88$\\
$30000$&Eq.\,\eqref{sigmasim}&$25$&$0.1$&$0.5$& $1.20$ & $1.69$\\
$30000$&Eq.\,\eqref{sigmasim}&$50$&$0.01$&$0.5$& $1.09$& $1.84$\\
$30000$&Eq.\,\eqref{sigmasim}&$10$&$0.01$&$0.5$& $1.16$&$1.86$ \\
$5000$&Eq.\,\eqref{sigmasim}&$25$&$0.01$&$0.5$&$1.13$ & $1.92$\\
$5000$&Eq.\,\eqref{sigmasim}&$50$&$0.01$&$0.5$&$1.08$ & $1.75$\\
$5000$&Eq.\,\eqref{sigmasim}&$10$&$0.01$&$0.5$& $1.09$ & $1.87$\\
\hline
\end{tabular}
\caption{\label{tab1}Relative Efficiencies (RE) of oracle and adaptive spectral integrated volatility estimator in finite-sample Monte Carlo study.}
\end{table}\noindent
\hspace*{-.5cm}The superposition of a continuous semimartingale as random component with a time-varying seasonality modeling volatility's typical U-shape mimics very general realistic volatility characteristics.
We implement the oracle version of the estimator \eqref{spev} and the adaptive two-stage procedure with pre-estimated optimal weights. Table \ref{tab1} presents Monte Carlo results for different scenario configurations. In particular, we consider different tuning parameters (bin-widths) and possible dependence of the finite-sample behavior on the leverage magnitude and the magnitude of the noise variance. We compute the estimators' root mean square errors (RMSE) at $t=1$, for each configuration based on 1000 Monte Carlo iterations, and fix in each configuration one realization of a volatility path to compare the RMSEs to the theoretical asymptotic counterparts in the realized relative efficiency (RE):
\begin{align}\text{RE}(\IV_{n,1})=\frac{\sqrt{\left((\text{mean}(\IV_{n,1})-\int_0^1\sigma_s^2\,ds)^2+\var(\IV_{n,1})\right)\cdot \sqrt{n}}}{\sqrt{8\eta\int_0^1\sigma_s^3\,ds}}\,.\end{align}
Our standard sample size is $n=30000$, a realistic number of observations in usual high-frequency applications as number of ticks over one trading day for liquid assets at NASDAQ. We also focus on smaller samples, $n=5000$.\\
Throughout all simulations we fix a maximum spectral cut-off $J=100$ in the pre-estimation step and $J=150$ for the final estimator, which is large enough to render the approximation error by neglecting higher frequencies negligible. In summary, the Monte Carlo study confirms that the estimator performs well in practice and the Monte Carlo variances come very close to the theoretical lower bound, even in the complex wiggly volatility setting. The fully adaptive approach performs less well than the oracle estimator which is in light of previous results on related estimation approaches not surprising, see e.g.\,\cite{bibingerreiss} for a study including an adaptive multi-scale estimator (global smoothing parameter, but chosen data-driven). Still the adaptive estimator's performance is remarkably well in almost all configurations. Under very small noise level, the relative efficiency is not as close to 1 any more. Apart from this case, the RE comes very close to 1 for the oracle estimator, not depending on the magnitude of leverage, also for small samples, and being very robust with respect to different bin-widths.\\
A simulation study of the multivariate method of moments estimator in a random volatility setup can be found in \cite{BHMR}.

\section{Proofs\label{sec:6}}
\subsection{{\bf{Preliminaries\label{sec:6.1}}}}
\begin{itemize}
\item {{\bf{\it{Empirical scalar products:}}}}\\
\begin{defi}\label{scalar}
Let $f,g:\left[0,1\right]\rightarrow\R$ be functions and $z=(z_i)_{1\le i\le n}\in\R^{n}$.
We call the quantities 
\vspace*{-.4cm}
\begin{align*}
\left\langle f,g\right\rangle _{n} & =  \frac{1}{n}\sum_{i=1}^{n}f\left(\frac{i}{n}\right)g\left(\frac{i}{n}\right)\,,\\
\left\langle z,g\right\rangle _{n} & =  \frac{1}{n}\sum_{i=1}^{n}z_{i}\,g\left(\frac{i}{n}\right)\,,
\end{align*}
the \emph{empirical scalar product} of $f$, $g$ and of $z$, $g$,
respectively. We further define the ``shifted'' empirical scalar products
\begin{align*}
\left[f,g\right]_{n} & =  \frac{1}{n}\sum_{i=1}^{n}f\left(\frac{i-\frac{1}{2}}{n}\right)g\left(\frac{i-\frac{1}{2}}{n}\right),\\
\left[z,g\right]_{n} & =  \frac{1}{n}\sum_{i=1}^{n}z_{i}\,g\left(\frac{i-\frac{1}{2}}{n}\right).
\end{align*}
\end{defi}
Recall the notation $\Delta^n Y=\big(\Delta_i^n Y\big)_{1\le i\le n}\in \mathds{R}^n$, the vector of increments and analogously $\Delta^n X$ and let $\epsilon=(\epsilon_{i/n})_{0\le i\le (n-1)}$.\\
To lighten notation, we presume for the following that in the one-dimensional equidistant observations setup we have on each bin $[(k-1)h_n,kh_n],k=1,\ldots,h_n^{-1}$ the same number $\lfloor n h_n\rfloor=nh_n\in\N$ of observations. If $nh_n\notin \N$, either $\lfloor nh_n\rfloor$ or $\lfloor nh_n\rfloor+1$ observations lie in a bin. By cutting-off a small interval at the end, however, we can obtain $h_n^{-1}$ equidistant bins with equal numbers of observations per bin covering $[0,\lfloor nh_n\rfloor/(nh_n)]$. The error by excluding $(\lfloor nh_n\rfloor/(nh_n) ,1]$ is asymptotically negligible. All proofs are easily generalized for arbitrary $n,h_n^{-1}\in \N$ using the exact construction introduced in Section \ref{sec:3}.\footnotemark[3]\footnotetext[3]{In case that $nh_n\notin\N$ and for the original partition in bins and basis functions \eqref{phi}, \eqref{varphi}, the discrete Fourier identities \eqref{or1} and \eqref{or2} hold with a remainder of order $(nh_n)^{-1}$ which is asymptotically negligible. For a given observation scheme it is possible to preserve exact Fourier identities redefining basis functions \eqref{phi}, \eqref{varphi} by bin-wise multiplication with $(\#\{i|i/n \in[(k-1)h_n,kh_n]\}/(nh_n))^{-1/2}$.}
\begin{lem}Suppose that $nh_n\in\N$.\footnotemark[3] It holds that
\begin{subequations}
\begin{align}\label{or1}\left\langle \Phi_{jk},\Phi_{mk}\right\rangle _{n}&=\delta_{jm}\,,\\
\label{or2}\left[\varphi_{jk},\varphi_{mk}\right]_{n}&=\delta_{jm}4n^{2}\sin^{2}\left(\frac{j\pi}{2 nh_n }\right)\,,\\
\left[\varphi_{jk}^{2},\varphi_{mk}^{2}\right]_{n} & =\left(2+\delta_{jm}\right)n^{2}\sin\left(\frac{j\pi}{ nh_{n} }\right)\sin\left(\frac{m\pi}{ nh_{n}}\right)\,.\label{eq:or3}\end{align}
Furthermore, we have the summation by parts decomposition of spectral statistics:
\begin{align}\label{sbp}\left\langle n\triangle^{n}Y,\Phi_{jk}\right\rangle _{n}=\left\langle n\triangle^{n}X,\Phi_{jk}\right\rangle _{n}-\left[\epsilon,\varphi_{jk}\right]_{n}.\end{align}
\end{subequations}
\end{lem}
\begin{proof}
The proofs of the orthogonality relations \eqref{or1} and \eqref{or2} are similar and we restrict ourselves to prove \eqref {or2}. In the following we use the shortcut $N=\lfloor nh_n \rfloor$ and without loss of generality we consider the first bin $k=1$. We make use of the trigonometric addition formulas which yield for $N \ge j\ge r\ge 1$:
\[\cos(j\pi N^{-1}(l+\tfrac12))\cos(r\pi N^{-1}(l+\tfrac12))=\cos((j+r)\pi N^{-1} (l+\tfrac12))+\cos((j-r)\pi N^{-1}(l+\tfrac12))\,.\]
We show that $\sum_{i=0}^{N-1}\cos(m\pi N^{-1}(i+\tfrac12))=0$ for $m\in\mathds{N}$. First, consider $m$ odd:
\begin{align*}\sum_{i=0}^{N-1}\cos(m\pi N^{-1}(i+\tfrac12))&=\sum_{i=0}^{\lfloor (N-2)/2\rfloor}\cos(m\pi N^{-1}(i+\tfrac12))+\sum_{i=\lceil N/2\rceil}^{N-1}\cos(m\pi N^{-1}(i+\tfrac12))\\
&=\sum_{i=0}^{\lfloor (N-2)/2\rfloor}\Big(\cos(m\pi N^{-1}(i+\tfrac12))+\cos(m\pi N^{-1}(N-(i+\tfrac12)))\Big)\\
&=0,
\end{align*}
since $\cos(x+\pi m)=-\cos(x)$ for $m$ odd. Note that for $i=(N-1)/2\in\mathds{N}$, we leave out one addend which equals $\cos(m\pi/2)=0$, and also that for $m$ even by $\cos(x)=\cos(x+m\pi)$ the two sums are equal. Since $\cos(0)=1$, this also implies the empirical norm for $j=r$.\\
For $m\in\mathds{ N}$ with $m$ even, we differentiate the cases $N=4k,k\in\mathds{N};\,N=4k+2,k\in\mathds{N}$ and $N=2k+1,k\in\mathds{N}$. If $N=4k+2$, we decompose the sum as follows:
\[\sum_{i=0}^{N-1}\cos(m\pi N^{-1}(i+\tfrac12))=\sum_{i=0}^{2k} \cos(m\pi (4k+2)^{-1}(i+\tfrac12))+\sum_{i=2k+1}^{4k+1}\cos(m\pi (4k+2)^{-1}(i+\tfrac12))\,.\]
The addends of the left-hand sum are symmetric around the point $m\pi/4$ at $i=k$ and of the right-hand sum around $3m\pi/4$ at $i=3k+1$. Thereby, both sums equal zero by symmetry. More precisely, for $m$ being not a multiple of $4$ the sums directly yield zero. If $m$ is a multiple of 4, we can split the sum into two or more sums which then equal zero again.\\
This observation for the first sum readily implies $\sum_{i=0}^{N-1}\cos(m\pi N^{-1}(i+\tfrac12))=0$ for $N=2k+1$, since in this case
\[\sum_{i=0}^{2k}\cos(m\pi N^{-1}(i+\tfrac12))=\sum_{i=1}^{2k}\cos(2m\pi(4k+2)^{-1}(i+\tfrac12))=0\,.\]
For $N=4k$, we may as well exploit symmetry relations of the cosine. Decompose the sum
\[\sum_{i=0}^{N-1}\cos(m\pi N^{-1}(i+\tfrac12))=\sum_{i=0}^{2k-1}\cos(m\pi (4k)^{-1}(i+\tfrac12))+\sum_{i=2k}^{4k-1}\cos(m\pi (4k)^{-1}(i+\tfrac12))\,.\]
Symmetry around $m\pi/4$ and $3m\pi/4$ is similar as above, but these points lie off the discrete grid this time. Yet, analogous reasoning as above yields that both sums equal zero again, what completes the proof of \eqref{or2}. Likewise and using
\begin{align*}
\cos^{2}\left(x\right)\cos^{2}\left(y\right) & =\frac{1}{4}\left(\cos\left(2x\right)+1\right)\left(\cos\left(2y\right)+1\right)\\
 & =\frac{1}{4}\left(\frac{1}{2}\cos\left(2\left(x+y\right)\right)+\frac{1}{2}\cos\left(2\left(x-y\right)\right)+\cos\left(2x\right)+\cos\left(2y\right)+1\right)\,,
\end{align*}
we deduce relation \eqref{eq:or3}.\\
Finally, we show \eqref{sbp}. Applying summation by parts to $\left\langle n\triangle^{n}\epsilon,\Phi_{jk}\right\rangle _{n}$ and using $\Phi_{jk}(1)=\Phi_{jk}(0)=0$, yields
\begin{eqnarray*}
\left\langle n\triangle^{n}\epsilon,\Phi_{jk}\right\rangle _{n} & = & \sum_{l=1}^{n}\triangle_{l}^{n}\epsilon\,\Phi_{jk}\left(\frac{l}{n}\right)= -\sum_{l=1}^{n}\epsilon_{\frac{l-1}{n}}\left(\Phi_{jk}\left(\frac{l}{n}\right)-\Phi_{jk}\left(\frac{l-1}{n}\right)\right)\,.
\end{eqnarray*}
The equality \(\sin\left(x+h\right)-\sin\left(x\right)=2\sin\left(\frac{h}{2}\right)\cos\left(x+\frac{h}{2}\right)\) for $x,h\in\mathds{R}$ gives
\[
\Phi_{jk}\left(\frac{l}{n}\right)-\Phi_{jk}\left(\frac{l-1}{n}\right)=\frac{1}{n}\varphi_{jk}\left(\frac{l-\frac{1}{2}}{n}\right)
\]
what yields the claim.
\end{proof}

\item {{\bf{\it{Extending local to uniform boundedness:}}}}\\
On the compact time span $[0,1]$, we can strengthen the structural Assumption \ref{H1} and assume $b_s$, $\sigma_s$, $\sigma_s^{-1}$ and the characteristics of $\sigma_s$ in the semimartingale case are uniformly bounded. This is based on the localization procedure given in \cite{jacodlecture}, Lemma 6.\,6 in Section 6.\,3.
\item {{\bf{\it{Basic estimates for drift and Brownian terms: }}}}
For all $p\ge 1$ and $s,(s+t)\in[(k-1)h_n,kh_n]$ for some $k=1,\ldots,h_n^{-1}$:
\begin{subequations}
\begin{align}\label{e1}\E\left[\|\tilde X_{s+t}-\tilde X_s\|^p\big|\mathcal{G}_s\right]\le K_p t^{\frac{p}{2}}\,,\end{align}
with $\tilde X$ from \eqref{tildeX}. The approximation error satisfies
\begin{align}\hspace*{-.255cm}\E\left[\|X_{s+t}-\tilde X_{s+t}-X_s+\tilde X_s\|^p\big|\mathcal{G}_s\right]\hspace*{-.05cm} &\le \hspace*{-.05cm} K_p\, \notag \E\left[\left(\int_s^{s+t}\|\sigma_{\tau}-\sigma_{s}\|^2\,d\tau\right)^{\frac{p}{2}}\Big|\mathcal{G}_s\right]\\
&\label{e2}\quad+K_p\,\E\left[\Big\|\int_s^{s+t}b_u\,du\Big\|^p\Big|\mathcal{G}_s\right]\le K_p t^{p}\,,\end{align}
with generic constant $K_p$ depending on $p$ by It\^{o}-isometry, Cauchy-Schwarz and Burkholder-Davis-Gundy inequalities with Assumption \ref{H1} and Assumption \ref{Hd}, respectively.
\end{subequations}
\item {{\bf{\it{Local quadratic variations of time:}}}}
\begin{align}\notag\hspace*{-.3cm}\sum_{(k-1)h_n\le t_i^{(l)}\le kh_n}\hspace*{-.45cm}\big(t_i^{(l)}-t_{i-1}^{(l)}\big)^2&\asymp \hspace*{-.15cm}\sum_{(k-1)h_n\le t_i^{(l)}\le kh_n}\hspace*{-.35cm}(F_l^{-1})'((k-1)h_n)n_l^{-1}\big(t_i^{(l)}-t_{i-1}^{(l)}\big)\\
&\label{qvt}=(F_l^{-1})'((k-1)h_n)n_l^{-1}h_n.\end{align}
The left-hand side is a localized measure of variation in observation times in the vein of the quadratic variation of time by \cite{zhangmykland}. It appears in the variance of the estimator and is used to estimate $(F_l^{-1})^{\prime}((k-1)h_n)$. Under $F_l'\in C^{\alpha}$ with $\alpha>1/2$ the approximation error is $\KLEINO(n^{-1/4})$. The asymptotic identity applies to deterministic observation times in deterministic manner and to random exogenous sampling in terms of convergence in probability.
\item {{\bf{\it{Order of optimal weights:}}}}\\
Recall the definition of the optimal weights \eqref{weights}. An upper bound for these weights is 
\begin{align}\notag w_{jk}\lesssim I_{jk}=\frac12\Big(\sigma^2_{(k-1)h_n}+\frac{\eta^2}{n}\left[\varphi_{jk},\varphi_{jk}\right]_{n}\Big)^{-2}&\lesssim \Big(1+\frac{j^2}{nh_n^2}\Big)^{-2}\\
&\label{orderweights}\lesssim \begin{cases}1~~~~~~~~~~~~~~\,\mbox{for}\,j\le \sqrt{n}h_n\\ j^{-4}n^2h_n^{4}~~\mbox{for}\,j>\sqrt{n}h_n\end{cases}\end{align}
what also gives
\begin{align}\nonumber \sum_{j=1}^{\lfloor nh_n \rfloor-1}\hspace*{-.05cm}w_{jk}\Big(\sigma^2_{(k-1)h_n}+\frac{\eta^2}{n}\left[\varphi_{jk},\varphi_{jk}\right]_{n}\Big) &\lesssim \hspace*{-.05cm}\sum_{j=1}^{\lfloor \sqrt{n} h_n\rfloor}\hspace*{-.1cm}\Big(1+\frac{j^2}{h_n^2 n}\Big)+\hspace*{-.125cm}\sum_{j=\lceil \sqrt{n}h_n\rceil}^{\lfloor nh_n \rfloor-1}\hspace*{-.125cm}\Big(1+\frac{j^2}{nh_n^2}\Big)j^{-4}n^2h_n^{4}\\
& \label{helpweights}\lesssim \sqrt{n}h_n+nh_n^2\,.\end{align}
\end{itemize}

\subsection{{\bf{Proof of Proposition \ref{propmain}}}}
\noindent
Recall the definition of spectral statistics \eqref{spec} and denote for $j=1,\dots,\lfloor nh_n\rfloor-1,k=1,\ldots,h_n^{-1}$:
\[\tilde{S}_{jk}=\left\langle n(\triangle^{n}\tilde{X}+\triangle^{n}\epsilon),\Phi_{jk}\right\rangle _{n}=\left\langle n\Delta^{n}\tilde X,\Phi_{jk}\right\rangle _{n}-\left[\epsilon,\varphi_{jk}\right]_{n}\,,\]
where $\tilde X$ is the signal process in the locally parametric experiment.
It holds that
\begin{eqnarray}
\E\left[\l{\tilde{S}_{jk}^{2}}\c G_{\left(k-1\right)h_n}\right] & = & \E\left[\l{\left(\left\langle n\triangle^{n}\tilde{X},\Phi_{jk}\right\rangle _{n}-\left[\epsilon,\varphi_{jk}\right]_{n}\right)^{2}}\c G_{\left(k-1\right)h_n}\right]\nonumber \\
 & = & \E\left[\l{\left\langle n\triangle^{n}\tilde{X},\Phi_{jk}\right\rangle _{n}^{2}-2\left\langle n\triangle^{n}\tilde{X},\Phi_{jk}\right\rangle _{n}\left[\epsilon,\varphi_{jk}\right]_{n}+\left[\epsilon,\varphi_{jk}\right]_{n}^{2}}\c G_{\left(k-1\right)h_n}\right]\nonumber \\
 & \overset{}{=} & \E\left[\l{\left\langle n\triangle^{n}\tilde{X},\Phi_{jk}\right\rangle _{n}^{2}}\c G_{\left(k-1\right)h_n}\right]+\E\left[\l{\left[\epsilon,\varphi_{jk}\right]_{n}^{2}}\c G_{\left(k-1\right)h_n}\right]\nonumber \\
 & \overset{}{=} & \sigma_{\left(k-1\right)h_n}^{2}+\frac{\eta^{2}}{n}\left[\varphi_{jk},\varphi_{jk}\right]_{n}\,.\label{drift}
\end{eqnarray}
We have defined $\zeta_k^n$ above such that
\begin{eqnarray*}
n^{\frac{1}{4}}\hspace*{-.025cm}\Big(\widetilde{\operatorname{IV}}_{n,t}-h_n\hspace*{-.075cm}\sum_{k=1}^{\lfloor t h_n^{-1}\rfloor }\hspace*{-.1cm}\sigma_{\left(k-1\right)h_n}^{2}\Big)\hspace*{-.05cm}=\hspace*{-.05cm}n^{\frac{1}{4}}h_n\hspace*{-.175cm}\sum_{k=1}^{\f{th_n^{-1}}}\sum_{j=1}^{\lfloor nh_n\rfloor-1}\hspace*{-.2cm}w_{jk}\hspace*{-.025cm}\Big(\tilde{S}_{jk}^{2}\hspace*{-.05cm}-\hspace*{-.05cm}\frac{\eta^{2}}{n}\left[\varphi_{jk},\varphi_{jk}\right]_{n}\hspace*{-.15cm}-\hspace*{-.05cm}\sigma_{\left(k-1\right)h_n}^{2}\Big)\hspace*{-.1cm}=\hspace*{-.15cm}\sum_{k=1}^{\f{th_n^{-1}}}\hspace*{-.2cm}\zeta_{k}^{n}
\end{eqnarray*}
when we shortly express $\widetilde{\operatorname{IV}}_{n,t}=\IV_{n,t}^{or}(\tilde X+\epsilon)$. We have to verify \eqref{s1}-\eqref{s5}. \eqref{s1} is trivial as the $\zeta_{k}^{n}$ are centered
conditional on $\c G_{(k-1)h_n}$. The proof of \eqref{s2} is done in two steps. In paragraph \ref{J2_var} we calculate explicitly the
variance which is the left-hand side of \eqref{s2}. For this we consider at first general
weights $w_{jk}\geq0$, $\sum_{j=1}^{\lfloor nh_{n}\rfloor -1}w_{jk}=1$ which satisfy $w_{jk}\in\c G_{\left(k-1\right)h_n}$ for all $k=1,\dots,h_{n}^{-1},\, j=1,\dots,\lfloor nh_{n}\rfloor -1$.
After that we find optimal weights minimizing the variance. In paragraph \ref{J2_avar} we let $n\rightarrow\infty$ and calculate
the resulting limiting asymptotic variance.
The proofs of \eqref{s3}, \eqref{s4} and \eqref{s5} follow in paragraph \ref{J_3}. 

\subsubsection{\bf{Computation of the variance}\label{J2_var}}
\vspace*{-.1cm}
\begin{eqnarray*}
\E\left[\l{\left(\zeta_{k}^{n}\right)^{2}}\c G_{\left(k-1\right)h_n}\right] & = & n^{\frac{1}{2}}h_n^{2}\sum_{j,m=1}^{\lfloor nh_n\rfloor-1}w_{jk}w_{mk}\,\E\bigg[\left(\tilde{S}_{jk}^{2}-\E\left[\l{\tilde{S}_{jk}^{2}}\c G_{\left(k-1\right)h_n}\right]\right)\\
 &  & \l{\hspace{9em}\cdot\left(\tilde{S}_{mk}^{2}-\E\left[\l{\tilde{S}_{mk}^{2}}\c G_{\left(k-1\right)h_n}\right]\right)}\c G_{\left(k-1\right)h_n}\bigg]\\
 &=& n^{\frac{1}{2}}h_n^{2}\sum_{j,m=1}^{\lfloor nh_n\rfloor-1}w_{jk}w_{mk}\big(T_{j,m,k}^n(1)+T_{j,m,k}^n(2)+T_{j,m,k}^n(3)\big)\,,
\end{eqnarray*}
with the following three addends:
 \begin{align*}
T_{j,m,k}^{n}\left(1\right) & =\E\left[\l{\left(\left\langle n\triangle^{n}\tilde{X},\Phi_{jk}\right\rangle _{n}^{2}-\sigma_{\left(k-1\right)h_{n}}^{2}\right)\left(\left\langle n\triangle^{n}\tilde{X},\Phi_{mk}\right\rangle _{n}^{2}-\sigma_{\left(k-1\right)h_{n}}^{2}\right)}\c G_{\left(k-1\right)h_{n}}\right],\\
T_{j,m,k}^{n}\left(2\right) & =\E\left[\l{4\left\langle n\triangle^{n}\tilde{X},\Phi_{jk}\right\rangle _{n}\left[\epsilon,\varphi_{jk}\right]_{n}\left\langle n\triangle^{n}\tilde{X},\Phi_{mk}\right\rangle _{n}\left[\epsilon,\varphi_{mk}\right]_{n}}\c G_{\left(k-1\right)h_{n}}\right],\\
T_{j,m,k}^{n}\left(3\right) & =\E\left[\l{\left(\left[\epsilon,\varphi_{jk}\right]_{n}^{2}-\frac{\eta^{2}}{n}\left[\varphi_{jk},\varphi_{jk}\right]_{n}\right)\left(\left[\epsilon,\varphi_{mk}\right]_{n}^{2}-\frac{\eta^{2}}{n}\left[\varphi_{mk},\varphi_{mk}\right]_{n}\right)}\c G_{\left(k-1\right)h_{n}}\right],
\end{align*}
for frequencies $j,m$. The i.i.d.\,structure of the noise and of Brownian increments yields
\begin{align*}
\E\left[\left[\epsilon,\varphi_{jk}\right]_{n}\left[\epsilon,\varphi_{mk}\right]_{n}\right] & =\frac{\eta^{2}}{n}\left[\varphi_{jk},\varphi_{mk}\right]_{n},\\
\E\left[\l{\left\langle n\triangle^{n}\tilde{X},\Phi_{jk}\right\rangle _{n}\left\langle n\triangle^{n}\tilde{X},\Phi_{mk}\right\rangle _{n}}\c G_{\left(k-1\right)h_{n}}\right] & =\delta_{jm}\sigma_{\left(k-1\right)h_{n}}^{2},
\end{align*}
which implies with independence of the noise and $X$ that 
\[
T_{j,m,k}^{n}\left(2\right)=4\frac{\eta^{2}}{n}\delta_{jm}\left[\varphi_{jk},\varphi_{mk}\right]_{n}\sigma_{\left(k-1\right)h_n}^{2}\,.
\]
We further obtain by another polynomial expansion
\begin{align*}
 & \E\left[\left[\epsilon,\varphi_{jk}\right]_{n}^{2}\left[\epsilon,\varphi_{mk}\right]_{n}^{2}\right]\\\
 & =n^{-4}\hspace*{-0.1cm}\sum_{l,l',p,p'=1}^{n}\hspace*{-0.1cm}\Big(\E\left[\epsilon_{l}\epsilon_{l'}\epsilon_{p}\epsilon_{p'}\right]\varphi_{jk}\Big(\frac{l-\frac{1}{2}}{n}\Big)\varphi_{jk}\Big(\frac{l'-\frac{1}{2}}{n}\Big)\varphi_{mk}\Big(\frac{p-\frac{1}{2}}{n}\Big)\varphi_{mk}\Big(\frac{p'-\frac{1}{2}}{n}\Big)\Big).\nonumber 
\end{align*}
Only the cases $l=l'\neq p=p'$, $l=p\neq l'=p'$ , $l=p'\neq l'=p$
or $l=l'=p=p'$ produce non-zero results in the expectation. Hence, denoting by $\eta'=\E[\epsilon_{t}^{4}]$ the fourth moment
of the observation errors, we end up with 
\begin{align*}
 & \E\left[\left[\epsilon,\varphi_{jk}\right]_{n}^{2}\left[\epsilon,\varphi_{mk}\right]_{n}^{2}\right] \hspace*{-0.05cm}=\hspace*{-0.05cm}\frac{1}{n^{4}}\sum_{l,l',p,p'}\hspace*{-0.1cm}\left(\eta^{4}\left(\delta_{ll'}\delta_{pp'}+\delta_{lp}\delta_{l'p'}+\delta_{lp'}\delta_{l'p}\right)+\eta'\delta_{lp}\delta_{l'p'}\delta_{ll'}-3\eta^{4}\delta_{lp}\delta_{l'p'}\delta_{ll'}\right)\\
 &\hspace*{5cm} \cdot\Big(\varphi_{jk}\Big(\frac{l-\frac{1}{2}}{n}\Big)\varphi_{jk}\Big(\frac{l'-\frac{1}{2}}{n}\Big)\varphi_{mk}\Big(\frac{p-\frac{1}{2}}{n}\Big)\varphi_{mk}\Big(\frac{p'-\frac{1}{2}}{n}\Big)\Big)\\
 & =\frac{\eta^{4}}{n^{2}}\left(\left[\varphi_{jk},\varphi_{jk}\right]_{n}\left[\varphi_{mk},\varphi_{mk}\right]_{n}+2\left[\varphi_{jk},\varphi_{mk}\right]_{n}^{2}\right) +\frac{\eta'-3\eta^{4}}{n^{4}}\sum_{l=1}^{n}\Big(\varphi_{jk}^{2}\Big(\frac{l-\frac{1}{2}}{n}\Big)\varphi_{mk}^{2}\Big(\frac{l-\frac{1}{2}}{n}\Big)\Big).
\end{align*}
Arguing similarly and using that $\E[(\triangle_{l}^{n}W)^{4}]=3\,\E[(\triangle_{l}^{n}W)^{2}]$
for $l\in\N$, we obtain
\begin{align*}
 & \E\left[\l{\left\langle n\triangle^{n}\tilde{X},\Phi_{jk}\right\rangle _{n}^{2}\left\langle n\triangle^{n}\tilde{X},\Phi_{mk}\right\rangle _{n}^{2}}\c G_{\left(k-1\right)h_{n}}\right]\\
 & =\sigma_{\left(k-1\right)h_{n}}^{4}\left(\langle\Phi_{jk},\Phi_{jk}\rangle_{n}\langle\Phi_{mk},\Phi_{mk}\rangle_{n}+2\langle\Phi_{jk},\Phi_{mk}\rangle_{n}^{2}\right)=\sigma_{\left(k-1\right)h_{n}}^{4}\left(1+2\delta_{jm}\right)\,.
\end{align*}
From the identities so far we obtain
\begin{align*}
T_{j,m,k}^{n}\left(1\right) & =\E\left[\l{\left\langle n\triangle^{n}\tilde{X},\Phi_{jk}\right\rangle _{n}^{2}\left\langle n\triangle^{n}\tilde{X},\Phi_{mk}\right\rangle _{n}^{2}}\c G_{\left(k-1\right)h_{n}}\right]\\
 & -\E\left[\l{\left\langle n\triangle^{n}\tilde{X},\Phi_{jk}\right\rangle _{n}^{2}}\c G_{\left(k-1\right)h_{n}}\right]\E\left[\l{\left\langle n\triangle^{n}\tilde{X},\Phi_{mk}\right\rangle _{n}^{2}}\c G_{\left(k-1\right)h_{n}}\right]\\
 & =\sigma_{\left(k-1\right)h_{n}}^{4}\left(1+2\delta_{mj}\right)-\sigma_{\left(k-1\right)h_{n}}^{4}=2\delta_{jm}\sigma_{\left(k-1\right)h_n}^{4},\\
T_{j,m,k}^{n}\left(3\right) & =\E\left[\left(\left[\epsilon,\varphi_{jk}\right]_{n}^{2}-\frac{\eta^{2}}{n}\left[\varphi_{jk},\varphi_{jk}\right]_{n}\right)\left(\left[\epsilon,\varphi_{mk}\right]_{n}^{2}-\frac{\eta^{2}}{n}\left[\varphi_{mk},\varphi_{mk}\right]_{n}\right)\right]\\
 & =\E\left[\left[\epsilon,\varphi_{jk}\right]_{n}^{2}\left[\epsilon,\varphi_{mk}\right]_{n}^{2}\right]-\frac{\eta^{4}}{n^{2}}\left[\varphi_{jk},\varphi_{jk}\right]_{n}\left[\varphi_{mk},\varphi_{mk}\right]_{n}\\
 & =\frac{2\eta^{4}}{n^{2}}\left[\varphi_{jk},\varphi_{mk}\right]_{n}^{2}+\frac{\eta'-3\eta^{4}}{n^{3}}\left[\varphi_{jk}^{2},\varphi_{mk}^{2}\right]_{n}\,.
\end{align*}
In all, the conditional variance is given by 
\begin{align*}
 & \E\left[\l{\left(\zeta_{k}^{n}\right)^{2}}\c G_{\left(k-1\right)h_{n}}\right]\\
 & =\sqrt{n}h_n^{2}\left(\sum_{j=1}^{\lfloor nh_{n}\rfloor -1}w_{jk}^{2}\left(2\sigma_{\left(k-1\right)h_{n}}^{4}+4\frac{\eta^{2}}{n}\sigma_{\left(k-1\right)h_{n}}^{2}\left[\varphi_{jk},\varphi_{jk}\right]_{n}+\frac{2\eta^{4}}{n^{2}}\left[\varphi_{jk},\varphi_{jk}\right]_{n}^{2}\right)\right)+R_{n}\\
 & =\sqrt{n}h_{n}^{2}\sum_{j=1}^{\lfloor nh_{n}\rfloor -1}w_{jk}^{2}\,2\left(\sigma_{\left(k-1\right)h_{n}}^{2}+\frac{\eta^{2}}{n}\left[\varphi_{jk},\varphi_{jk}\right]_{n}\right)^{2}+R_{n}
\end{align*}
with remainder $R_{n}.$
Observe that $R_{n}=0$ for Gaussian noise. In this case, analogous to \cite{bibingerreiss}, we find that the optimal weights minimizing the variance,
under the constraint $\sum_{j=1}^{\lfloor nh_{n}\rfloor -1}w_{jk}=1$, which ensures
unbiasedness of the estimator, are given by \eqref{weights}. The
optimization can be done with Lagrange multipliers. $R_n$ is then a remainder in case that $\eta'\ne3\eta^{4}$. With the weights
\eqref{weights}, Cauchy-Schwarz and using \eqref{eq:or3} and \eqref{helpweights}, we can bound $R_{n}$
by:
\[
R_{n}\lesssim\frac{\sqrt{n}h_{n}^{2}}{n^{3}}\left(\sum_{j=1}^{\lfloor nh_{n}\rfloor-1}w_{jk}n\left|\sin\left(\frac{j\pi}{nh_{n}}\right)\right|\right)^{2}\leq\frac{h_{n}^{2}}{n}\,.
\]
We therefore obtain
\[
\sum_{k=1}^{\lfloor{th_{n}^{-1}}\rfloor}\E\left[\l{\left(\zeta_{k}^{n}\right)^{2}}\c G_{\left(k-1\right)h_{n}}\right]=\sqrt{n}h_{n}^{2}\sum_{k=1}^{\lfloor{th_{n}^{-1}}\rfloor}\sum_{j=1}^{\lfloor nh_{n}\rfloor -1}(I_{k}^{-2}I_{jk}^{2})I_{jk}^{-1}+\KLEINO(1)=\sqrt{n}h_{n}^{2}\sum_{k=1}^{\lfloor{th_{n}^{-1}}\rfloor}I_{k}^{-1}+\KLEINO(1)
\]
as variance of the estimator.
\subsubsection{\bf{The asymptotic variance of the estimator}\label{J2_avar}}
The key to the asymptotic variance is to recognize 
\[(\sqrt{n}h_{n})^{-1}I_{k}=\frac{1}{\sqrt{n}h_{n}}\sum_{j=1}^{\lfloor nh_{n}\rfloor -1}\frac{1}{2}\left(\sigma_{(k-1)h_{n}}^{2}+\frac{\eta^{2}}{n}\left[\varphi_{jk},\varphi_{jk}\right]_{n}\right)^{-2}\]
as a Riemann sum, ending up with the ``double-Riemann-sum'' $\sum_{k=1}^{\f{th_{n}^{-1}}}h_{n}((\sqrt{n}h_{n})^{-1}I_{k})^{-1}$.
The scaling factor $(\sqrt{n}h_{n})^{-1}$ is the right choice for
the first Riemann sum which becomes clear after two Taylor expansions.
First, expanding the sine for each frequency $j$ we find $0\le\xi_{j}\le j\pi/(2nh_{n})$
with 
\[I_{jk}=\frac{1}{2}\left(\sigma_{(k-1)h_{n}}^{2}+4\eta^{2}n\left(\frac{j\pi}{2nh_{n}}-\frac{\xi_{j}^{3}}{6}\right)^{2}\right)^{-2}.\]
Second, we expand $x\mapsto\frac{1}{2}\left(\sigma_{(k-1)h_{n}}^{2}+4\eta^{2}nx^{2}\right)^{-2}$
which yields $\frac{j\pi}{2nh_{n}}-\frac{\xi_{j}^{3}}{6}\leq\xi_{j}'\leq\frac{j\pi}{2nh_{n}}$
such that 
\begin{align}
I_{jk}=\tilde{I}_{jk}+R_{jk}~~~\mbox{with}~~~R_{jk}=\frac{4\eta^{2}n\xi_{j}'}{(\sigma_{(k-1)h_{n}}^{2}+4\eta^{2}n\xi_{j}'^{2})^{3}}\frac{\xi_{j}^{3}}{6}\,
\end{align}
where we define $\tilde{I}_{jk}=\frac{1}{2}(\sigma_{\left(k-1\right)h_{n}}^{2}+\eta^{2}(\frac{j\pi}{\sqrt{n}h_{n}})^{2})^{-2}$.
Now it becomes clear that $\sqrt{n}h_{n}$ is indeed the right factor because 
\begin{align*}
 & \left|\frac{1}{\sqrt{n}h_{n}}\sum_{j=1}^{nh_{n}-1}\tilde{I}_{jk}-\int_{0}^{\sqrt{n}-\frac{1}{\sqrt{n}h_{n}}}\frac{1}{2}\left(\sigma_{(k-1)h_{n}}^{2}+\eta^{2}\pi^{2}x^{2}\right)^{-2}\, dx\right|\\
 & =\left|\sum_{j=1}^{nh_{n}-1}\int_{\frac{j-1}{\sqrt{n}h_{n}}}^{\frac{j}{\sqrt{n}h_{n}}}\Big(\frac{1}{2}\big(\sigma_{(k-1)h_{n}}^{2}+\eta^{2}\pi^{2}j^{2}h_{n}^{-2}n^{-1}\big)^{-2}-\frac{1}{2}\big(\sigma_{(k-1)h_{n}}^{2}+\eta^{2}\pi^{2}x^{2}\big)^{-2}\Big)\, dx\right|\\
 & \lesssim\hspace{1em}\sum_{j=1}^{nh_{n}-1}\int_{\frac{j-1}{\sqrt{n}h_{n}}}^{\frac{j}{\sqrt{n}h_{n}}}\left|x-\frac{j}{\sqrt{n}h_{n}}\right|\, dx\,\max_{\frac{j-1}{\sqrt{n}h_{n}}\leq y\leq\frac{j}{\sqrt{n}h_{n}}}\big(y\big(\sigma_{(k-1)h_{n}}^{2}+\eta^{2}\pi^{2}y^{2}\big)^{-3}\big)\\
 & \leq\hspace{1em}\left(\frac{1}{\sqrt{n}h_{n}}\right)^{2}\left(\sum_{j=1}^{nh_{n}-1}\Big(\max_{\frac{j-1}{\sqrt{n}h_{n}}\leq y\leq\frac{j}{\sqrt{n}h_{n}}}\big(y\big(\sigma_{(k-1)h_{n}}^{2}+\eta^{2}\pi^{2}y^{2}\big)^{-3}\big)\Big)\right)\\
 & =\hspace{1em}\left(\frac{1}{\sqrt{n}h_{n}}\right)^{2}\left(\sum_{j=1}^{\f{\sqrt{n}h_{n}}}\frac{j}{\sqrt{n}h_{n}}+\sum_{j=\lceil\sqrt{n}h_{n}\rceil}^{nh_{n}-1}\left(\frac{\sqrt{n}h_{n}}{j-1}\right)^{5}\right)\\
 & \lesssim\hspace{1em}\left(\frac{1}{\sqrt{n}h_{n}}\right)^{2}\left(\sqrt{n}h_{n}+\sum_{j=1}^{nh_{n}-1-\lceil\sqrt{n}h_{n}\rceil}\left(\frac{\sqrt{n}h_{n}}{j+\lceil\sqrt{n}h_{n}\rceil}\right)^{5}\right)\lesssim \frac{1}{\sqrt{n}h_{n}}\,.
\end{align*}
We choose $h_{n}$ such that $\sqrt{n}h_{n}\rightarrow\infty$. 
Though we consider all possible spectral frequencies $j=1,\ldots,\lfloor nh_{n}\rfloor -1$,
we shall see in the following that the $I_{jk}$ for $j\ge\lceil n^{\beta}h_{n}\rceil$
become asymptotically negligible for a suitable $0\le\beta<1$. By
virtue of monotonicity of the sine on $\left[0,\frac{\pi}{2}\right]$
and $\sin(x)\geq\ x/2$ for $0\leq x\leq1$, it follows that 
\begin{eqnarray*}
\frac{1}{\sqrt{n}h_{n}}\sum_{j=\lceil n^{\beta}h_{n}\rceil}^{\lfloor nh_{n}\rfloor -1}I_{jk} & \lesssim & \frac{1}{\sqrt{n}h_{n}}\sum_{j=\lceil n^{\beta}h_{n}\rceil}^{\lfloor nh_{n}\rfloor -1}\left(n\sin^{2}\left(\frac{n^{\beta}h_{n}\pi}{2nh_{n}}\right)\right)^{-2}\\
 & \leq & \frac{1}{\sqrt{n}h_{n}}nh_{n}\left(n\sin^{2}\left(\frac{n^{\beta}h_{n}\pi}{2nh_{n}}\right)\right)^{-2}\\
 & \leq & \sqrt{n}\left(n\left(\frac{n^{\beta-1}\pi}{4}\right)^{2}\right)^{-2}\lesssim  n^{\frac{1}{2}-4\beta+2}=n^{\frac{5}{2}-4\beta}\,.
\end{eqnarray*}
We deduce that $\frac{1}{\sqrt{n}h_{n}}\sum_{j=\lceil n^{\beta}h_{n}\rceil}^{\lfloor nh_{n}\rfloor -1}I_{jk}=\KLEINO\left(1\right)$,
for every $5/8<\beta<1$. Moreover, we obtain for the first $\f{n^{\beta}h_{n}}$
summands of the remainder term 
\begin{eqnarray*}
\frac{1}{\sqrt{n}h_{n}}\sum_{j=1}^{\f{n^{\beta}h_{n}}}R_{jk} & = & \frac{1}{\sqrt{n}h_{n}}\sum_{j=1}^{\f{n^{\beta}h_{n}}}\frac{4\eta^{2}n\xi_{j}'}{\left(\sigma_{(k-1)h_{n}}^{2}+4\eta^{2}n\xi_{j}'^{2}\right)^{3}}\frac{\xi_{j}^{3}}{6}\lesssim  \frac{n}{\sqrt{n}h_{n}}\sum_{j=1}^{\f{n^{\beta}h_{n}}}\left(\xi_{j}^{3}\xi_{j}'\right)\\
 & \leq & \frac{n}{\sqrt{n}h_{n}}\sum_{j=1}^{\f{n^{\beta}h_{n}}}\left(\frac{j\pi}{nh_{n}}\right)^{4} \lesssim \frac{1}{\sqrt{n}h_{n}}n^{\beta}h_{n}n^{4\left(\beta-1\right)+1}=n^{5\beta-\frac{7}{2}}\,.
\end{eqnarray*}
Hence $\frac{1}{\sqrt{n}h_{n}}\sum_{j=1}^{\f{n^{\beta}h_{n}}}R_{jk}=\KLEINO\left(1\right)$
for every $\beta<7/10$. As the tails are asymptotic negligible we
thus have $\frac{1}{\sqrt{n}h_{n}}\sum_{j=1}^{\lfloor nh_{n}\rfloor -1}R_{jk}=\KLEINO\left(1\right)$
and, in particular, 
\[
\frac{1}{\sqrt{n}h_{n}}\sum_{j=1}^{\lfloor nh_{n}\rfloor -1}I_{jk}=\int_{0}^{\sqrt{n}-\frac{1}{\sqrt{n}h_{n}}}\frac{1}{2}\left(\sigma_{(k-1)h_{n}}^{2}+\eta^{2}\pi^{2}x^{2}\right)^{-2}\, dx+\KLEINO\left(1\right).
\]
Computing the integral expression yields
\begin{align*}
 & \int_{0}^{y}\frac{1}{2}\left(\sigma_{(k-1)h_{n}}^{2}+\eta^{2}\pi^{2}x^{2}\right)^{-2}\, dx =\\
 & \frac{y}{4\left|\sigma_{(k-1)h_{n}}\right|^{4}\left(1+\left(\frac{\eta\pi}{\left|\sigma_{(k-1)h_{n}}\right|}y\right)^{2}\right)}+\frac{1}{4\eta\pi\left|\sigma_{(k-1)h_{n}}\right|^{3}}\arctan\left(\frac{\eta\pi}{\left|\sigma_{(k-1)h_{n}}\right|}y\right).
\end{align*}
As $c<\left|\sigma_{s}\right|<C$ uniformly for all $0\leq s\leq1$ with constants $c,C$ and because $\arctan\left(x\right)\rightarrow\pi/2$ as $x\rightarrow\infty$,
as well as $\sqrt{n}-\frac{1}{\sqrt{n}h_{n}}\rightarrow\infty$ as
$n\rightarrow\infty$, we have 
\begin{align*}
\frac{1}{\sqrt{n}h_{n}}\sum_{j=1}^{\lfloor nh_{n}\rfloor -1}I_{jk} & =\frac{\sqrt{n}-\frac{1}{\sqrt{n}h_{n}}}{4\left|\sigma_{(k-1)h_{n}}\right|^{4}\left(1+\left(\frac{\eta\pi}{\left|\sigma_{(k-1)h_{n}}\right|}\left(\sqrt{n}-\frac{1}{\sqrt{n}h_{n}}\right)\right)^{2}\right)}\\
 & +\frac{1}{4\eta\pi\left|\sigma_{(k-1)h_{n}}\right|^{3}}\arctan\left(\frac{\eta\pi}{\left|\sigma_{(k-1)h_{n}}\right|}\left(\sqrt{n}-\frac{1}{\sqrt{n}h_{n}}\right)\right)+\KLEINO\left(1\right)\\
= & \frac{1}{8\eta\left|\sigma_{\left(k-1\right)h_{n}}\right|^{3}}+\KLEINO\left(1\right).
\end{align*}
The final step in the proof is another Taylor approximation:
\begin{align}
\sum_{k=1}^{\lfloor{th_{n}^{-1}}\rfloor}\E\left[\l{\left(\zeta_{k}^{n}\right)^{2}}\c G_{\left(k-1\right)h_{n}}\right] & =\sqrt{n}h_{n}^{2}\sum_{k=1}^{\lfloor{th_{n}^{-1}}\rfloor}I_{k}^{-1}+\KLEINO(1)=h_{n}\sum_{k=1}^{\lfloor{th_{n}^{-1}}\rfloor}\left(\frac{1}{\sqrt{n}h_{n}}\sum_{j=1}^{\lfloor nh_{n}\rfloor -1}I_{jk}\right)^{-1}+\KLEINO(1)\nonumber \\
 & =h_{n}\sum_{k=1}^{\lfloor{th_{n}^{-1}}\rfloor}\left(\frac{1}{8\eta\left|\sigma_{\left(k-1\right)h_{n}}\right|^{3}}+\KLEINO\left(1\right)\right)^{-1}+\KLEINO(1)\nonumber \\
 & =\left(h_{n}\sum_{k=1}^{\lfloor{th_{n}^{-1}}\rfloor}8\eta\left|\sigma_{\left(k-1\right)h_{n}}\right|^{3}\right)+\KLEINO\left(1\right).\label{eq:-2}
\end{align}
The last equality is true by Taylor and because $\sigma$ is uniformly
bounded. Because $\sigma$ is continuous we obtain the claim by Riemann
approximation, i.e. 
\begin{align*}
\sum_{k=1}^{\f{th_{n}^{-1}}}\E\left[\l{\left(\zeta_{k}^{n}\right)^{2}}\c G_{\left(k-1\right)h_{n}}\right]\rightarrow8\eta\int_{0}^{t}\left|\sigma_{s}\right|^{3}\, ds
\end{align*}
almost surely as $n\rightarrow\infty$ establishing \eqref{s2} with the asymptotic expression of Theorem \ref{thm:1}.

\subsubsection{\bf{Lyapunov's criterion and stability of convergence}\label{J_3}}
So far, we have proved \eqref{s1} and \eqref{s2}. Next, we shall
prove that the Lyapunov condition \eqref{s3} is satisfied. For the
sum of fourth moments, we obtain by Minkowski's inequality, Jensen's
inequality and $w_{jk}\in\c G_{\left(k-1\right)h_{n}}$ for all $k=1,\dots,h_{n}^{-1}$
and $j=1,\dots,\lfloor nh_{n}\rfloor-1$: 
\begin{eqnarray*}
\E\left[\left.\left(\zeta_{k}^{n}\right)^{4}\right|\c G_{\left(k-1\right)h_{n}}\right] & = & ( nh_{n})^{4}\left(\E\left[\l{\left(\sum_{j=1}^{\lfloor nh_{n}\rfloor-1}w_{jk}\left(\tilde{S}_{jk}^{2}-\E\left[\l{\tilde{S}_{jk}^{2}}\c G_{\left(k-1\right)h_{n}}\right]\right)\right)^{4}}\c G_{\left(k-1\right)h_{n}}\right]\right)\\
 & \leq & nh_{n}^{4}\left(\sum_{j=1}^{\lfloor nh_{n}\rfloor-1}w_{jk}\left(\E\left[\l{\left(\tilde{S}_{jk}^{2}-\E\left[\l{\tilde{S}_{jk}^{2}}\c G_{\left(k-1\right)h_{n}}\right]\right)^{4}}\c G_{\left(k-1\right)h_{n}}\right]\right)^{\frac{1}{4}}\right)^{4}\\
 & \lesssim & nh_{n}^{4}\left(\sum_{j=1}^{\lfloor nh_{n}\rfloor -1}w_{jk}\left(\E\left[\l{\tilde{S}_{jk}^{8}}\c G_{\left(k-1\right)h_{n}}\right]\right)^{\frac{1}{4}}\right)^{4}\,.
\end{eqnarray*}
If we can show
\begin{align}
\E\left[\l{\left\langle n\triangle^{n}\tilde{X},\Phi_{jk}\right\rangle _{n}^{8}}\c
G_{\left(k-1\right)h_{n}}\right] &
\lesssim\sigma_{\left(k-1\right)h_{n}}^{8},\label{norm_est}\\
\E\left[\left[\epsilon,\varphi_{jk}\right]_{n}^{8}\right] &
\lesssim\left(\eta^{2}\right)^{4}\frac{\left[\varphi_{jk},\varphi_{jk}\right]_{n}^{4}}{n^{4}},\label{noise_est}
\end{align}
then we are able to conclude that
\begin{eqnarray}
\E\left[\l{\tilde{S}_{jk}^{8}}\c G_{\left(k-1\right)h_{n}}\right] & \lesssim &
\E\left[\l{\left\langle n\Delta^{n}\tilde{X},\Phi_{jk}\right\rangle _{n}^{8}}\c
G_{\left(k-1\right)h_{n}}\right]+\E\left[\left[\epsilon,\varphi_{jk}\right]_{n}^{8}\right]\nonumber
\\
& \lesssim &
\sigma_{\left(k-1\right)h_{n}}^{8}+\left(\eta^{2}\right)^{4}\frac{\left[\varphi_{jk},\varphi_{jk}\right]_{n}^{4}}{n^{4}}\lesssim\left(\sigma_{\left(k-1\right)h_{n}}^{2}+\eta^{2}\frac{\left[\varphi_{jk},\varphi_{jk}\right]_{n}}{n}\right)^{4}.\label{stat_moments}
\end{eqnarray}
Hence, we obtain from \eqref{helpweights}
\begin{align*}
\sum_{k=1}^{h_{n}^{-1}}\E\left[\left.\left(\zeta_{k}^{n}\right)^{4}\right|\c
G_{\left(k-1\right)h_{n}}\right]\lesssim\sum_{k=1}^{h_{n}^{-1}}nh_{n}^{4}\left(\sum_{j=1}^{\lfloor nh_{n}\rfloor -1}\hspace*{-.1cm}w_{jk}\left(\sigma_{\left(k-1\right)h_{n}}^{2}+\frac{\eta^{2}}{n}\left[\varphi_{jk},\varphi_{jk}\right]_{n}\right)\right)^{4}\hspace*{-.1cm}\lesssim
n^{2}h_{n}^{6}=\KLEINO\left(1\right)
\end{align*}
which proves \eqref{s3}. We are therefore left with proving \eqref{norm_est}
and \eqref{noise_est}. The first inequality holds because $\langle
n\triangle^{n}\tilde{X},\Phi_{jk}\rangle_{n}$
is $N(0,\sigma_{\left(k-1\right)h_{n}}^{2})$-distributed conditional
on $\c G_{\left(k-1\right)h_{n}}$. In order to see why the second
inequality is satisfied, let $g_{l}=\epsilon_{l/n}\varphi_{jk}((l-\frac{1}{2})/{n})$
for $l=1,\dots, n $. The $g_{l}$ are independent and centered
such that for any $1\leq l_{1},\dots,l_{8}\leq  n$ with $\E[g_{l_{1}}\cdots
g_{l_{8}}]\neq0$, referring to observations on the same bin, each $g_{l}$ appears at least twice and there are at most four distinct
$g_{l}$. If there are exactly four distinct $g_{l}$, e.g.
$l_{1}=l_{2},l_{3}=l_{4},l_{5}=l_{6},l_{7}=l_{8}$,
we arrive at the bound
\begin{eqnarray*}
\sum_{l_{1},l_{3},l_{5},l_{7}}\E\left[g_{l_{1}}\cdots g_{l_{8}}\right]\leq &
\left(\eta^{2}\right)^{4}n^{4}\left[\varphi_{jk},\varphi_{jk}\right]_{n}^{4}.
\end{eqnarray*}
The leading term in \eqref{noise_est} does not include eighth moments of the noise, but the fourth power of the second moment which we denote $(\eta^2)^4$ to prevent any confusion.
If there are less than four distinct $g_{l}$, we obtain from \eqref{or2}
and \eqref{eq:or3} with the assumption $\E[\epsilon_t^8]<\infty$, that the respective sums are asymptotically of
smaller order. The terms with eighth moments are thus negligible. This implies \eqref{noise_est}:
\begin{align*}
\E\left[\left[\epsilon,\varphi_{jk}\right]_{n}^{8}\right] & =n^{-8}\sum_{1\leq
l_{1},\dots,l_{8}\leq nh_{n}}\E\left[g_{l_{1}}\cdots
g_{l_{8}}\right]\lesssim\left(\eta^{2}\right)^{4}\frac{\left[\varphi_{jk},\varphi_{jk}\right]_{n}^{4}}{n^{4}}.
\end{align*}
It remains to verify \eqref{s4} and \eqref{s5}. The proof follows a similar strategy
as the proofs of Proposition 5.10, step 4, of \cite{Jacod2010} and Lemma 5.7 of \cite{JLMPV}. It is sufficient
to show with $\delta_{k}^{n}(M)=M_{kh_{n}}-M_{(k-1)h_{n}}$ that
\begin{equation}
\sum_{k=1}^{\lfloor{th_{n}^{-1}}\rfloor}\E\left[\zeta_{k}^{n}\delta_{k}^{n}\left(M\right)\big|\mathcal{G}_{(k-1)h_{n}}\right]\stackrel{\P}{\rightarrow}0\label{eq:-1}
\end{equation}
for any $M\in\c N$, the set of square-integrable $(\c G_{t})_{0\leq t\leq1}$-martingales.
Note that \eqref{eq:-1} is closed under $L^{2}$-convergence with
respect to the terminal variables $M_{1}\in L^{2}(\c G)$ for $M\in\c N$ what follows by Cauchy-Schwarz inequality.
Define subsets $\c N^{0},\c N^{1}$, $\c N^{2}$ of $\c N$, where
$\c N^{0}$ is the space of all square-integrable martingales adapted
to $\c W=\sigma(W_{s}:s\leq1)$, i.e.\,every such martingale has the
form $C+\int_{0}^{t}h_{s}\,dW_{s}$ for some constant $C$ and a
predictable square-integrable process $h\in\c W$. $\c N^{1}$ is
the set of all square-integrable $(\c F_{t})$-martingales which are
orthogonal to $W$, and $\c N^{2}$ is the space of all square-integrable
martingales adapted to the filtration $\c E_{t}=\sigma(\epsilon_{s}:s\leq t)$,
generated by the noise process. Then the set of square-integrable
martingales of the form $M\cdot N$, for $M\in\c N^{0}\cup\c N^{1}$,
$N\in\c N^{2}$, is total in $\c N$ (by independence any process
of the form $M\cdot N$ is again a martingale) and it is enough to
show \eqref{eq:-1} for such processes. Using the decomposition
\begin{equation}
\delta_{k}^{n}\left(MN\right)=\delta_{k}^{n}\left(M\right)\delta_{k}^{n}\left(N\right)+N_{\left(k-1\right)h_{n}}\delta_{k}^{n}\left(M\right)+M_{\left(k-1\right)h_{n}}\delta_{k}^{n}\left(N\right)\label{eq:martingale_decomp}
\end{equation}
we have by independence of $W$ and noise for any $k=1,\ldots,h_n^{-1}$: 
\begin{align*}
 & \E\left[\left.\zeta_{k}^{n}\delta_{k}^{n}\left(MN\right)\right|\c G_{\left(k-1\right)h_{n}}\right]\\
 & =n^{1/4}h_{n}\sum_{j=1}^{\lfloor nh_{n}\rfloor -1}w_{jk}\left(\E\left[\l{\tilde{S}_{jk}^{2}\delta_{k}^{n}\left(MN\right)}\c G_{\left(k-1\right)h_{n}}\right]-\E\left[\l{\tilde{S}_{jk}^{2}}\c G_{\left(k-1\right)h_{n}}\right]\E\left[\l{\delta_{k}^{n}\left(MN\right)}\c G_{\left(k-1\right)h_{n}}\right]\right)\\
 & =n^{1/4}h_{n}\sum_{j=1}^{\lfloor nh_{n}\rfloor -1}w_{jk}\Bigg(N_{\left(k-1\right)h_{n}}\E\left[\l{\left\langle n\triangle^{n}\tilde{X},\Phi_{jk}\right\rangle _{n}^{2}\delta_{k}^{n}\left(M\right)}\c F_{\left(k-1\right)h_{n}}\right]\\
 & \ \ \ \ \ \ \ \ \ \ \ -2\E\left[\l{\left\langle n\triangle^{n}\tilde{X},\Phi_{jk}\right\rangle _{n}\delta_{k}^{n}\left(M\right)}\c F_{\left(k-1\right)h_{n}}\right]\E\left[\l{\left[\epsilon,\varphi_{jk}\right]_{n}\delta_{k}^{n}\left(N\right)}\c E_{\left(k-1\right)h_{n}}\right]\\
 & \ \ \ \ \ \ \ \ \ \ \ +M_{\left(k-1\right)h_{n}}\E\left[\l{\left[\epsilon,\varphi_{jk}\right]_{n}^{2}\delta_{k}^{n}\left(N\right)}\c E_{\left(k-1\right)h_{n}}\right]\Bigg).
\end{align*}
Let first $M\in\c N^{0}$. As $\c N^{0}$ is closed and because the
case $M$ constant is trivial, we can assume that $M=\int_{0}^{\cdot}\gamma_{s}\,dW_{s}$
for $\gamma$ bounded, adapted to $\c W$ and
piecewise constant on intervals $(T_{q},T_{q+1}]$ for some $0=T_{0}<T_{1}<\dots,T_{m}=1$,
$m\geq1$, such that
\begin{align}
 & \E\hspace*{-.05cm}\left[\l{\left\langle n\triangle^{n}\tilde{X},\Phi_{jk}\right\rangle _{n}^{2}\delta_{k}^{n}\left(M\right)}\c F_{\left(k-1\right)h_{n}}\right]\label{eq:-3}\\
 & =\sum_{l,p=1}^{n}\sigma_{\left(k-1\right)h_{n}}^{2}\sum_{q=1}^{m}\E\left[\l{\triangle_{l}^{n}W\triangle_{p}^{n}W\gamma_{t_{q}}\left(W_{T_{q+1}\wedge kh_{n}}\hspace*{-.05cm}-\hspace*{-.05cm}W_{T_{q}\vee\left(k-1\right)h_{n}}\right)}\c F_{\left(k-1\right)h_{n}}\right]\hspace*{-.05cm}\Phi_{jk}\hspace*{-.05cm}\Big(\frac{l}{n}\Big)\hspace*{-.05cm}\Phi_{jk}\hspace*{-.05cm}\Big(\frac{p}{n}\Big).\nonumber 
\end{align}
For $n$ large enough there is at most one $T_{q}$ per bin.
If there is no $T_{q}$ on the $k$th bin, the conditional expectation
above vanishes by independence of the Brownian increments for any
$l,p$. On the other hand, there are only $m$ bins containing some
$T_{q}$ and for every such bin the left-hand side of \eqref{eq:-3}
is bounded, what can be seen e.g.\,by \eqref{norm_est} and because $M$ is square-integrable.
Hence, 
\begin{equation}
n^{1/4}h_{n}\sum_{k=1}^{\lfloor{th_{n}^{-1}}\rfloor}\sum_{j=1}^{\lfloor nh_{n}\rfloor -1}w_{jk}N_{\left(k-1\right)h_{n}}\E\left[\l{\left\langle n\triangle^{n}\tilde{X},\Phi_{jk}\right\rangle _{n}^{2}\delta_{k}^{n}\left(M\right)}\c F_{\left(k-1\right)h_{n}}\right]=\KLEINO_{\P}(1).\label{eq:term1}
\end{equation}
Next, let $M\in\c N^{1}$, i.e.\,$M$ is orthogonal to $W$. The left-hand side in \eqref{eq:-3}
is now equal to
\[
\sum_{l,p,q=1}^{n}\sigma_{\left(k-1\right)h_{n}}^{2}\E\left[\l{\triangle_{l}^{n}W\triangle_{p}^{n}W\triangle_{q}^{n}M}\c F_{\left(k-1\right)h_{n}}\right]\Phi_{jk}\Big(\frac{l}{n}\Big)\Phi_{jk}\Big(\frac{p}{n}\Big).
\]
The conditional expectation vanishes, except for $l=p=q$, $p<l=q$
or $l<p=q$. For $l=p=q$ we obtain by It\^{o}'s formula 
\begin{align*}
 & \E\left[\l{\left(\triangle_{l}^{n}W\right)^{2}\triangle_{l}^{n}M}\c F_{\left(k-1\right)h_{n}}\right]=\E\left[\l{\left(\left(\triangle_{l}^{n}W\right)^{2}-\left(\frac{1}{n}\right)\right)\triangle_{l}^{n}M}\c F_{\left(k-1\right)h_{n}}\right]\\
 & \hspace*{0.5cm}+\E\left[\l{\left(\frac{1}{n}\right)\triangle_{l}^{n}M}\c F_{\left(k-1\right)h_{n}}\right]=\E\left[\l{\left(\int_{\frac{l-1}{n}}^{\frac{l}{n}}W_{s}\,dW_{s}\right)\triangle_{l}^{n}M}\c F_{\left(k-1\right)h_{n}}\right].
\end{align*}
However, $((\int_{0}^{t}W_{s}\,dW_{s})\cdot M_{t})_{0\leq t\leq1}$
is an $(\c F_{t})$-martingale by orthogonality such that the last
expression vanishes. The
cases $p<l=q$ and $l<p=q$ follow similarly. Hence, \eqref{eq:term1}
is still satisfied; the left-hand side is actually zero. \\ 
With respect to $N$, as $\c N^{2}$ is closed, we can assume without
loss of generality that $N_{1}=f(\epsilon_{T_{1}},\dots,\epsilon_{T_{m'}})$
for some measurable function $f$ and some $0\leq T_{1}<\dots<T_{m'}\leq1$,
$m'\geq1$. Similar as before, for $n$ large there is at most one
$T_{q'}$ per bin. On any bin not containing such a
$T_{q'}$ it holds that $\delta_{k}^{n}(N)=0$. Bounding the terms for the $m'$ other bins 
yields for $M\in\mathcal{N}^1\cup\mathcal{N}^2$
\begin{equation}
n^{1/4}h_{n}\sum_{k=1}^{\lfloor{th_{n}^{-1}}\rfloor}\sum_{j=1}^{\lfloor nh_{n}\rfloor -1}w_{jk}M_{\left(k-1\right)h_{n}}\E\left[\l{\left[\epsilon,\varphi_{jk}\right]_{n}^{2}\delta_{k}^{n}\left(N\right)}\c E_{\left(k-1\right)h_{n}}\right]=\KLEINO_{\P}\left(1\right).\label{eq:term2}
\end{equation}
From the previous discussion we further see that for all but at most
$m+m'$ blocks: 
\[
\sum_{j=1}^{nh_{n}-1}w_{jk}\E\left[\l{\left\langle n\triangle^{n}\tilde{X},\Phi_{jk}\right\rangle _{n}\delta_{k}^{n}\left(M\right)}\c F_{\left(k-1\right)h_{n}}\right]\E\left[\l{\left[\epsilon,\varphi_{jk}\right]_{n}\delta_{k}^{n}\left(N\right)}\c E_{\left(k-1\right)h_{n}}\right]=0,
\]
whereas bounds on the remaining $m+m'$ blocks guarantee that the cross terms tend to zero in probability.
We conclude that \eqref{eq:-1} holds. This completes the proof of Proposition \ref{propmain}.

\subsection{{\bf{Proof of Proposition \ref{propremainder}}}}
We first give a general outline of the proof, deferring some technical
details to the end of this section. By Taylor we have for all $k=1,\dots,h_{n}^{-1}$
and $j=1,\dots,\lfloor nh_{n}\rfloor -1$, the existence of random variables $\xi_{jk}$
such that $S_{jk}^{2}-\tilde{S}_{jk}^{2}=2\tilde{S}_{jk}(S_{jk}-\tilde{S}_{jk})+2(\xi_{jk}-\tilde{S}_{jk})(S_{jk}-\tilde{S}_{jk})$
and $|\xi_{jk}-\tilde{S}_{jk}|\leq|S_{jk}-\tilde{S}_{jk}|$. This
yields
\begin{align*}
n^{\frac{1}{4}}\left(\IV_{n,t}^{or}(Y)-\IV_{n,t}^{or}(\tilde{X}+\epsilon)\right) & =n^{\frac{1}{4}}\Big(h_{n}\sum_{k=1}^{\f{th_{n}^{-1}}}\sum_{j=1}^{\lfloor nh_{n}\rfloor -1}w_{jk}\left(S_{jk}^{2}-\tilde{S}_{jk}^{2}\right)\Big)\\
 & =\Big(n^{\frac{1}{4}}h_{n}\sum_{k=1}^{\lfloor th_{n}^{-1}\rfloor}\sum_{j=1}^{\lfloor nh_{n}\rfloor -1}w_{jk}\left(2\tilde{S}_{jk}\left(S_{jk}-\tilde{S}_{jk}\right)\right)\Big)\\
 & +\Big(n^{\frac{1}{4}}h_{n}\sum_{k=1}^{\lfloor th_{n}^{-1}\rfloor}\sum_{j=1}^{\lfloor nh_{n}\rfloor -1}w_{jk}\left(\xi_{jk}-\tilde{S}_{jk}\right)\left(S_{jk}-\tilde{S}_{jk}\right)\Big).
\end{align*}
For the second sum above, which we denote by $Z_t^{n}$, we obtain by the Markov inequality and Step 1 below
for any $\varepsilon>0$
\begin{eqnarray*}
\P\left(\sup_{0\leq t\leq1}\left|Z_{t}^{n}\right|>\varepsilon\right) & \leq & \P\Bigg(\Big(n^{\frac{1}{4}}h_{n}\sum_{k=1}^{h_{n}^{-1}}\sum_{j=1}^{\lfloor nh_{n}\rfloor -1}w_{jk}\left|S_{jk}-\tilde{S}_{jk}\right|^{2}\Big)>\varepsilon\Bigg)\\
 & \leq & \varepsilon^{-1}n^{\frac{1}{4}}h_{n}\sum_{k=1}^{h_{n}^{-1}}\sum_{j=1}^{\lfloor nh_{n}\rfloor -1}w_{jk}\E\left[\left(S_{jk}-\tilde{S}_{jk}\right)^{2}\right]\\
 & \lesssim & \varepsilon^{-1}n^{\frac{1}{4}}h_{n}\rightarrow 0.
\end{eqnarray*}
Let $T_{k}^{n}=\sum_{j=1}^{\lfloor nh_{n}\rfloor -1}w_{jk}\left(2\tilde{S}_{jk}\left(S_{jk}-\tilde{S}_{jk}\right)\right)$
and write the first sum above as $M_{t}^{n}+R_{t}^{n}$ with
\begin{eqnarray*}
M_{t}^{n} & = & n^{\frac{1}{4}}h_{n}\sum_{k=1}^{\lfloor th_{n}^{-1}\rfloor}\left(T_{k}^{n}-\E\left[\l{T_{k}^{n}}\mathcal{G}_{\left(k-1\right)h_{n}}\right]\right),\\
R_{t}^{n} & = & n^{\frac{1}{4}}h_{n}\sum_{k=1}^{\lfloor th_{n}^{-1}\rfloor}\E\left[\l{T_{k}^{n}}\mathcal{G}_{\left(k-1\right)h_{n}}\right].
\end{eqnarray*}
In Step 2 we show that
\[
\sum_{k=1}^{\lfloor th_{n}^{-1}\rfloor}\E\left[\left(n^{\frac{1}{4}}h_nT_{k}^{n}\right)^{2}\right]\xrightarrow{}0,\quad n\xrightarrow{}\infty.
\]
A well known result thereby yields $M_{t}^{n}\xrightarrow{ucp}0$. Finally,
observe that
\begin{align*}
 & \E\left[\l{\left(2\tilde{S}_{jk}\left(S_{jk}-\tilde{S}_{jk}\right)\right)}\c G_{\left(k-1\right)h_{n}}\right]\\
 & =\E\left[\l{2\left(\left\langle n\triangle^{n}\tilde{X},\Phi_{jk}\right\rangle _{n}-\left[\epsilon,\varphi_{jk}\right]_{n}\right)\left\langle n\triangle^{n}\left(X-\tilde{X}\right),\Phi_{jk}\right\rangle _{n}}\c G_{\left(k-1\right)h_{n}}\right]\\
 & =\E\left[\l{2\left\langle n\triangle^{n}\tilde{X},\Phi_{jk}\right\rangle _{n}\left\langle n\triangle^{n}\left(X-\tilde{X}\right),\Phi_{jk}\right\rangle _{n}}\c G_{\left(k-1\right)h_{n}}\right],
\end{align*}
i.e.\,the noise terms vanish, thereby simplifying the following calculations.\\
Write \(\E[(2\tilde{S}_{jk}(S_{jk}-\tilde{S}_{jk}))|\c G_{\left(k-1\right)h_{n}}]\)
as the sum $D_{jk}^{n}+V_{jk}^{n}$ with
\begin{eqnarray*}
D_{jk}^{n} & = & \E\left[\l{2\left\langle n\triangle^{n}\tilde{X},\Phi_{jk}\right\rangle _{n}\left(\sum_{l=1}^{n}\left(\int_{\frac{l-1}{n}}^{\frac{l}{n}}b_{s}\, ds\right)\Phi_{jk}\left(\frac{l}{n}\right)\right)}\c G_{\left(k-1\right)h_{n}}\right],\\
V_{jk}^{n} & = & \E\left[\l{2\left\langle n\triangle^{n}\tilde{X},\Phi_{jk}\right\rangle _{n}\left(\sum_{l=1}^{n}\left(\int_{\frac{l-1}{n}}^{\frac{l}{n}}\left(\sigma_{s}-\sigma_{\left(k-1\right)h_{n}}\right)\, dW_{s}\right)\Phi_{jk}\left(\frac{l}{n}\right)\right)}\c G_{\left(k-1\right)h_{n}}\right].
\end{eqnarray*}
In Step 3 we show that $\left|D_{jk}^n+V_{jk}^n\right|\lesssim h_{n}^{\beta}$
for some $\beta>1/2$. This yields immediately
\begin{eqnarray*}
\sup_{0\leq t\leq1}\left|R_{t}^{n}\right| & \leq & n^{\frac{1}{4}}h_{n}\sum_{k=1}^{h_{n}^{-1}}\sum_{j=1}^{\lfloor nh_{n}\rfloor -1}w_{jk}\left|D_{jk}^{n}+V_{jk}^{n}\right|\lesssim n^{\frac{1}{4}}h_{n}^{\beta}=\KLEINO\left(1\right),
\end{eqnarray*}
implying $ucp$-convergence. We therefore conclude that
\[
n^{\frac{1}{4}}\left(\IV_{n,t}^{or}(Y)-\IV_{n,t}^{or}(\tilde{X}+\epsilon)\right)\xrightarrow{ucp}0,\,\,\, n\xrightarrow{}\infty.
\]
The second claim
\[
n^{\frac14}\int_{0}^{t}\big(\sigma_{s}^{2}-\sigma_{\lfloor sh_{n}^{-1}\rfloor h_{n}}^{2}\big)\, ds\xrightarrow{ucp}0,\,\,\, n\xrightarrow{}\infty,
\]
follows on Assumption \ref{H1} for both, H\"older smooth and semimartingale volatility, by Equations \eqref{JM2} and \eqref{JM}. This proves Proposition \ref{propremainder}. We end this section with detailed proofs
of Steps $1$ -- $3$.\\[.2cm]
Step 1: $\E[(S_{jk}-\tilde{S}_{jk})^{4}]\lesssim h_{n}^{2}$.\\[.1cm]
Using the decomposition
\begin{align}
S_{jk}-\tilde{S}_{jk} & =\left\langle n\triangle^{n}(X-\tilde{X}),\Phi_{jk}\right\rangle _{n}\label{S_diff}\\
 & =\sum_{l=1}^{n}\left(\int_{\frac{l-1}{n}}^{\frac{l}{n}}b_{s}\, ds\right)\Phi_{jk}\left(\frac{l}{n}\right)+\sum_{l=1}^{n}\left(\int_{\frac{l-1}{n}}^{\frac{l}{n}}\left(\sigma_{s}-\sigma_{\left(k-1\right)h_{n}}\right)\, dW_{s}\right)\Phi_{jk}\left(\frac{l}{n}\right)\nonumber 
\end{align}
into drift and volatility terms, we obtain
\begin{align*}
\E\left[\left(S_{jk}-\tilde{S}_{jk}\right)^{4}\right] & \lesssim\hspace{1em}\E\left[\left(\sum_{l=1}^{n}\left(\int_{\frac{l-1}{n}}^{\frac{l}{n}}b_{s}\, ds\right)\Phi_{jk}\left(\frac{l}{n}\right)\right)^{4}\right]\\
 & \quad +\E\left[\left(\sum_{l=1}^{n}\left(\int_{\frac{l-1}{n}}^{\frac{l}{n}}\left(\sigma_{s}-\sigma_{\left(k-1\right)h_{n}}\right)\, dW_{s}\right)\Phi_{jk}\left(\frac{l}{n}\right)\right)^{4}\right].
\end{align*}
The first addend is bounded by $h_{n}^{2}$. For the second let $\kappa_{l}=\int_{\frac{l-1}{n}}^{\frac{l}{n}}\left(\sigma_{s}-\sigma_{\left(k-1\right)h_{n}}\right)\, dW_{s}$, such that
\begin{align*}
 & \E\left[\left(\sum_{l=1}^{n}\left(\int_{\frac{l-1}{n}}^{\frac{l}{n}}\left(\sigma_{s}-\sigma_{\left(k-1\right)h_{n}}\right)\, dW_{s}\right)\Phi_{jk}\left(\frac{l}{n}\right)\right)^{4}\right]\\
 & =\sum_{l,l',p,p'}\E\left[\kappa_{l}\kappa_{l'}\kappa_{p}\kappa_{p'}\right]\Phi_{jk}\left(\frac{l}{n}\right)\Phi_{jk}\left(\frac{l'}{n}\right)\Phi_{jk}\left(\frac{p}{n}\right)\Phi_{jk}\left(\frac{p'}{n}\right).
\end{align*}
The only choices for $l,l',p,p'$ with non-vanishing results
are $l,l'<p=p'$, $l<l'=p=p'$ and $l=l'=p=p'$. In all three cases
we can conclude by the Burkholder inequality and \eqref{e2} that
\begin{align*} \left|\E\left[\kappa_{l}\kappa_{l'}\kappa_{p}\kappa_{p'}\right]\Phi_{jk}\left(\frac{l}{n}\right)\Phi_{jk}\left(\frac{l'}{n}\right)\Phi_{jk}\left(\frac{p}{n}\right)\Phi_{jk}\left(\frac{p'}{n}\right)\right|
 \lesssim n^{-4}h_{n}^{-2}.
\end{align*}
Observe that in any of the three mentioned cases we find at least
two identical integers $l,l',p$ or $p'$. In all, there are $\lfloor nh_{n}\rfloor \cdot\binom{\lfloor nh_{n}\rfloor -1}{2}\cdot4!$
possibilities to choose such indices. Hence, we obtain
\[
\E\left[\left(\sum_{l=1}^{n}\left(\int_{\frac{l-1}{n}}^{\frac{l}{n}}\left(\sigma_{s}-\sigma_{\left(k-1\right)h_{n}}\right)\, dW_{s}\right)\Phi_{jk}\left(\frac{l}{n}\right)\right)^{4}\right]\lesssim\left(nh_{n}\right)^{3}n^{-4}h_{n}^{-2}=n^{-1}h_{n}\lesssim h_{n}^{2}
\]
and therefore the claim holds.\\[.2cm] 
Step 2: $\sum_{k=1}^{\lfloor th_{n}^{-1}\rfloor}\E[(n^{\frac{1}{4}}h_{n}T_{k}^{n})^{2}]\xrightarrow{}0,\quad n\xrightarrow{}\infty$.\\[.1cm]
Applying the Minkowski and Cauchy-Schwarz inequalities, we obtain
\begin{align*}
 & \Big\Vert n^{\frac{1}{4}}h_{n}\sum_{j=1}^{\lfloor nh_{n}\rfloor -1}w_{jk}\left(2\tilde{S}_{jk}\left(S_{jk}-\tilde{S}_{jk}\right)\right)\Big\Vert _{L^{2}\left(\P\right)}^{2}\\
 & \overset{}{\leq}n^{\frac{1}{2}}h_{n}^{2}\left(\sum_{j=1}^{\lfloor nh_{n}\rfloor -1}\left\Vert w_{jk}\left(2\tilde{S}_{jk}\left(S_{jk}-\tilde{S}_{jk}\right)\right)\right\Vert _{L^{2}\left(\P\right)}\right)^{2}\\
 & \overset{\mbox{}}{\leq}n^{\frac{1}{2}}h_{n}^{2}\left(\sum_{j=1}^{\lfloor nh_{n}\rfloor -1}w_{jk}\left(\E\left[\tilde{S}_{jk}^{4}\right]\right)^{\frac{1}{4}}\left(\E\left[\left(S_{jk}-\tilde{S}_{jk}\right)^{4}\right]\right)^{\frac{1}{4}}\right)^{2}.
\end{align*}
By Step $1$ we already know that $\E\big[\big(S_{jk}-\tilde{S}_{jk}\big)^{4}\big]\lesssim h_{n}^{2}$.
Because $\sigma$ is bounded, we obtain by \eqref{stat_moments} the bound
\begin{align*}
\E\left[\tilde{S}_{jk}^{4}\right] & \leq\E^{\frac{1}{2}}\left[\E\left[\l{\tilde{S}_{jk}^{8}}\c G_{\left(k-1\right)h_{n}}\right]\right]\lesssim\E^{\frac{1}{2}}\left[\left(\sigma_{\left(k-1\right)h_{n}}^{2}+\frac{\eta^{2}}{n}\left[\varphi_{jk},\varphi_{jk}\right]_{n}\right)^{4}\right]\\
 & \lesssim\left(1+\frac{\eta^{2}}{n}\left[\varphi_{jk},\varphi_{jk}\right]_{n}\right)^{2}\leq\left(1+\frac{\eta^{2}}{n}\left[\varphi_{jk},\varphi_{jk}\right]_{n}\right)^{4}.
\end{align*}
Together with \eqref{helpweights} it follows that
\begin{align*}
\sum_{k=1}^{\lfloor th_{n}^{-1}\rfloor}\E\left[\left(n^{\frac{1}{4}}h_{n}T_{k}^{n}\right)^{2}\right] & \lesssim n^{\frac{1}{2}}h_{n}^{3}\sum_{k=1}^{\lfloor th_{n}^{-1}\rfloor}\left(\sum_{j=1}^{\lfloor nh_{n}\rfloor -1}w_{jk}\left(1+\frac{\eta^{2}}{n}\left[\varphi_{jk},\varphi_{jk}\right]_{n}\right)\right)^{2}\\
 & \lesssim n^{\frac{1}{2}}h_{n}^{2}\cdot n^{2}h_{n}^{4}=\KLEINO\left(1\right).
\end{align*}
Step 3: $\left|D_{jk}^n+V_{jk}^n\right|\lesssim h_n^{\beta}$ for some
$\beta>1/2$.\\[.2cm]
Expanding the sums in $V_{jk}^n$ and It\^{o} isometry
yield 
\begin{eqnarray*}
V_{jk}^n & = & \sum_{l,m=1}^{n}\left(\E\left[\l{\triangle_{l}^{n}\tilde{X}\left(\int_{\frac{m-1}{n}}^{\frac{m}{n}}\left(\sigma_{s}-\sigma_{\left(k-1\right)h_{n}}\right)\,dW_{s}\right)}\c G_{\left(k-1\right)h_{n}}\right]\Phi_{jk}\left(\frac{l}{n}\right)\Phi_{jk}\left(\frac{m}{n}\right)\right)\\
 & \overset{}{=} & \sum_{l=1}^{n}\E\left[\int_{\frac{l-1}{n}}^{\frac{l}{n}}\left(\sigma_{\left(k-1\right)h_{n}}\left(\sigma_{s}-\sigma_{\left(k-1\right)h_{n}}\right)\right)\,ds\right]\Phi_{jk}^{2}\left(\frac{l}{n}\right).
\end{eqnarray*}
By Assumption \ref{H1} for $s\in[(k-1)h_n,kh_n]$:
\begin{align*}
\left|\E\left[\left(\sigma_{\left(k-1\right)h_{n}}\left(\sigma_{s}-\sigma_{\left(k-1\right)h_{n}}\right)\right)\right]\right| & =\left|\E\left[\sigma_{\left(k-1\right)h_{n}}\E\left[\l{\sigma_{s}-\sigma_{\left(k-1\right)h_{n}}}\c G_{\left(k-1\right)h_{n}}\right]\right]\right|\lesssim h_{n}^{\alpha},
\end{align*}
and hence by Fubini $\left|V_{jk}\right|\lesssim h_{n}^{\alpha},$
as well. With respect to $D_{jk}^{n}$, we need an additional approximation.
By Assumption \ref{H1}, decomposing $b_{s}=g(b^A_s,b^B_s)$ into H\"{o}lder and semimartingale part we have by the
boundedness of $\E[|\langle n\triangle^{n}\tilde{X},\Phi_{jk}\rangle_{n}|]$ from \eqref{norm_est}: 
\begin{align*}
 & \left|\E\left[\l{\left\langle n\triangle^{n}\tilde{X},\Phi_{jk}\right\rangle _{n}\int_{\frac{l-1}{n}}^{\frac{l}{n}}b_{s}\,ds}\c G_{\left(k-1\right)h_{n}}\right]\right|\\
 & \lesssim\left|\E\left[\l{\left\langle n\triangle^{n}\tilde{X},\Phi_{jk}\right\rangle _{n}\int_{\frac{l-1}{n}}^{\frac{l}{n}}\left(b_{s}-b_{\left(k-1\right)h_{n}}\right)\,ds}\c G_{\left(k-1\right)h_{n}}\right]\right|\\
 & +\left|\frac{b_{\left(k-1\right)h_{n}}}{n}\E\left[\l{\left\langle n\triangle^{n}\tilde{X},\Phi_{jk}\right\rangle _{n}}\c G_{\left(k-1\right)h_{n}}\right]\right|\lesssim h_{n}^{\nu\wedge\frac{1}{2}}n^{-1}.
\end{align*}
Using this bound we find that 
\begin{eqnarray*}
\left|D_{jk}^{n}\right| & \leq & \sum_{l=1}^{n}\left|\E\left[\l{2\left\langle n\triangle^{n}\tilde{X},\Phi_{jk}\right\rangle _{n}\int_{\frac{l-1}{n}}^{\frac{l}{n}}b_{s}\,ds}\c G_{\left(k-1\right)h_{n}}\right]\right|\left|\Phi_{jk}\left(\frac{l}{n}\right)\right|\lesssim h_{n}^{(\nu\wedge\frac{1}{2})+\frac{1}{2}}\,.
\end{eqnarray*}
We obtain the claim with $\beta=\min\left\{ (\nu\wedge\frac{1}{2})+\frac{1}{2},\alpha\right\} $.
This is the only time we need the smoothness of the drift in Assumption \ref{H1}.

\subsection{{\bf{Proofs of Theorem \ref{thm:2} and Theorem \ref{thm:3} for oracle estimation}}}
We decompose $X$ similarly as in the proof of Theorem \ref{thm:1}:
\begin{align}X_t=X_0+\bar B_t+\tilde B_t+\bar C_t+\tilde C_t\,,\end{align}
where we denote
\begin{subequations}
\begin{align}\bar B_t=\int_0^t  b_{\lfloor sh_n^{-1}\rfloor h_n}\,ds~,~\tilde B_t=\int_0^t (b_s-b_{\lfloor sh_n^{-1}\rfloor h_n})\,ds\,,\end{align}
\begin{align}\label{appr3}\bar C_t=\int_0^t  \sigma_{\lfloor sh_n^{-1}\rfloor h_n}\,dW_s~,~\tilde C_t=\int_0^t (\sigma_s-\sigma_{\lfloor sh_n^{-1}\rfloor h_n})\,dW_s\,.\end{align}
\end{subequations}
In order to establish a functional CLT, we decompose the estimation errors of \eqref{lmm} (and likewise \eqref{specv}) in the following way:
\begin{subequations}
\begin{align}\label{dec1b}&\LMM_{n,t}^{or}(Y)-\operatorname{vec}\Big(\int_0^t\Sigma_s\,ds\Big)=\LMM_{n,t}^{or}(\bar C+\epsilon)-\operatorname{vec}\Big(\int_0^t\Sigma_{\lfloor sh_n^{-1}\rfloor h_n}\,ds\Big)\\
&\label{dec2b}+\LMM_{n,t}^{or}(Y)-\LMM_{n,t}^{or}(\bar C+\epsilon)-\operatorname{vec}\Big(\int_0^t\big(\Sigma_s-\Sigma_{\lfloor sh_n^{-1}\rfloor h_n}\big)\,ds\Big)
\,.\end{align}
\end{subequations}
One crucial step to cope with multi-dimensional non-synchronous data is Lemma \ref{key} which is proved next. Below, we give a concise proof of the functional CLTs for the estimators \eqref{lmm} and \eqref{specv}, where after restricting to a synchronous reference scheme many steps develop as direct extensions of the one-dimensional case. The stable CLTs for the \textit{leading terms}, namely the right-hand side of \eqref{dec1b} and the analogue for estimator \eqref{specv}, are established in paragraph \ref{subsec:l}. The \textit{remainder terms} \eqref{dec2b} and their analogues are handled in paragraph \ref{subsec:r}.
\subsubsection{{\bf{Proof of Lemma \ref{key}}}}
Consider for $l,m=1,\ldots,d,$ observation times $t_i^{(l)}=F_l^{-1}(i/n_l)$ and $t_i^{(m)}=F_m^{-1}(i/n_m)$. Define a next-tick interpolation function by
 \begin{align*}t_+^{(l)}(s)=\min{\left(t_v^{(l)},v=0,\dots,n_l|t_v^{(l)}\ge s\right)}\,,l=1,\ldots,d,\end{align*}
 and analogously a previous-tick interpolation function by
 \begin{align*}t_-^{(l)}(s)=\max{\left(t_v^{(l)},v=0,\dots,n_l|t_v^{(l)}\le s\right)}\,,l=1,\ldots,d.\end{align*}
 We decompose increments of $X^{(l)}$ between adjacent observation times $t_{v-1}^{(l)},t_{v}^{(l)},v=1,\ldots,n_l$, in the sum of increments of $X^{(l)}$ over all time intervals $[t_{i-1}^{(m)},t_i^{(m)}]$ contained in $[t_{v-1}^{(l)},t_v^{(l)}]$ and the remaining time intervals at the left $\big[t_{v-1}^{(l)},t_+^{(m)}\big(t_{v-1}^{(l)}\big)\big]$ and the right border $\big[t_-^{(m)}\big(t_{v}^{(l)}\big),t_{v}^{(l)}\big]$:
 \begin{align*}X_{t_v^{(l)}}^{(l)}-X_{t_{v-1}^{(l)}}^{(l)}=\Big(X_{t_v^{(l)}}^{(l)}-X_{t_-^{(m)}(t_{v}^{(l)})}^{(l)}\Big)+\sum_{\Delta_it^{(m)}\subset \Delta_vt^{(l)}}\Big(X_{t_i^{(m)}}^{(l)}-X_{t_{i-1}^{(m)}}^{(l)}\Big)+\Big(X_{t_+^{(m)}(t_{v-1}^{(l)})}^{(l)}-X_{t_{v-1}^{(l)}}^{(l)}\Big)\,.\end{align*}
  If there is only one observation of $X^{(m)}$ in $[t_{v-1}^{(l)},t_v^{(l)}]$, set $\sum_{\Delta_it^{(m)}\subset \Delta_vt^{(l)}}\big(X_{t_i^{(m)}}^{(l)}-X_{t_{i-1}^{(m)}}^{(l)}\big)=0$.\\ If there is no observation of $X^{(m)}$ in $[t_{v-1}^{(l)},t_v^{(l)}]$ we take the union of a set of intervals $\bigcup_{v\in V} [t_{v-1}^{(l)},t_v^{(l)}]$ which contains at least one observation time of $X^{(m)}$. We use an expansion of $\Phi_{jk}(t)-\Phi_{jk}(s)$.
By virtue of $\sin(t)-\sin(s)=2\cos((t+s)/2)\sin((t-s)/2)$ and the sine expansion, we obtain for $s,t\in[kh_n,(k+1)h_n)$:
\begin{align}\label{Phiexp}\Phi_{jk}(t)-\Phi_{jk}(s)\asymp \sqrt{2}h_n^{-3/2}j\pi\cos\big(j\pi h_n^{-1}(\tfrac{t+s}{2}-kh_n)\big)\,(t-s)\,.\end{align}
In particular, for $t-s=\mathcal{O} (n^{-1})$ we have that $\Phi_{jk}(t)-\Phi_{jk}(s)=\mathcal{O}\left(\varphi_{jk}(\tfrac{t+s}{2})n^{-1}\right)$.\\ With \( u_v^{(m)}=(1/2)(t_+^{(m)}(t_{v}^{(l)})-t_-^{(m)}(t_{v}^{(l)}))\) and \(\tilde u_v^{(m)}=(1/2)(t_+^{(m)}(t_{v-1}^{(l)})-t_-^{(m)}(t_{v-1}^{(l)}))\), we infer 
 \begin{align*}&\sum_{k=1}^{\lfloor th_n^{-1}\rfloor}h_n\sum_{j\ge 1}w_{jk}^{l,m}\sum_{i=1}^{n_l}\left(X_{t_i^{(l)}}^{(l)}-X_{t_{i-1}^{(l)}}^{(l)}\right) X^{(l)}\Phi_{jk}(\bar t_i^{(l)})\sum_{v=1}^{n_m}\left(X_{t_v^{(m)}}^{(m)}-X_{t_{v-1}^{(m)}}^{(m)}\right)\Phi_{jk}(\bar t_v^{(m)})\\
 &=\sum_{k=1}^{\lfloor th_n^{-1}\rfloor}h_n\sum_{j\ge 1}w_{jk}^{l,m}\sum_{i=1}^{n_l}\left(X_{t_i^{(l)}}^{(l)}-X_{t_{i-1}^{(l)}}^{(l)}\right)\Phi_{jk}(\bar t_i^{(l)})\sum_{v=1}^{n_l}\left(X_{t_v^{(l)}}^{(m)}-X_{t_{v-1}^{(l)}}^{(m)}\right)\Phi_{jk}(\bar t_v^{(l)})\\
 &+\sum_{k=1}^{\lfloor th_n^{-1}\rfloor}h_n\sum_{j\ge 1}w_{jk}^{l,m}\sum_{v=1}^{n_l}\Big(X_{t_v^{(l)}}^{(l)}-X_{t_{v-1}^{(l)}}^{(l)}\Big)\Phi_{jk}(\bar t_v^{(l)})\\
& \hspace*{.2cm}\times \Big(\sum_{\d_i t^{(m)}\subset\d_v t^{(l)}}\Big(X_{t_i^{(m)}}^{(m)}-X_{t_{i-1}^{(m)}}^{(m)}\Big)(\Phi_{jk}(\bar t_i^{(m)})-\Phi_{jk}(\bar t_v^{(l)}))+\\
 &\hspace*{.45cm}\Big(X_{t_+^{(m)}(t_{v-1}^{(l)})}^{(m)}-X_{t_{v-1}^{(l)}}^{(m)}\Big)\big(\Phi_{jk}(\tilde u_v^{(m)})\hspace*{-.05cm}-\hspace*{-.05cm}\Phi_{jk}(\bar t_v^{(l)})\big)\hspace*{-.05cm}+\hspace*{-.05cm}\Big(X_{t_v^{(l)}}^{(m)}-X_{t_-^{(m)}(t_{v}^{(l)})}^{(m)}\Big)\big(\Phi_{jk}( u_v^{(m)})\hspace*{-.05cm}-\hspace*{-.05cm}\Phi_{jk}(\bar t_v^{(l)})\big)\hspace*{-.05cm}\Big).\end{align*}
Since the observation times are independent of $X$ according to Assumption \ref{noise2}, we can employ basic estimates \eqref{e1} and \eqref{e2} for the above increments of $X$. Applying the bound \eqref{Phiexp}, we find that the order of the last summand is $\sum_k h_n\sum_j w_{jk}^{l,m} j/(nh_n)$ and since for all weights the bound \eqref{orderweights} holds we conclude that the approximation error is uniformly of order $\mathcal{O}_{\P}(h_n)=\KLEINO_{\P}(n^{-1/4})$. 

\subsubsection{{\bf{Leading terms}}\label{subsec:l}}
This paragraph develops the asymptotics for the right-hand side of \eqref{dec1b} and the sum of the increments in \eqref{illustrbiv}.
Observe that 
\begin{align}\label{h}\sum_{i=1}^{n_l-1}\varphi_{jk}^2\big(t_i^{(l)}\big)\Big(\tfrac{t_{i+1}^{(l)}-t_{i-1}^{(l)}}{2}\Big)^2\asymp\sum_{i=1}^{n_l-1}\varphi_{jk}^2\big(t_i^{(l)}\big)\tfrac{t_{i+1}^{(l)}-t_{i-1}^{(l)}}{2}\tfrac{(F_l^{-1})'(t_i^{(l)})}{n_l} \asymp \Big(\int_0^1\varphi_{jk}^2(t)\,dt\Big)\tfrac{H^{kh_n}_l}{\eta_l^2}\,.\end{align}
The left approximation uses $(t_{i+1}^{(l)}-t_i^{(l)})/2=(H_l^{kh_n}+\mathcal{O}(h_n^{\alpha}))/(\eta_l^2n_l)$ as in \eqref{qvt} with $\alpha>1/2$ by Assumption \ref{noise2}. Writing the integral on the right-hand side as sum over the subintervals and using mean value theorem, the differences when passing to the arguments $(t_i^{(l)})_i$  induce approximation errors of order $jh_n^{-1}n^{-1}$. Thus, the total approximation errors are of order $(h_n^{\alpha}+j(nh_n)^{-1})j^2(nh_n^2)^{-1}$.\\
We focus on the oracle versions of \eqref{lmm} and \eqref{specv} with their deterministic optimal weights. The proof follows the same methodology as the proof of Proposition \ref{propmain} after restricting to a synchronous reference observation scheme. We concisely go through the details for cross terms and the proof for the bivariate spectral covolatility estimator.\\
We apply Theorem 3-1 of \cite{jacodkey} (or equivalently Theorem 2.6\,in \cite{podvet}) again. For the spectral estimator \eqref{specv}, consider
\begin{align}\zeta_k^n=n^{1/4}\,h_n\Big(\sum_{j\ge 1}w_{jk}^{p,q}\zeta_{jk}^{(pq)}-\Sigma_{(k-1) h_n}^{(pq)}\Big)\,,\end{align}
with the random variables
\begin{align}\zeta_{jk}^{(pq)}&=\left(\Big(\sum_{i=1}^{n_p}\Delta_i^n \bar C^{(p)}\Phi_{jk}\big(\bar t_i^{(p)}\big)-\sum_{i=1}^{n_p-1}\epsilon_{t_i^{(p)}}^{(p)}\varphi_{jk}\big(t_i^{(p)}\big)\frac{t_{i+1}^{(p)}-t_{i-1}^{(p)}}{2}\Big)\right.\\
&\left. \notag \quad \times \Big(\sum_{v=1}^{n_q}\Delta_v^n \bar C^{(q)}\Phi_{jk}\big(\bar t_v^{(q)}\big)-\sum_{v=1}^{n_q-1}\epsilon_{t_{v}^{(q)}}^{(q)}\varphi_{jk}\big(t_v^{(q)}\big)\frac{t_{v+1}^{(q)}-t_{v-1}^{(q)}}{2}\Big)\right)-\pi^2 j^2h_n^{-2}\delta_{p,q}\hat H^{kh_n}_p \,.
\end{align}
The accordance with \eqref{illustrbiv} follows from a generalization of the summation by parts identity \eqref{sbp}:
\begin{align*}S_{jk}^{(p)}&\asymp_p -\sum_{v=1}^{n_p-1} Y_v^{(p)} \Big(\Phi_{jk}\big(\bar t_{v+1}^{(p)}\big)-\Phi_{jk}\big(\bar t_{v}^{(p)}\big)\Big)\\
&\asymp_p -\sum_{v=1}^{n_p-1}Y_v^{(p)}\varphi_{jk}(t_v^{(p)})\frac{t_{v+1}^{(p)}-t_{v-1}^{(p)}}{2}\,.\end{align*}
In the first relation the remainder is only due to end-effects as $\bar t_1^{(p)}\ne 0$ and $\bar t_{n_p}^{(p)}\ne 1$ and asymptotically negligible. Also, the second remainder by application of mean value theorem and passing to arguments $t_v^{(p)}$ is asymptotically negligible. This remainder can be treated as the approximation error between discrete and continuous-time norm of the $(\varphi_{jk})$ in the following.\\
By Lemma \ref{key} we may without loss of generality work under synchronous observations $t_i,i=0,\ldots,n$, when considering the signal part $X$. 
Set $\bar t_i=(t_{i+1}-t_i)/2$. We shall write in the sequel terms of the signal part as coming from observations on a synchronous grid $(t_i)$, while keeping to the actual grids for the noise terms. For the expectation we have
\begin{align*}&\E\left[\zeta_{jk}^{(pq)}\right]=\sum_{i=1}^n\Phi_{jk}^2(\bar t_i)\E\big[\Delta_i^n\bar C^{(p)}\Delta_i^n\bar C^{(q)}\big]\\
&+\hspace*{-.05cm}\sum_{i,v=1}^{(n_p\vee n_q)-1}\E\left[\epsilon_{t_i^{(p)}}^{(p)}\epsilon_{t_i^{(q)}}^{(q)}\right]\hspace*{-.05cm}\varphi_{jk}\big(t_i^{(p)}\big)\Big(\tfrac{t_{i+1}^{(p)}-t_{i-1}^{(p)}}{2}\Big)\varphi_{jk}\big(t_v^{(q)}\big)\Big(\tfrac{t_{v+1}^{(q)}-t_{v-1}^{(q)}}{2}\Big)-\pi^2j^2h_n^{-2}\delta_{p,q}\E\left[\hat H^{kh_n}_p\right]\\
&=\sum_{i=1}^n\Phi_{jk}^2(\bar t_i)(t_{i+1}-t_i)\Sigma^{(pq)}_{(k-1)h_n}\hspace*{-.05cm}+\delta_{p,q}\left(\eta_p^2\sum_{i=1}^{n_p}\varphi_{jk}^2\big(t_i^{(p)}\big)\Big(\tfrac{t_{i+1}^{(p)}-t_{i-1}^{(p)}}{2}\Big)^2\hspace*{-.05cm}-\hspace*{-.05cm}\pi^2j^2h_n^{-2}H^{kh_n}_p\right)\\
&=\Sigma_{(k-1)h_n}^{(pq)}+R_{n,k}\end{align*}
by It\^{o} isometry. The remainders due to the approximation \eqref{h} satisfy with \eqref{orderweights} uniformly
\begin{align*}R_{n,k}\lesssim \sum_{j=1}^{\lfloor \sqrt{n}h_n\rfloor}j^2n^{-1}h_n^{-2}\big(h_n^{\alpha}+jn^{-1}h_n^{-1}\big)+\sum_{\lceil \sqrt{n}h_n\rceil}^{\lfloor nh_n\rfloor-1}\big(j^{-1}h_n+j^{-2}h_n^2nh_n^{\alpha}\big)=\KLEINO\big(n^{-1/4}\big)\,.\end{align*}
Since $\sum_{j\ge 1}w_{jk}^{p,q}=1$, asymptotic unbiasedness is ensured:
\begin{align*}\sum_{k=1}^{\lfloor th_n^{-1}\rfloor}\E\left[\zeta_k^n\big|\mathcal{G}_{(k-1)h_n}\right]=\sum_{k=1}^{\lfloor th_n^{-1}\rfloor}n^{1/4}h_n\Bigg(\sum_{j\ge 1}w_{jk}^{p,q}\E[\zeta_{jk}^{(pq)}]-\Sigma^{(pq)}_{(k-1)h_n}\Bigg)\stackrel{ucp}{\longrightarrow}0\,.\end{align*}
We now determine the asymptotic variance expression in \eqref{avarspecv}:
\begin{align*}\var\big(\zeta_{jk}^{(pq)}\big)&=\Big(\sum_{i=1}^{n}\Phi_{jk}^2(\bar t_i)(t_{i+1}-t_i)\Big)^2\big(\big(\Sigma_{(k-1)h_n}^{(pq)}\big)^2+\Sigma_{(k-1)h_n}^{(pp)}\Sigma_{(k-1)h_n}^{(qq)}\big)\\
&\quad +\eta_p^2\eta_q^2\Big(\sum_{i=1}^{n_p-1}\varphi_{jk}^2\big(t_i^{(p)}\big)\Big(\tfrac{t_{i+1}^{(p)}-t_{i-1}^{(p)}}{2}\Big)^2\Big)\Big(\sum_{i=1}^{n_q-1}\varphi_{jk}^2\big(t_i^{(q)}\big)\Big(\tfrac{t_{i+1}^{(q)}-t_{i-1}^{(q)}}{2}\Big)^2\Big)\\
& \quad +\Bigg(\sum_{i=1}^n\Phi_{jk}^2(\bar t_i)(t_{i+1}-t_i)\Big(\eta_p^2 \Sigma_{(k-1)h_n}^{(qq)}\sum_{i=1}^{n_p-1}\varphi_{jk}^2\big(t_i^{(p)}\big)\Big(\tfrac{t_{i+1}^{(p)}-t_{i-1}^{(p)}}{2}\Big)^2\\
&\quad +\eta_m^2\Sigma^{(pp)}_{(k-1)h_n}\sum_{i=1}^{n_q-1}\varphi_{jk}^2\big(t_i^{(q)}\big)\Big(\tfrac{t_{i+1}^{(q)}-t_{i-1}^{(q)}}{2}\Big)^2\Big)\Bigg)\\
&\hspace*{-0cm}\asymp\big(\Sigma_{(k-1)h_n}^{(pq)}\big)^2+\Sigma_{(k-1)h_n}^{(pp)}\Sigma_{(k-1)h_n}^{(qq)}+\pi^2j^2h_n^{-2}\big(H^{kh_n}_p\Sigma_{(k-1)h_n}^{(qq)}+H^{kh_n}_q\Sigma_{(k-1)h_n}^{(pp)}\big)\\
&\hspace*{5.8cm}+\pi^4j^4h_n^{-4}H^{kh_n}_pH^{kh_n}_q,\end{align*}
where the remainder is negligible by the same bounds as for the bias above.
The sum of conditional variances with $w_{jk}^{p,q}=I_k^{-1}I_{jk}$, $I_k=\sum_{j\ge 1}I_{jk}$, thus yields
\begin{align*}\sum_{k=1}^{\lfloor th_n^{-1}\rfloor}\E\left[\left(\zeta_k^n\right)^2\big|\mathcal{G}_{(k-1)h_n}\right]+\KLEINO(1)&=\sum_{k=1}^{\lfloor th_n^{-1}\rfloor}h_n^2n^{1/2}\sum_{j\ge 1}\big(w_{jk}	^{(pq)}\big)^2\var \big(\zeta_{jk}^{(pq)}\big)\\&=\sum_{k=1}^{\lfloor th_n^{-1}\rfloor}h_n^2n^{1/2}\sum_{j\ge 1} I_{jk}I_k^{-2}=\sum_{k=1}^{\lfloor th_n^{-1}\rfloor}h_n^2n^{1/2}I_k^{-1}.\end{align*}
As $h_n\sqrt{n}\rightarrow \infty$, we obtain an asymptotic expression as the solution of an integral
\begin{align*}\sum_{k=1}^{\lfloor th_n^{-1}\rfloor}\E\left[\left(\zeta_k^n\right)^2\big|\mathcal{G}_{(k-1)h_n}\right]&=\sum_{k=1}^{\lfloor th_n^{-1}\rfloor}h_n(\sqrt{n}h_n)I_k^{-1}\rightarrow\int_0^t\left(\int_0^{\infty}(f(\Sigma,\mathcal{H}(t),\nu_p,\nu_q;z))^{-1}dz\right)^{-1}ds
\end{align*}
with a continuous limit function $f$ which is the same as in \cite{bibingerreiss}.
Computing the solution of the integral using the explicit form of $I_k$ and $f$ yields the variance $\int_0^t \big(v_s^{(p,q)}\big)^2\,ds$ with
\begin{align*}\big(v_s^{(p,q)}\big)^2&=2\left({(F_p^{-1})}^{\prime}(s){(F_q^{-1})}^{\prime}(s)\nu_p\nu_q(A_s^2-B_s)B_s\right)^{\frac12}\\
&\quad\quad\times \big(\sqrt{A_s+\sqrt{A_s^2-B_s}}-\operatorname{sgn}(A_s^2-B_s)\sqrt{A_s-\sqrt{A_s^2-B_s}}\big)\,,\end{align*}
and the terms
\vspace*{-.5cm}
$$A_s=\Sigma^{(pp)}_s\frac{{(F_q^{-1})}^{\prime}(s)\nu_q}{{(F_p^{-1})}^{\prime}(s)\nu_p}+\Sigma^{(qq)}_s\frac{{(F_p^{-1})}^{\prime}(s)\nu_p}{{(F_q^{-1})}^{\prime}(s)\nu_q}\,,$$
$$B_s=4\left(\Sigma^{(pp)}_s\Sigma^{(qq)}_s+\big(\Sigma^{(pq)}_s\big)^2\right)\,.$$
The detailed computation is carried out in \cite{bibingerreiss} and we omit it here. Contrarily to the one-dimensional case, in the cross term there is no effect of non-Gaussian noise on the variance because fourth noise moments do not occur and because of component-wise independence.\\
The Lyapunov criterion follows from
\begin{align*}&\E\left[\big(\zeta_{jk}^{(pq)}\big)^4|\mathcal{G}_{(k-1)h_n}\right]\asymp 3\sum_{j\ge 1}\big(w_{jk}^{p,q}\big)^4I_{jk}^{-2}\asymp 3\,I_k^{-4}\sum_{j\ge 1}I_{jk}^2=\mathcal{O}(1)\\
&\Rightarrow~~\sum_{k=1}^{\lfloor th_n^{-1}\rfloor}\E\left[\Big(\zeta_k^n\Big)^4\big|\mathcal{G}_{(k-1)h_n}\right]=\mathcal{O}\Big(n\sum_{k=1}^{\lfloor th_n^{-1}\rfloor}h_n^4\Big)=\KLEINO\left(n^{-1/4}\right)\,.\end{align*} 
By Cauchy-Schwarz and Burkholder-Davis-Gundy inequalities, we deduce
\begin{align*}
&\E\left[h_n\sum_{j\ge 1}w_{jk}^{p,q}\sum_{i=1}^n\Delta_i^n\bar C^{(p)}\Delta_i^n\bar C^{(q)}\Phi_{jk}^2(\bar t_i)\sum_{i=1}^n\Delta_i^n W^{(p)}\right]\\
&\quad =h_n\sum_{j\ge 1}w_{jk}^{p,q}\sum_{i=1}^n\E\left[\Delta_i^n \bar C^{(p)}\Delta_i^n\bar C^{(q)}\Delta_i^n W^{(p)}\right]\Phi_{jk}^2(\bar t_i)\\
&\le h_n\sum_{j\ge 1}w_{jk}^{p,q}\sum_{i=1}^n(t_i-t_{i-1})^{\frac32}\Phi_{jk}^2(\bar t_i)=\KLEINO\big( n^{-1/4}\big)\,.
\end{align*}
By the analogous estimate with $\Delta_i^n W^{(q)}$ the stability condition \eqref{s4} is valid. Condition \eqref{s5} is shown using a decomposition as in \eqref{eq:martingale_decomp} above and can be verified in an analogous way. This proves stable convergence of the leading term to the limit given in Theorem \ref{thm:2}.\\
The heart of the proof of Theorem \ref{thm:3} is the asymptotic theory for the leading term \eqref{dec1b}, namely the analysis of the asymptotic variance-covariance structure. This is carried out in detail in \cite{BHMR} for the idealized locally parametric experiment using bin-wise orthogonal transformation to a diagonal covariance structure. The only difference between our main term and the setup considered in \cite{BHMR} is the Gaussianity of the noise component. Yet, in the deduction of the variance this only affects the terms with fourth noise moments where $\E[\epsilon_i^4]\ne 3\E[\epsilon_i^2]$ in general. Above, we explicitly proved that the resulting remainder converges to zero for the one-dimensional estimator and this directly extends to the diagonal elements here. An intuitive heuristic reason why this holds is that the smoothed statistics are asymptotically still close to a normal distribution, though the normality which could have been used in \cite{BHMR} does not hold here for fixed $n$ in general. Based on the expressions of variances for cross products and squared spectral statistics above, coinciding their counterparts in the normal noise model when separating the remainder induced for the squares, we can pursue the asymptotics along the same lines as the proof of Corollary 4.3 in \cite{BHMR}. At this stage, we restrict to shed light on the connection between the expressions in \eqref{clt3} and the asymptotic variance-covariance matrix. 
Observe that $(A\otimes B)^{\top}=A^{\top}\otimes B^{\top}$ for matrices $A,B$, ${\cal{Z}}{\cal{Z}}=2\cal{Z}$ and that $(A\otimes B)(C\otimes D)=(AC\otimes BD)$ for matrices $A,B,C,D$, such that
\begin{align*}&\Big(\Sigma_s^{\frac12}\otimes\big(\Sigma_s^{\cal{H}}\big)^{\frac14}\Big){\cal{Z}}\Big(\Big(\Sigma_s^{\frac12}\otimes\big(\Sigma_s^{\cal{H}}\big)^{\frac14}\Big)\cal{Z}\Big)^{\top}\\
&\quad =\Big(\Sigma_s^{\frac12}\otimes\big(\Sigma_s^{\cal{H}}\big)^{\frac14}\Big)2{\cal{Z}}\Big(\Sigma_s^{\frac12}\otimes\big(\Sigma_s^{\cal{H}}\big)^{\frac14}\Big)^{\top}\\
&\quad = 2\,\Big(\Sigma_s\otimes\big(\Sigma_s^{\cal{H}}\big)^{\frac12}\Big)\cal{Z}\,,\end{align*}
since $\cal{Z}$ commutes with $\Big(\Sigma_s^{\frac12}\otimes\big(\Sigma_s^{\cal{H}}\big)^{\frac14}\Big)$. Therefore, the expression in \eqref{clt3} is natural for the matrix square root of the asymptotic variance-covariance, where we use two independent terms because of non-commutativity of matrix multiplication.
Conditions \eqref{s1} and \eqref{s3} and the stability conditions \eqref{s4} and \eqref{s5} can be analogously shown by element-wise adopting the results for squared and cross products of spectral statistics from above. Since any component of the estimator is a weighted sum of the entries of $S_{jk}S_{jk}^{\top}$, bias-corrected on the diagonal, the convergences to zero in probability follow likewise.

\subsubsection{\bf{Remainder terms}\label{subsec:r}}
After applying triangular inequality to \eqref{dec2b}, it suffices to prove that
\begin{align}\label{dec2c}n^{1/4}\|\LMM_{n,t}^{or}(Y)-\LMM_{n,t}^{or}(\bar C+\epsilon)\|\stackrel{ucp}{\longrightarrow}0\,,\end{align}
\begin{align}\label{dec2d}n^{1/4}\left\|\int_0^t\operatorname{vec}\big(\Sigma_s-\Sigma_{\lfloor sh_n^{-1}\rfloor h_n}\big)\,ds\right\|\stackrel{ucp}{\longrightarrow}0\,.\end{align}
For $A,B\in\R^d$, we use in the following several times the elementary bound:
\begin{align}\label{bound}\left\|AA^{\top}-BB^{\top}\right\|=\left\| B(A^{\top}-B^{\top})+(A-B)A^{\top}\right\|\le \big(\|A\|+\|B\|\big)\|A-B\|\,.\end{align}
Define analogously as above $\tilde S_{jk}=\big(\sum_{i=1}^{n_p}\Delta_i^n {\bar C}^{(p)}\Phi_{jk}\big(\bar t_i^{(p)}\big)\big)_{1\le p\le d}$, the spectral statistics in the locally constant volatility experiment. Then we can bound uniformly for all $t$:
\begin{align*}&\|\LMM_{n,t}^{or}(Y)-\LMM_{n,t}^{or}(\bar C+\epsilon)\|\\
&\le \sum_{k=1}^{h_n^{-1}}h_n\,\left\|\sum_{j=1}^{\lfloor nh_n\rfloor -1} W_{jk}\operatorname{vec}\big(S_{jk}S_{jk}^{\top}-\tilde S_{jk}\tilde S_{jk}^{\top}\big)\right\|\\
&\le \sum_{k=1}^{h_n^{-1}}h_n\sum_{j=1}^{\lfloor nh_n\rfloor -1} \|W_{jk}\|\big(\|S_{jk}\|+\|\tilde S_{jk}\|\big)\|S_{jk}-\tilde S_{jk}\|\\
&\lesssim \sum_{k=1}^{h_n^{-1}}h_n\sum_{j=1}^{\lfloor nh_n\rfloor -1} \Big(1+\frac{j^2}{nh_n^2}\Big)^{-2}\|S_{jk}-\tilde S_{jk}\|=\mathcal{O}_{\P}\big(h_n\big)=\KLEINO_{\P}(n^{-1/4}\big)\,,\end{align*}
what yields \eqref{dec2c}. We have used Lemma C.1 from \cite{BHMR} for the magnitude of $\|W_{jk}\|$, the bound \eqref{bound} and a bound for the sum over $j$, for which holds 
\[\frac{1}{\sqrt{n}h_n}\sum_{j=1}^{\lfloor nh_n \rfloor-1}\Big(1+\frac{j^2}{nh_n^2}\Big)^{-2}\rightarrow \frac{\pi}{2}\]
by an analogous integral approximation as used in the limiting variance before. Drift terms and cross terms including the drift are asymptotically negligible and are handled similarly as before. Directly neglecting drift terms, we deduce $\|S_{jk}-\tilde S_{jk}\|=\mathcal{O}_{\P}(h_n)$ uniformly from $(S_{jk}-\tilde S_{jk})^{(p)}\asymp_p \sum_{i=1}^{n_p}\Delta_i^n\tilde C^{(p)}\Phi_{jk}(\bar t_i)$ with \eqref{e2}. \eqref{dec2d} is equivalent to
\begin{align}\label{dec2e}n^{1/4}\Bigg\|\sum_{k=1}^{h_n^{-1}}\int_{(k-1)h_n}^{kh_n}\big(\Sigma_s-\Sigma_{(k-1)h_n}\big)\,ds\Bigg\|\stackrel{ucp}{\longrightarrow} 0\,.\end{align}
Using the decomposition
\begin{align*}\Sigma_s-\Sigma_{(k-1)h_n}&=\sigma_s\sigma_s^{\top}-\sigma_{(k-1)h_n}\sigma_{(k-1)h_n}^{\top}\\
&=(\sigma_s-\sigma_{(k-1)h_n})\sigma_{(k-1)h_n}^{\top}+\sigma_{(k-1)h_n}(\sigma_s^{\top}-\sigma_{(k-1)h_n}^{\top})\\
&\quad +(\sigma_s-\sigma_{(k-1)h_n})(\sigma_s^{\top}-\sigma_{(k-1)h_n}^{\top})\end{align*}
for $s\in[(k-1)h_n,kh_n]$, it is easy to find that it suffices to bound terms $\|\sigma_s-\sigma_{(k-1)h_n}\|$. Then, Assumption \ref{Hd} guarantees \eqref{dec2e} and \eqref{dec2d} in the same way as for the one-dimensional model.\\
For the spectral covolatility estimator \eqref{specv} we may conduct an analysis of the remainder similarly as in the proof of Proposition \ref{propremainder}. One can as well employ 
integration by parts of It\^{o} integrals after supposing again a synchronous observation design $t_i,i=0,\ldots,n$, possible according to Lemma \ref{key}: 
\begin{align}&\notag\Delta_i^n \tilde C^{(p)}\Delta_i^n \tilde C^{(q)}-\int_{t_{i-1}}^{t_i}\big(\Sigma^{(pq)}_s-\Sigma_{\lfloor sh_n^{-1}\rfloor h_n}^{(pq)}\big)ds\\
&\quad \label{ibp}=\int_{t_{i-1}}^{t_i}\big(\tilde C_s^{(p)}-\tilde C_{t_{i-1}}^{(p)}\big)d\tilde C_s^{(q)}+\int_{t_{i-1}}^{t_i}\big(\tilde C_s^{(q)}-\tilde C_{t_{i-1}}^{(q)}\big)d\tilde C_s^{(p)}\,.\end{align}
with $\tilde C$ approximation errors as in \eqref{appr3}. Consider the random variables
\begin{align*}&\tilde \zeta_{jk}^{(pq)}=\sum_{i=1}^n\Delta_i\tilde C^{(p)}\Phi_{jk}(\bar t_i)\sum_{v=1}^n\Delta_v\tilde C^{(q)}\Phi_{jk}(\bar t_v)\,,\\
&\tilde \zeta_k^n=h_n\sum_{j\ge 1}w_{jk}^{p,q}\tilde\zeta_{jk}^{(pq)}-\int_{kh_n}^{(k+1)h_n}\big(\Sigma^{(pq)}_s-\Sigma_{\lfloor sh_n^{-1}\rfloor h_n}^{(pq)}\big)ds\,.\end{align*}
Inserting \eqref{ibp} for $\Delta_i^n\tilde C^{(p)} \Delta_i^n\tilde C^{(q)}$, using $[\int Z\,dX,\int Z\,dX]=\int Z^2\,d[X,X]$ for It\^{o} integrals and applying Burkholder-Davis-Gundy inequalities and using the bound \eqref{e2} for $\E\big[\big(\Delta_i^n\tilde C^{(p)}\big)^2\big]$,\\ $\E\big[\big(\Delta_i^n\tilde C^{(q)}\big)^2\big]$, it follows that $\E\big[\big(\tilde \zeta_k^n\big)^2\big]=\mathcal{O}(n^{-1})$. Bounds for cross terms with $\tilde C$ and $\bar C$ readily follow by standard estimates and we conclude our claim.
\subsection{\bf{Proofs for adaptive estimation}}
We carry out the proof of Proposition \ref{tight1} in the case $d=1$
explicitly. We need to show that 
\begin{align}
n^{\frac{1}{4}}\left|\IV_{n,t}-\IV_{n,t}^{or}(Y)\right|\stackrel{ucp}{\longrightarrow}0~~\mbox{as}~n\rightarrow\infty\,.\label{lastaim}
\end{align}
Let us first act as if the noise level $\eta$ was known and concentrate
on the harder problem of analyzing the plug-in estimation of the instantaneous
squared volatility process $\sigma_{t}^{2}$ in the weights. We have
to bound 
\[
\IV_{n,t}-\IV_{n,t}^{or}(Y)=\sum_{k=1}^{\lfloor th_{n}^{-1}\rfloor}h_{n}\sum_{j=1}^{\lfloor nh_{n}\rfloor-1}\big(\hat{w}_{jk}-w_{jk}\big)\Big(S_{jk}^{2}-[\varphi_{jk},\varphi_{jk}]_{n}\frac{\eta^{2}}{n}\Big)\,,
\]
uniformly with $w_{jk}$ being the optimal oracle weights \eqref{weights}
and $\hat{w}_{jk}$ their adaptive estimates. 
We introduce a coarse grid of blocks of lengths $r_{n}$ such that $r_{n}h_{n}^{-1}\rightarrow\infty$
as $n\rightarrow\infty$, $r_nh_n^{-1}\in\N$. We analyze the above difference in this
double asymptotic framework, where the plug-in estimators are evaluated
on the coarse grid first. Denoting the adaptive and oracle estimators
with weights evaluated on the coarse grid by $\IV_{n,t}^{c}$ and
$\IV_{n,t}^{or,c}(Y)$, respectively, we have 
\begin{equation}
\IV_{n,t}^{c}-\IV_{n,t}^{or,c}(Y)=\sum_{m=1}^{\lfloor tr_{n}^{-1}\rfloor}h_{n}\sum_{k=(m-1)r_{n}h_{n}^{-1}+1}^{mr_{n}h_{n}^{-1}}\sum_{j=1}^{\lfloor nh_{n}\rfloor-1}\big(w_{j}(\hat{\sigma}_{(m-1)r_{n}}^{2})-w_{j}(\sigma_{(m-1)r_{n}}^{2})\big)\,Z_{jk}\label{eq:-5}
\end{equation}
with $Z_{jk}=S_{jk}^{2}-[\varphi_{jk},\varphi_{jk}]_{n}\eta^{2}/n-\sigma_{(k-1)h_{n}}^{2}$
where the weights are functions (independent of the bin $k$, as
$[\varphi_{jk},\varphi_{jk}]_{n}$ does not depend on $k$) 
\[
w_{j}(x)=\frac{\big(x+\frac{\eta^{2}}{n}[\varphi_{jk},\varphi_{jk}]_{n}\big)^{-2}}{\sum_{l=1}^{\lfloor nh_{n}\rfloor-1}\big(x+\frac{\eta^{2}}{n}[\varphi_{lk},\varphi_{lk}]_{n}\big)^{-2}}\,,
\]
which are well-defined for $x\in\R_+$ and satisfy $\sum w_{j}(x)=1$. As $\sigma$ is uniformly bounded
from below and from above by Assumption \ref{H1} and Section \ref{sec:6.1},
there exists a constant $C_{1}>0$ such that uniformly 
\begin{equation}
w_{j}\left(\sigma_{t}^{2}\right)\lesssim w_{j}\left(C_{1}\right)\label{eq:uniform_weight_bound}
\end{equation}
for all $j\geq1$. Based on the standard estimates \eqref{JM2}, \eqref{JM} as well as \eqref{e2}, we may consider $\tilde{Z}_{jk}=\tilde{S}_{jk}^{2}-[\varphi_{jk},\varphi_{jk}]_{n}\eta^{2}/n-\sigma_{(k-1)h_{n}}^{2}$ where the $\tilde{S}_{jk}$ are the statistics under locally parametric
volatility and without drift, and the remainder is asymptotically negligible. Moreover, by subtracting $\sigma_{(k-1)h_{n}}^{2}$
in the definition of $\tilde{Z}_{jk}$ equation \eqref{drift} shows
that the $\tilde{Z}_{jk}$ are uncorrelated for different $k$. Hence,
$\var\big(\sum_{k}\tilde{Z}_{jk}\big)=\sum_{k}\var(\tilde{Z}_{jk})$ and thus
\begin{equation}
\var\big(\sum_{k}Z_{jk}\big)=\sum_{k}\var(Z_{jk})+\KLEINO(1)\,.\label{eq:var_Z}
\end{equation}
We prove \eqref{lastaim} in two steps. We show first that \eqref{eq:-5}
is $\KLEINO_{\P}(n^{-1/4})$ uniformly and then that the difference
between estimating on the coarse and finer grid is $\KLEINO_{\P}(n^{-1/4})$,
as well.\\
The crucial property to ensure tightness of the adaptive approach
is a uniform bound on the first derivatives of the weight functions:
$w_{j}(x)$ is continuously differentiable with derivatives satisfying:
\begin{align}
\big|w_{j}'\left(x\right)\big|\lesssim w_{j}\left(x\right)\log^{2}(n)\,.\label{derivativebound}
\end{align}
To see why this holds set $c_{j}=\frac{\eta^{2}}{n}\left[\varphi_{jk},\varphi_{jk}\right]_{n}$
and observe that 
\begin{eqnarray*}
\left|w_{j}'\left(x\right)\right| & = & \left|\frac{-2\left(x+c_{j}\right)^{-3}\sum_{m=1}^{\lfloor nh\rfloor-1}\left(x+c_{m}\right)^{-2}-\left(x+c_{j}\right)^{-2}\sum_{m=1}^{\lfloor nh\rfloor-1}\left(\left(-2\right)\left(x+c_{m}\right)^{-3}\right)}{\left(\sum_{m=1}^{\lfloor nh\rfloor-1}\left(x+c_{m}\right)^{-2}\right)^{2}}\right|\\
 & \leq & 2w_{j}\left(x\right)\frac{\sum_{m=1}^{\lfloor nh\rfloor-1}\left(x+c_{m}\right)^{-2}\left|\left(x+c_{j}\right)^{-1}-\left(x+c_{m}\right)^{-1}\right|}{\sum_{m=1}^{\lfloor nh\rfloor-1}\left(x+c_{m}\right)^{-2}}\lesssim w_{j}\left(x\right)\log^{2}(n)
\end{eqnarray*}
for $n$ sufficiently large. The last inequality follows from 
\begin{eqnarray*}
\left|\left(x+c_{j}\right)^{-1}-\left(x+c_{m}\right)^{-1}\right|\leq\frac{1}{c_{j}}+\frac{1}{c_{m}}\lesssim\frac{1}{c_{1}}=\c O\left(\log^{2}(n)\right).
\end{eqnarray*}
The plug-in estimator \eqref{pilot1} satisfies $\|\hat{\sigma}^{2}-\sigma^{2}\|_{L^{1}}=\mathcal{O}_{\P}\big(\delta_{n}\big)$
for the $L^{1}$-norm $\|\,\cdot\,\|_{L^{1}}$ with a sequence $\delta_{n}\rightarrow0$
as $n\rightarrow\infty$, $\delta_{n}\lesssim n^{-1/8}$ for optimal
window length on Assumption \ref{H1}. Hence, by \eqref{eq:uniform_weight_bound}
$w_{j}(\hat{\sigma}_{(m-1)r_{n}}^{2}-w_{j}(\sigma_{(m-1)r_{n}}^{2})=\mathcal{O}_{\P}(w_{j}(C_{1})\delta_{n}\log^{2}(n))$.
This, \eqref{eq:var_Z}, Cauchy-Schwarz and \eqref{orderweights}
show that 
\begin{align}
 \notag \E\hspace*{-.05cm}\left[\big|\IV_{n,t}^{c}\hspace*{-.05cm}-\hspace*{-.05cm}\IV_{n,t}^{or,c}(Y)\big|\right]&\hspace*{-.05cm} \lesssim\E\hspace*{-.05cm}\left[\sum_{m=1}^{\lfloor tr_{n}^{-1}\rfloor}\hspace*{-.1cm}h_{n}\hspace*{-.1cm}\sum_{j=1}^{\lfloor nh_{n}\rfloor-1}\hspace*{-.05cm}\left|w_{j}(\hat{\sigma}_{(m-1)r_{n}}^{2})\hspace*{-.05cm}-\hspace*{-.05cm}w_{j}(\sigma_{(m-1)r_{n}}^{2})\right|\hspace*{-.05cm}\left|\sum_{k=(m-1)r_{n}h_{n}^{-1}+1}^{mr_{n}h_{n}^{-1}}\hspace*{-.25cm}Z_{jk}\right|\right]\nonumber \\
 &\notag \lesssim\delta_{n}\left(\log^{2}(n)\right)\sum_{m=1}^{\lfloor tr_{n}^{-1}\rfloor}h_{n}\sum_{j=1}^{\lfloor nh_{n}\rfloor-1}w_{j}\left(C_{1}\right)\Big(\var\Big(\sum_{k=(m-1)r_{n}h_{n}^{-1}+1}^{mr_{n}h_{n}^{-1}}Z_{jk}\Big)\Big)^{1/2}\nonumber \\
 & \lesssim\sqrt{\frac{h_{n}}{r_{n}}}\delta_{n}\log^{2}{n}+\KLEINO\left(1\right).\label{neueordnung}
\end{align}
The required order $\KLEINO_{\P}(n^{-1/4})$ for \eqref{eq:-5} is
thus achieved if $r_{n}\rightarrow0$ not too fast, i.e.\,$r_{n}^{-1}\lesssim n^{1/4}(\log{n})^{-5}$.\\
Consider the remainder by the difference of coarse and fine grid.
Since for $\IV_{n,t}^{or,c}(Y)$ and $\IV_{n,t}^{or}(Y)$ the statistics
for each block $k$ are uncorrelated, it is enough to bound the variance
of the difference by 
\[
\sum_{m=1}^{r_{n}^{-1}}\sum_{k=(m-1)r_{n}h_{n}^{-1}+1}^{mr_{n}h_{n}^{-1}}\hspace*{-0.1cm}h_{n}^{2}\left(\sum_{j=1}^{\lfloor nh_{n}\rfloor -1}\left(\E\left[\big(w_{j}(\sigma_{(k-1)h_{n}}^{2})-w_{j}(\sigma_{(m-1)r_{n}}^{2})\big)^{2}\tilde{Z}_{jk}^{2}\right]\right)^{\frac{1}{2}}\right)^{2}\hspace*{-0.1cm}\lesssim h_{n}r_{n}\log^{4}{(n)}\,,
\]
using \eqref{derivativebound} and \eqref{neueordnung}. 
This shows that $|\IV_{n,t}^{or,c}(Y)-\IV_{n,t}^{or}(Y)|=\KLEINO_{\P}(n^{-1/4})$
uniformly. Exploiting the same ingredients as above we obtain likewise
that $|\IV_{n,t}^{c}-\IV_{n,t}|=\KLEINO_{\P}(n^{-1/4})$ uniformly.
In order to analyze the estimation error induced by pre-estimation
of $\eta^{2}$, we can consider the weights as functions of $\eta^{2}$
and compute their derivatives. As $\eta^{2}$ does not depend on time
and we have $|\hat{\eta}^{2}-\eta^{2}|=\mathcal{O}_{\P}(n^{-1/2})$,
a simpler computation yields that the pre-estimation of $\eta^{2}$
is of smaller order as the error by plug-in estimation of local volatilities.
Thus, using triangle inequality we conclude \eqref{lastaim}.\\
The proofs that Theorem \ref{thm:2} and Theorem \ref{thm:3} extend from
the oracle to the adaptive versions of the estimators \eqref{specv}
and \eqref{lmm} can be conducted in an analogous way. For covariation
matrix estimation, the key ingredient is the uniform bound on the
norm of the matrix derivative of the weight matrix function $W_{j}(\Sigma)$
w.r.t.\,$\Sigma$, which is a matrix with $d^{6}$ entries and requires
a notion of matrix derivatives, see Lemma C.2 in \citet{BHMR}. The
proof is then almost along the same lines as the proof of Theorem
4.4 in \citet{BHMR}, with the only difference in the construction
being that the $Z_{jk}$ are not independent, but still have negligible
correlations. The adaptivity in the proof of Theorem 4.4 of \citet{BHMR}
is proved under more delicate asymptotics of asymptotically separating
sample sizes. For this reason, but at the same time not having the
remainders, the restrictions on $r_{n}$ are different there.
\bibliographystyle{elsart-harv}
\bibliography{literatur}
\end{document}